\documentclass[11pt]{article}

\usepackage[a4paper,left=2.5cm, right=2.5cm, top=2.5cm, bottom=2.5cm]{geometry}

\usepackage{amsmath}    
\usepackage{amssymb}
\usepackage{graphicx}   
\usepackage{verbatim}   
\usepackage{color}      
\usepackage{hyperref}   
\usepackage{enumerate}
\usepackage{listings}

\usepackage{bbold}
\usepackage{pgfplots}

\usepackage{multirow}

\usepackage{tikz}
\usepackage{adjustbox}
\usepackage{amsmath}    
\usepackage{amssymb}
\usepackage{amsthm}

\usepackage{pgfplots}
\usepackage{cancel}

\newcommand{\del}{\partial}
\renewcommand{\theta}{\vartheta}
\renewcommand{\phi}{\varphi}

\newcommand{\vecc}[2]{\left ( \begin{array}{c}#1\\#2\\ \end{array}\right )}
\newcommand{\veccc}[3]{\begin{pmatrix}#1\\#2\\#3 \end{pmatrix}}
\newcommand{\dd}{\mathrm{d}}

\newcommand{\ii}{\mathbb{i}}

\renewcommand{\and}{\wedge}

\renewcommand{\vec}{\mathbf}

\newcommand{\id}{\mathbb{1}}

\newcommand{\dt}{\Delta t}
\newcommand{\dx}{{\Delta x}}
\newcommand{\dy}{{\Delta y}}

\DeclareMathOperator*{\argmin}{arg\,min}

\newtheorem{lemma}{Lemma}[section]
\newtheorem{proposition}{Proposition}[section]
\newtheorem{corollary}{Corollary}[section]
\newtheorem{example}{Example}[section]

\newtheorem{definition}{Definition}[section]

\begin{document}

\title{Structure preserving nodal continuous Finite Elements   via Global Flux quadrature}

\author{Wasilij Barsukow$^1$, Mario Ricchiuto$^{2}$, Davide Torlo$^3$}

\date{{\footnotesize $^1$ Institut de Mathématiques de Bordeaux (IMB), CNRS UMR 5251, 351 Cours de la Libération, 33405 Talence, France, \url{wasilij.barsukow@math.u-bordeaux.fr}}\\
{\footnotesize $^2$ INRIA, Univ. Bordeaux, CNRS, Bordeaux INP, IMB, UMR 5251, 200 Avenue de la Vieille Tour, 33405 Talence cedex, France, \url{mario.ricchiuto@inria.fr}}\\
{\footnotesize $^3$ Dipartimento di Matematica G. Castelnuovo, Universit\`a di Roma La Sapienza, piazzale Aldo Moro, 5, 00185, Rome, Italy, \url{davide.torlo@uniroma1.it}}}

\maketitle
\begin{abstract}
	Numerical methods for hyperbolic PDEs require stabilization. For linear acoustics, divergence-free vector fields should remain stationary, but classical Finite Difference methods add incompatible diffusion that dramatically restricts the set of discrete stationary states of the numerical method. Compatible diffusion should vanish on stationary states, e.g. should be a gradient of the divergence. 
	Some Finite Element methods allow to naturally embed this grad-div structure, e.g. the SUPG method or OSS. 
	We prove here that the particular discretization associated to them still fails to be constraint preserving. 
	We then introduce a new framework on Cartesian grids based on surface (volume in 3D) integrated operators inspired by Global Flux quadrature and related to mimetic approaches. We are able to construct constraint-compatible stabilization operators (e.g. of SUPG-type) and show that the resulting methods are vorticity-preserving. We show that the Global Flux approach is even super-convergent on stationary states, we characterize the kernels of the discrete operators and we provide projections onto them. 
\end{abstract}

\section{Introduction}
\label{sec:intro}

\subsection{Acoustic equations}

This paper focuses on the discretization of hyperbolic PDEs. Although
we have in mind applications to  hyperbolic conservation laws such
as the Euler or shallow water equations with source terms, we will work here in the much simpler setting of 
the linear wave equations in first-order form in the 2D and general forms:
\begin{align}\label{eq:acoustic}
	\begin{cases}
 \del_t u + \del_x p = 0, &  \\
 \del_t v + \del_y p = 0, \\
 \del_t p + \del_x u + \del_y v = 0, 
	\end{cases}\qquad 	\begin{cases}
	\del_t \vec v + \nabla p = 0, \\
	\del_t p + \nabla \cdot \vec v = 0,
	\end{cases}
\end{align}
for $u,v,p \colon \Omega \subset \mathbb R^2 \to \mathbb R$, and
with, in 2-d, the notation $\vec v = (u,v)$. We also introduce the notation 
\begin{align}
 \del_t q + J^x \del_x q + J^y \del_y q &= 0 
\end{align}
with 
\begin{align}
 q &= \veccc{u}{v}{p}, & J^x &= \left( \begin{array}{ccc} 0 & 0 & 1 \\ 0 & 0 &  0 \\ 1 & 0 & 0 \end{array} \right ), & J^y &= \left( \begin{array}{ccc} 0 & 0 & 0 \\ 0 & 0 &  1 \\ 0 & 1 & 0 \end{array} \right ) .
\end{align}
The system of linear acoustics possesses an involution:
\begin{align}
 \del_t (\nabla \times \vec v) &= 0,
\end{align}
which is reminiscent of involutions appearing e.g. in the (vacuum) Maxwell equation. 
The stationary states of linear acoustics are divergence-free, i.e.
\begin{equation}
	\partial_t q \equiv 0 \Longleftrightarrow \nabla \cdot \vec v \equiv 0 \text{ and }p\equiv p_0 \in \mathbb R.
\end{equation}
Both acoustics and Maxwell equations can be seen as toy-models for the more complex Euler equations and those of magnetohydrodynamics.

\subsection{Structure-preserving Finite Difference methods}\label{sec:FDstructure_preserving}

Numerical methods for hyperbolic problems require appropriate stabilization. It is introduced to obtain $L^2$ stability (or entropy stability in the non-linear case),
or to manage discontinuous solutions. One way to stabilize is to add numerical dissipation.
Many numerical methods for multi-dimensional hyperbolic problems contain stabilization initially derived in a one-dimensional setup.  
For instance, it is customary to compute the normal flux across an edge (or a face) by 
ignoring any signals propagating in the transverse direction, or emanating from the corners of the cell.
This one-dimensional stabilization applied in different directions has a practical impact onto numerical solutions, e.g. on stationary states characterized by a balance of contributions from different directions (see e.g. \cite{barsukow17a}). The restriction of this datum onto one direction is not stationary, and unsurprisingly the combined numerical diffusion from the one-dimensional problems does not generally cancel out. One observes the datum being diffused away instead of being kept stationary.

Not every kind of multi-dimensional information, however, leads to a method with special properties and improved behavior. For example, in \cite{barsukow17} the exact solution of the 4-quadrant Riemann problem has been used to derive a truly multi-dimensional Godunov method. Although it takes into account all the signals from corners, and generally makes no approximation in the evolution step, it fails to be stationarity preserving, or involution preserving.

In the context of linear acoustics on Cartesian grids, a numerical method based on one-dimensional Riemann solvers amounts to the following stabilization 
\begin{align}
	\begin{split}
 \del_t u + \del_x p &= \frac12 \Delta x \del_x^2 u + \text{h.o.t.},\\
 \del_t v + \del_y p &= \frac12 \Delta y \del_y^2 v + \text{h.o.t.} ,\\
 \del_t p + \del_x u + \del_y v &= \frac12 (\Delta x \del_x^2 p + \Delta y \del_y^2 p) + \text{h.o.t.}.
	\end{split}
\end{align}

One can show that $\del_x u + \del_y v = 0$ no longer remains stationary, unless $\del_x u = 0$, $\del_y v = 0$ individually. 
This has a dramatic impact on simulations, as setups that should remain stationary are now diffused, which can be understood as loss of consistency for long-time simulations. 
Numerical methods whose stationary states are described by a discretization of $\nabla \cdot \vec v = 0$ (without further constraints) are called \emph{stationarity preserving} \cite{barsukow17a} and therefore possess a rich set of stationary states. One can show that the low Mach number limit of the Euler equations is related to the long-time limit of linear acoustics \cite{jung2022steady,jung2024behavior} and that stationarity-preserving methods are also involution-preserving. Preservation of discrete involutions in the context of Maxwell equations usually is relevant for long-time stability and the correct coupling to matter.

An obvious approach to deriving a stationarity-preserving method is to ensure that the numerical diffusion of $\vec v$ is a function of the divergence. All the methods suggested in \cite{morton01,sidilkover02,jeltsch06,mishra09preprint,lung14,barsukow17a} essentially imply
\begin{align}
	\begin{split}
 \del_t u + \del_x p &= \frac12 \Delta x \del_x(\del_x u + \del_{y} v) + \text{h.o.t.}, \\
  \del_t v + \del_y p &= \frac12 \Delta y \del_y(\del_{x} u + \del_y v) + \text{h.o.t.}, \\
 \del_t p + \del_x u + \del_y v &= \frac12 (\Delta x \del_x^2 p + \Delta y \del_y^2 p) + \text{h.o.t.}, \label{eq:statpresdiffusion}
	\end{split}
\end{align}
i.e., the choice of $\nabla (\nabla \cdot \vec v)$ as the appropriate diffusion for the evolution of $\vec v$. Observe that $\nabla p = 0$ and $\del_x u + \del_y v = 0$ now are again characterizing the stationary states, i.e. the stationary states of the PDE are exempt from the effect of numerical diffusion.

An interesting dichotomy thus appears to govern the field of truly multi-dimensional methods. While numerical methods that perform well are often derived ad-hoc (e.g. \cite{barsukow20cgk}), those with a first-principles derivation do not generally show improved behavior. 
Another example is \cite{barsukow23nodal}, where a modified idea of how global conservation is related to local conservation yields a stationarity preserving method, but at one point an explicit choice is made without providing a fundamental reason for it. 
This points to a general lack of understanding how numerical methods for multi-dimensional problems should be derived and explains the interest in structure preserving numerical methods. 
Their improved performance in practice is another reason, of course. 
To provide a more general strategy to achieve stationarity preservation for stabilized continuous Finite Element methods is the aim of this paper.

\subsection{Structure-preserving Finite Element methods and failure of SUPG}

Finite element methods (FEM) are successful in achieving high order of accuracy and can be used on different types of computational grids. There exists a vast literature on structure preserving FEM, for instance among mixed FEM (e.g. \cite{campos16,wimmer20,arnold06}). Hyperbolic systems of conservation laws, however, require stabilization. This paper focuses on continuous Galerkin methods with an artificial diffusion and with the same choice of discretization space for scalars as for vector components. We consider to this end the Streamline-upwind Petrov-Galerkin (SUPG) \cite{brooks82,hughes86} and the Orthogonal Subscale Stabilization (OSS)
\cite{CODINA1997373,CODINA20001579,OSSCodinaBadia,michel2022spectral} approaches. Both methods allow to introduce the proper numerical diffusion structure. To see this for  SUPG
it is enough to recall that the method is obtained as  the weak form of the  modified equation
\begin{align}
 \del_t q + J^x \del_x q + J^y \del_y q =  \partial_x(\alpha h \,  J_x ( \del_t q + J^x \del_x q + J^y \del_y q )) + \partial_y(\alpha h \, J_y ( \del_t q + J^x \del_x q + J^y \del_y q ))
\end{align}
where $h$ is a characteristic mesh size, and $\alpha$ a stabilization parameter/matrix, which we consider constant in the following. For linear acoustics, when expanding the right-hand side we obtain the stabilized equations
\begin{align}
	\begin{cases}\label{eq:acoustic-supg}
 \del_t u + \del_x p =  \del_x ( \alpha h (  \del_t p + \del_x u + \del_y v)) ,\\
 \del_t v + \del_y p =  \del_y ( \alpha h (  \del_t p + \del_x u + \del_y v)) , \\
 \del_t p + \del_x u + \del_y v =\del_x ( \alpha h ( \del_t u + \del_x p))+\del_y ( \alpha h ( \del_t v + \del_y p))  .
	\end{cases}
\end{align}
Beside the time derivatives, the stabilization terms are essentially variational approximations of the grad-div operator for the velocity stabilization,
and of a standard Laplacian for the pressure equation. 
This approach thus seems to produce, besides some mixed space-time derivatives, exactly the terms in \eqref{eq:statpresdiffusion}. One might thus expect that this numerical method will be stationarity preserving and therefore vorticity preserving \cite{barsukow17a}. However, 
as we will show in Section~\ref{sec:standardstab} as well as  in the numerical tests, this is not the case.
The reason for this is  related to the fact that  the implication 
\begin{align}
 \del_x u + \del_y v &\equiv 0  & &\Rightarrow & \del_x^2 u + \del_x \del_y v &\equiv 0
\end{align}
is not true in the discrete:
\begin{align}
 \int \phi (\del_x u + \del_y v)  \dd x &\equiv 0  \text{ }\forall \varphi\in V_h^K & &\not\Rightarrow & \int \del_x \phi (\del_x u +\del_y v) \dd x = 0 \; \forall \varphi\in V_h^K.
\end{align}
where, on a given Cartesian tessellation  of the spatial domain, we denote by
$V_h^K$ is the $K$-th degree finite element approximation space.

\subsection{Restoring stationarity preservation  via  Global Flux quadrature}\label{sec:keyidea}

Flux globalization dates back to the work of    \cite{GASCON2001261,donat11}   in the context of approximations of balance laws 
$$
\partial_t q + \partial_x F = S(q,x)
$$
For this problem a relevant aspect is  the  super-convergent    (or even exact) approximation  of non-trivial steady states. 
To this end,    one can write the source term as a flux $R$ which is the primitive of the source term:
$$
R = R_0  -\int_{x_0}^xS(q(s,t),s)) ds.
$$
In this setting, discrete  steady equilibria  verify the relation
$$
\partial_x(F + R) =0 \Rightarrow F+R = G_0 \in \mathbb R
$$
with $G=F+R$ the so-called  global flux.

Following  \cite{MANTRI2023112673}, we can now  construct a finite element approximation $R_h$
of the flux $R$ which is, just as $S$, in the space $V^K_h$ of polynomials of degree $K$ (despite being a primitive of $S$), by using an integral operator $I_x$ which, in FEM, is local for each cell. This leads to 
\begin{equation}
	R_\alpha = R_0 - \int_{x_0}^{x_\alpha} S_h(x) \dd x, \, \forall \alpha   \Longleftrightarrow R=R_0 - I_x S,
\end{equation} 
with $R_\alpha$ the degrees of freedom of $R_h$ and the right notation is a vectorial version of the left one.

The continuous SEM approximation of the balance law  (with periodic BCs or neglecting BCs) 
can be succinctly written as 
\begin{equation}\label{eq:GF1DSEM}
M_x\dfrac{dq}{dt} + D_x {F} - D_xI _x {S} =0.
\end{equation}
where $M_x$ is the mass matrix, $D_x$ is a finite element weak derivative matrix and $I_x$ is the integrator localized in each cell (see Section~\ref{sec:supggfintegrators} for precise definitions of the notation).
In this approach, we have  replaced the mass matrix  $M_x$ in front of the nodal values of the source,
with   the product matrix $ D_xI_x $.    This modification  allows to factor the derivative matrix, so that 
 at steady state  the  scheme reduces to  
$$
{F} = {F}_0 + I _x {S} \,.
$$
The integrator naturally turns out to be the ODE solver associated  to the table $I_x$ applied to the flux ODE  \cite{MANTRI2023112673}
 $$
 F' = S(q(F),x)\,.
 $$
This    provides   a  clear characterization of the discrete steady state and gives a so-called approximate or discrete
well balanced principle, in the spirit of e.g. \cite{Castro2020,math9151799}. 
The integration table $I_x$ defines the  properties of the ODE solver. For Lagrangian basis functions on Gauss-Lobatto points, the well known LobattoIIIA
methods arise~\cite{Prothero74,hairer}. This method has  nodal consistency of order $h^{K+2}$ at internal nodes,
and $h^{2K}$ at  end-nodes, thus leading to a super-convergent method at steady state.

Due to the fully local structure of the method, and to the fact that it merely involves  a  particular approximation of the  weighted source integral $\int\varphi S$, 
the authors of \cite{MANTRI2023112673} proposed to refer to it as   \emph{global flux quadrature} (GFq). In this work we propose a genuine generalization of the above idea to multi-dimensional equilibria.
When considering the last equation in~\eqref{eq:acoustic}, we observe that it can be written in two ways
$$
\begin{aligned}
\del_t p + \del_x u + \del_y v =  \del_t p +  \del_x u +\del_x  \bigg(\int_{x_0}^x \del_y v \, \dd s\bigg)  =  \del_t p +   \del_y v +\del_y  \bigg(\int_{y_0}^y \del_x u\, \dd s\bigg)  = 0.
\end{aligned}
$$
We propose to couple the $x$ and $y$ derivatives by symmetrizing
the two directional global flux quadrature formulations as
\begin{equation}\label{eq:GF2Didea}
 \del_t p + \del_x \del_y  \bigg(\int_{y_0}^y \del_x u \,\dd s + \int_{x_0}^x \del_y v\, \dd s\bigg)  = 0.
\end{equation}
In this paper we combine this idea with a high-order grad-div stabilized continuous Finite Element approximations leading to stationarity preserving methods.


\subsection{Overview of the paper}

The paper is structured as follows: Section~\ref{sec:finitedifference} introduces a difference/matrix notation for tensor-product FEM spaces on Cartesian grids that is used subsequently. In Section~\ref{sec:standardstab}, we prove that classical grad-div stabilizations (such as SUPG and OSS) are in general not constraint preserving: the kernels of the stabilizing term and of the Galerkin term do not have a sufficiently large intersection. Global Flux gives rise to a discretization of the divergence different from the one of continuous FEM; it is analyzed in Section~\ref{sec:globalflux}, where in particular exact projections in the discrete kernel space and nodal consistency estimates and the super-convergent behavior are provided (Section~\ref{sec:well_prepared_ic}). The Global Flux approach is applied in Sections~\ref{ssec:globalfluxsupg} and~\ref{ssec:globalfluxoss} to construct constraint-compatible SUPG and OSS stabilizations. In Section~\ref{sec:kernels}, spurious modes in the kernels of the discrete operators are studied. In Section~\ref{ssec:discreteinvolutions}, curl involutions are characterized using Fourier symbols, and explicit formulas are provided in the $\mathbb Q^1$ case.
The time discretization is described in Section~\ref{sec:time} and numerical results follow in Section~\ref{sec:numerical}.

\section{Cartesian grids, tensor products, and  discrete Fourier transform for  Finite Elements} \label{sec:finitedifference}

\subsection{General definitions}

\subsubsection{One-dimensional Finite Element spaces}

The study shall be restricted to Cartesian grids.  We consider two one-dimensional domains $\Omega^x,\Omega^y\subset \mathbb R$ and their product $\Omega := \Omega^x \times \Omega^y$. We define a uniform tessellation of each one-dimensional domain $\Omega_{\Delta x}^x = \cup_{i=0}^{N_x-1} E^x_i$, $\Omega^y_{\Delta y} = \cup_{j=0}^{N_y-1} E^y_j$, of elements $E_i^x=[x_{i},x_{i+1}],E_j^y=[y_{j},y_{j+1}]$ with $|E^x_i|=\Delta x$ and $|E_j^y|=\Delta y$ for all $i,j$. 

To define the continuous Finite Element spaces, we introduce $K+1$ points in each one-dimensional cell $$x_{i,p} = \hat{x}_p \Delta x+x_{i} \in E^x_i, \text{ for }p=0,\dots,K, \text{ and }y_{j,\ell} = \hat{x}_\ell \Delta y+y_{j}\in E^y_j, \text{ for }\ell=0,\dots,K$$ that we will use to define the Lagrangian basis functions.
In particular, here we will consider points $\lbrace\hat{x}_p\rbrace_{p=0}^K \subset [0,1]$ with $\hat x_0 = 0< \dots < x_i < \dots <\hat x_K = 1$, e.g. those of Gauss--Lobatto, such that $x_{i,0}=x_{i-1,K}$ for $i=1,\dots,N_x-1$ and similarly for $y$. 

We introduce a unique numbering for the point $p$ in cell $i$ with a Greek alphabet index $\alpha := iK+p \in [0,N_xK]$, $p \in [0, K-1]$, so that we will refer uniquely to point $x_\alpha = x_{i,p}$. Let us define by $M_x+1=N_xK+1$ and $M_y+1=N_yK+1$ the number of points in each direction. There are on average $K-1$ points per cell, but each cell has access to $K$ points. We will switch between these two notations according to our needs.

We can now introduce the continuous Finite Element spaces of degree $K$ over one/two-dimensional domains as
\begin{subequations}
\begin{align}
	V_\dx:=V^K_{\Delta x}(\Omega_{\Delta_x}^x):=&\left\lbrace q \in \mathcal{C}^0(\Omega_{\dx}^x): q|_E\in \mathbb P^K(E), \forall E \in \Omega^x_\dx\right\rbrace,\\
	V_\dy:=V^K_{\dy}(\Omega_{\Delta_y}^y):=&\left\lbrace q \in \mathcal{C}^0(\Omega_{\dy}^y): q|_E\in \mathbb P^K(E), \forall E \in \Omega^y_\dy\right\rbrace,
\end{align}
\end{subequations}
where we denote by $\mathbb P^K$ the space of univariate polynomials of degree at most $K$ and by $\mathbb Q^K$ the space of multivariate polynomials of degree at most $K$ in each variable.

In particular, we choose as basis of these spaces the high order hat functions that interpolate the points defined above. In each one-dimensional cell $E_i^x$, we consider $\varphi^x_{i,p}(x) \in V_{\Delta x}$ such that $\varphi^x_{i,p}|_{E^x_i}(x) \in\mathbb P^K(E^x_i)$ and $\varphi^x_{i,p}(x_{i,\ell}) = \delta_{j,\ell}$ for all $\ell,p = 0,\dots,K$, with $\delta$ the Kronecker delta. Moreover, since $\varphi$ must be continuous, we have
$$\text{supp}(\varphi^x_{i,p})= E^x_i\text{ for }p=1,\dots,K-1,\quad \text{supp}(\varphi^x_{i,0}) = E^x_{i-1} \cup E^x_i \text{ and } \text{supp}(\varphi^x_{i,K}) = E^x_{i} \cup E^x_{i+1},$$ 
recalling that $\varphi^x_{i-1,K}=\varphi^x_{i,0}$. 
The same holds for the basis functions of the $y$ space.
Moreover, $V_\dx^K(\Omega_{\Delta x}^x)=\text{span}\lbrace \varphi^x_\alpha \rbrace_{\alpha=0}^{M_x}$ with $\varphi^x_\alpha(x) \equiv \varphi^x_{i,p}(x)$ for $\alpha=iK+p$. We use the same spaces to discretize vector components as those we use for scalars.

\subsubsection{Tensor-product Finite Element spaces}

We define the two dimensional tessellation of $\Omega$ as
\begin{align}
	\Omega_h  = \bigcup_{i,j=0}^{N_x-1,N_y-1} E_{ij} 
\end{align}
with $h=\min \lbrace \Delta x, \Delta y \rbrace$ and $E_{ij}:= E^x_i \times E^y_j$. 

This leads to the definition of $V_h$ as a tensor product of the functional spaces
\begin{align}
 V_h:=V^K_h(\Omega_h):=&\left\lbrace q \in \mathcal{C}^0(\Omega_h): q|_E\in \mathbb Q^K(E), \forall E \in \Omega_h\right\rbrace,
\end{align}
which is evident in its basis $\lbrace \varphi_{\alpha;\beta} \rbrace_{\alpha,\beta=0}^{M_x,M_y}$ with $\varphi_{\alpha;\beta}(x,y):=\varphi^x_{\alpha}(x) \varphi^y_\beta(y)$.
Finally, we will describe a function $q \in V_\dx$ as 
\begin{align}
q_h(x) = \sum_{\alpha=0}^{M_x} q_\alpha \varphi^x_\alpha(x)= \sum_{i=0}^{N_x-1}\sum_{p=0}^K q_{i,p}\,\varphi_{i,p}|_{E^x_i}(x) \label{eq:qhdef1d}
\end{align}
and a function $q \in V_h$ as 
\begin{align}
q_h(x,y) = \sum_{\alpha=0;\beta=0}^{M_x;M_y} q_{\alpha;\beta} \varphi_{\alpha;\beta}(x,y)=\sum_{i=0;j=0}^{N_x;N_y}\sum_{p=0;\ell=0}^{K;K} q_{i,p;j,\ell}\, \varphi^x_{i,p}|_{E^x_i}(x)\,\varphi^y_{j,\ell}|_{E^y_j}(y). \label{eq:qhdef2d}
\end{align}

See Figure~\ref{fig:FEM_DOFs} for a graphical representation of degrees of freedom (DOFs) $q_{i,p;j,\ell}$ in a cell $E_{ij}$ for $K=4$.

In the following, we will give a more Finite Difference flavored description of classical FEM operators, in order to introduce Fourier symbols.

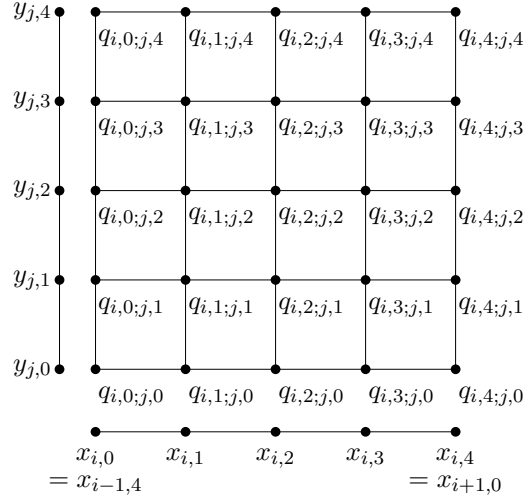
\begin{figure}
	\centering
\adjustbox{max width=0.45\textwidth}{
\begin{tikzpicture}[scale=1.3]
	\draw[step=1cm,black,very thin] (0,0) grid (4,4);
	\foreach \x in {0,...,4}
	{
		\foreach \y in {0,...,4}
		{
			\draw (\x+0.4,\y-0.3) node{$q_{i,\x;j,\y}$};
			\node[circle, fill=black, inner sep=0.5mm] at (\x,\y){};
		}
	}
	\draw[step=1cm,black,very thin] (0,-0.7) -- (4,-0.7);
	\foreach \x in {0,...,4}
	{
		\draw (\x,-1) node{$x_{i,\x}$};
		\node[circle, fill=black, inner sep=0.5mm] at (\x,-0.7){};
	}
	\draw (4.0,-1.3) node{$=x_{i+1,0}$};
	\draw (-0.0,-1.3) node{$=x_{i-1,4}$};
	\draw[step=1cm,black,very thin] (-0.4,0) -- (-0.4,4);
	\foreach \y in {0,...,4}
	{
		\draw (-0.7,\y) node{$y_{j,\y}$};
		\node[circle, fill=black, inner sep=0.5mm] at (-0.4,\y){};
	}
\end{tikzpicture}
}
\caption{Notation of the degrees of freedom for a function $q_h$ in element $E_{ij}$ for $\mathbb Q^4$ elements}
\label{fig:FEM_DOFs}
\end{figure}

\subsection{Bridging finite element and uni-directional difference formulae}

We start with a definition which can be applied to 
 $\mathbb P^1$ FEM and finite differences which can be mapped onto each other. We assume periodic boundary conditions for simplicity.

\begin{definition}[Finite differences] \label{def:findiff}
 Consider a one-dimensional equidistant grid with values $(q_{i})_{i\in\mathbb Z}$ and a linear unidirectional finite difference formula
 \begin{align}
 (Dq)_i = \sum_{k \in \mathbb Z} \alpha_k q_{i + k}. \label{eq:findiff1d}
 \end{align}
 
 \begin{itemize}
  \item Define $k_{\text{max}} \in \mathbb N^0$ as the smallest value for which the sum can be restricted
  \begin{align}
   (Dq)_i \equiv \sum_{k = - k_\text{max}}^{k_\text{max}} \alpha_k q_{i + k}.
  \end{align}

  \item We call $D \colon \mathbb R^\mathbb Z \to  \mathbb R^\mathbb Z$ whose action on $q$ at $i$ is defined by \eqref{eq:findiff1d} the \emph{finite difference operator}. 
  
  \item We call a finite difference formula \emph{compact} if $\mathrm{supp}\, \alpha \subset \mathbb Z$ is finite, i.e., if the sum in \eqref{eq:findiff1d} is finite. We shall also call the corresponding finite difference operator \emph{compact} in this case.
  
  \item The \emph{characteristic polynomial} $\mathbb F_{t_x}(D)$ of $D$ is the univariate Laurent polynomial $\sum_{k \in \mathbb Z} \alpha_k t_x^k$ in $t_x$. 

 \end{itemize}
\end{definition}

Up to prefactors, the discrete Fourier transform (i.e. inserting $q_i := \exp(\ii k_x i \Delta x)$) of $(Dq)_i$ lets appear precisely its characteristic polynomial if one defines $t_x = \exp(\ii k_x \Delta x)$ (see e.g. \cite{barsukow17a}).

Throughout the paper we use arbitrarily high-order methods, discussed now. 
In the context of Galerkin methods, for $K>1$, different degrees of freedom are involved. 

\begin{example}
The following derivative operator evaluated in the DoF $\beta = (i,s)$ reads
\begin{align}
 (D_x q)_{i,s} = \int \phi^x_{i,s}(x) \partial_x q_h(x) \dd x = &
 \sum_{j=0}^{N_x-1} \int_{E_j^x} \phi^x_{i,s}(x)  \sum_{p=0,K} \partial_x \phi^x_{j,p}(x) q_{j,p}  \\
 \equiv&\sum_{k \in \lbrace -1,0,1\rbrace} \sum_{p= 0}^{K-1} q_{i+k,p} \int_{E_{i+k}^x} \phi^x_{i,s}(x)\phi^x_{i+k,p}(x) \dd x , \label{eq:Dxexamplestencil}
\end{align}
Due to translation invariance the right-hand side integral is going to depend only on $k$, $s$ and $p$, not on $i$.
\end{example}

This motivates the following

\begin{definition}[High-order differences]
On  a one-dimensional equidistant grid, embedding repeated sets of not necessarily equidistant collocation points, consider a linear, high-order unidirectional difference formula
 \begin{align}
  (Dq)_{i,s} = \sum_{k \in \mathbb Z} \sum_{p = 1}^K \alpha^s_{k,p} q_{i+k,p} \label{eq:highorderfdformulageneral}
 \end{align}
The fact that the collocation points are repeated implies translational invariance, which allows to choose $\alpha$ without a dependence on $i$. Here, $q_{i+k,p}$ are the same as in \eqref{eq:qhdef1d}.
 \begin{itemize}
  \item We call $D : \mathbb R^{\mathbb Z} \times [1, K]  \to \mathbb R^{\mathbb Z} \times [1, K]$ the \emph{high-order difference operator}. Observe that the operator has $K$ components (in function of which basis element the expression is tested against).
  \item Define $k_\text{max} \in \mathbb N^0$ to be the smallest integer for which one can restrict the summation:
  \begin{align}
  (Dq)_{i,s} \equiv \sum_{k = - k_\text{max}}^{k_\text{max}} \sum_{p = 1}^K \alpha^s_{k,p} q_{i+k,p} .
 \end{align}
 As is obvious from \eqref{eq:Dxexamplestencil}, for the usual operators appearing in FEM, $k_\text{max} = 1$.

 \item The \emph{characteristic polynomial} $\mathbb F_{t_x}(D)$ of $D$ is the matrix of univariate Laurent polynomials in $t_x$
 \begin{align}
  \mathbb F_{t_x}(D) := \left( \begin{array}{ccc}  \sum_{k \in \mathbb Z} \alpha^1_{k,1} t_x^k & \cdots &  \sum_{k \in \mathbb Z} \alpha^1_{k,K} t_x^k  \\
   \vdots & \ddots & \vdots \\
   \sum_{k \in \mathbb Z} \alpha^K_{k,1} t_x^k & \cdots & \sum_{k \in \mathbb Z} \alpha^K_{k,K} t_x^k 
   \end{array} \right), \qquad \text{i.e., }\mathbb F_{t_x}(D)_{s,p} = \sum_{k\in\mathbb Z}\alpha^s_{k,p} t_x^k .
  \end{align}
  (We still call this matrix a ``polynomial'' because it can be considered a polynomial in $t_x$ with matrix-valued coefficients.)
   \end{itemize}
\end{definition}

\begin{example} 
 Consider the mass matrix appearing in
 \begin{align}
  \Delta x (M_x q)_{i,s} = \int \phi^x_{i,s}(x) q_h(x) \dd x &=  \sum_{k =0}^{N_x-1} \sum_{p = 1}^K \left( \int_{E_k^x} \phi^x_{i,s} (x) \phi^x_{k,p}(x) \dd x \right )q_{k,p}.
 \end{align}
 Then, by comparison with \eqref{eq:highorderfdformulageneral} one finds (using the translational invariance)
 \begin{align}
  \alpha_{k,p}^s &= \int_{\mathbb R} \phi_{0,s}  \phi_{k,p}(x) \dd x.
 \end{align}
\end{example}

\begin{definition}[Composition in the high-order case] \label{def:compositionhighorder}
 Given two high-order unidirectional difference formulas on a one-dimensional grid
 \begin{align}
  (Aq)_{i,r} &= \sum_{k \in \mathbb Z} \sum_{s = 1}^K \alpha^r_{k,s} q_{i+k,s} &
  (Bq)_{i,r} &= \sum_{k \in \mathbb Z} \sum_{s = 1}^K \beta^r_{k,s} q_{i+k,s}, &
 \end{align}
 we define the \emph{composition} $AB$ of the high-order difference operators $A$ and $B$ by
 \begin{align}
  ((AB) q)_{i,r} &:= \sum_{k \in \mathbb Z} \sum_{s = 1}^K \alpha^r_{k,s} (Bq)_{i+k,s} = \sum_{k \in \mathbb Z} \sum_{s = 1}^K \sum_{k' \in \mathbb Z} \sum_{s' = 1}^K \alpha^r_{k,s}  \beta^s_{k',s'} q_{i+k+k',s'}.
 \end{align}
\end{definition}
%
%
%
%

\begin{proposition}\label{thm:charpolcomposition}
 The characteristic polynomial $(\mathbb F_{t_x}(RS))_{r,s} $ of the composition of two high-order difference operators $R$ and $S$ is the (matrix) product $\sum_{p = 1}^K (\mathbb F_{t_x}(R))_{r,p}  (\mathbb F_{t_x}(S))_{p,s}.$ of their characteristic polynomials.
\end{proposition}
The proof is given in Appendix~\ref{sec:proofcharpolcomposition}. In the high-order case the composition is not commutative, as can easily be seen from the fact that the associated operation on the characteristic polynomials is a matrix product.

\subsection{Tensor products and multidirectional difference formulae}

Consider now a 2-dimensional Cartesian grid as described in Section~\ref{sec:finitedifference} with $q_h\in V_h^K$.
We  consider the following generalized high order finite differences in this context.
 

Obviously, tensor based FEM allows to factor one dimensional operators. For example for the mass matrix one has 
\begin{equation*}\begin{split}
  \iint \phi^x_{i,s}(x) \phi^y_{j,p}(y) q_h(x,y) \dd x \dd y &= \sum_{E_{k\ell}\in\Omega_h} \sum_{r,t=0}^K \left(  \iint_{E_{k\ell}}  \phi^x_{i,s}(x) \phi^y_{j,p}(y)\phi^x_{k,r}(x) \phi^y_{\ell,t}(y)\dd x \dd y \right)  q_{k,r; \ell,t}\\
  &\!\!\!\!\!\!\!\!\!\!\!\!\!\!\!\!\!\!\!\!\!\!\!\!\!\!\!\!\!\!\!\!\!\!\!\!\!\!\!\!
  =\sum_{E_k^x \in \Omega_{\Delta x}^x}  \sum_{r,t=0}^K \left( \int_{E_k^x} \phi^x_{i,s}(x) \phi^x_{k,r}(x) \dd x  \right)  \sum_{E^y_\ell \in \Omega_{\Delta y}^y} \left(
   \int_{E^y_\ell} \phi^y_{j,p}(y) \phi^y_{\ell,t}(y) \dd y\right) q_{k,r; \ell, t} \\
   &\!\!\!\!\!\!\!\!\!\!\!\!\!\!\!\!\!\!\!\!\!\!\!\!\!\!\!\!\!\!\!\!\!\!\!\!\!\!\!\!
   =\sum_{k = 0}^{N_x-1} \sum_{\ell=0}^{N_y-1} \sum_{s,p = 0}^K q_{i+k,s; j+\ell,p} \int_{E^x_k} \phi^x_{i,r}(x) \phi^x_{i+k,s}(x) \dd x \int_{E^y_\ell} \phi^y_{j,t}(y) \phi^y_{j+\ell,p}(y) \dd y . \label{eq:q2massmatrixexample}
\end{split} \end{equation*}

This motivates the following definition.

\begin{definition}[Tensor-product high-order operators]
 Consider two high-order unidirectional difference formulas on 1-dimensional Cartesian grids
 \begin{align}
  (A u)_{i,r} &= \sum_{k \in \mathbb Z} \sum_{s = 1}^K \alpha^r_{k,s} u_{i+k,s} &
 (B v)_{j,t} &= \sum_{\ell \in \mathbb Z} \sum_{p = 1}^K \beta^t_{\ell,p} v_{j+\ell,p}.
 \end{align}
 \begin{itemize}
 \item Then, the linear bidirectional high-order difference formula applied on $q\in V_h^K$ 
 \begin{align}
  ((A \otimes B)q)_{i,r;j,t} &:= \sum_{(k,\ell) \in \mathbb Z^2} \sum_{s,p = 1}^K \alpha^r_{k,s} \beta^t_{\ell,p} q_{i+k,s;j+\ell,p}
 \end{align}
 is said to be the difference formula associated to the \emph{tensor product} $A \otimes B$ of the difference operators $A$ and $B$.

 \item The \emph{characteristic polynomial} $\mathbb F_{t_x,t_y}(A\otimes B)$ of a high-order difference operator $A\otimes B$ 
 is the following matrix of bivariate Laurent polynomials in $t_x, t_y$
  \begin{align}
 	(\mathbb F_{t_x, t_y}(A\otimes B))_{r,t} &:= \sum_{(k,\ell) \in \mathbb Z^2} \sum_{s,p = 1}^K \alpha^{r}_{k,s} \beta^t_{\ell,p} t_x^k t_y^\ell.
 \end{align}
 \item The composition of two tensor-product high-order operators $R := A \otimes B$, $S := C \otimes D$ is defined as
 \begin{align}
 	(RS)q &:= R(Sq).
 \end{align}
 \end{itemize}
\end{definition}

\begin{proposition}[Fourier transform of the tensor product in the high-order case]
 The characteristic polynomial $\mathbb F_{t_x,t_y}(A \otimes B)$ of $A \otimes B$ is the standard Kronecker product of the polynomials of $A$ and $B$:
 \begin{align}
  \mathbb F_{t_x,t_y}(A \otimes B) = \mathbb F_{t_x}(A) \otimes \mathbb F_{t_y}(B).
 \end{align}
\end{proposition}
\begin{proof}
 Using the definition, we obtain
 \begin{align}
  \mathbb F_{t_x, t_y}(A \otimes B)_{r,z} &= \sum_{(k,\ell) \in \mathbb Z^2} \sum_{s,p = 1}^K \alpha^{r}_{k,s}\beta^{z}_{\ell,p} t_x^k t_y^\ell 
 =\sum_{k \in \mathbb Z} \sum_{s=1}^K  \alpha^{r}_{k,s} t_x^k\sum_{\ell \in \mathbb Z} \sum_{p = 1}^K \beta^{z}_{\ell,p}  t_y^\ell, 
 \end{align}
which is the statement of the theorem.
\end{proof}



\begin{proposition} \label{thm:highordercompositiontensor}
Consider high-order unidirectional difference formulas on a two-dimensional grid
 \begin{align}
 	\begin{split}
  (A^xu)_{i,r} = \sum_{k \in \mathbb Z} \sum_{s = 1}^K (\alpha^x)^{r}_{k,s}  u_{i+k,s},\qquad &
  (A^yv)_{j,t} = \sum_{\ell \in \mathbb Z} \sum_{p = 1}^K (\alpha^y)^{t}_{\ell,p}  v_{j+\ell,p}, \\
  (B^xu)_{i,r} = \sum_{k \in \mathbb Z} \sum_{s = 1}^K (\beta^x)^{r}_{k,s}  u_{i+k,s},\qquad  &
  (B^yv)_{j,t} = \sum_{\ell \in \mathbb Z} \sum_{p = 1}^K (\beta^y)^{t}_{\ell,p}  q_{j+\ell,p} .
 	\end{split}
 \end{align}
 The composition of tensor products is the tensor product of compositions:
 \begin{align}
  (A^x \otimes A^y)  (B^x \otimes B^y) = (A^x B^x) \otimes (A^y B^y).
 \end{align}

\end{proposition}
The proof can be found in the Appendix~\ref{sec:proofhighordercompositiontensor}.

In the paper we use  several   operators  which are listed hereafter for completeness:
\begin{subequations}\label{eq:FEM1Dmatrices}
\begin{align}
    &(M_x)_{\alpha,\beta}:= \int_{\mathbb R} \varphi_\alpha^x(x) \varphi_\beta^x (x ) \dd x,\quad & (\id_x)_{\alpha,\beta}=\delta_{\alpha,\beta},\quad &\\
	&(D_x)_{\alpha,\beta}:= \int_{\mathbb R} \varphi_\alpha^x(x) \partial_x \varphi_\beta^x (x ) \dd x,\quad 
	&(D^x)_{\alpha,\beta}:= \int_{\mathbb R} \partial_x \varphi_\alpha^x(x) \partial_x \varphi_\beta^x (x ) \dd x,\\	
	&(D_x^x)_{\alpha,\beta}:= \int_{\mathbb R} \varphi_\alpha^x(x) \partial_x \varphi_\beta^x (x ) \dd x. &
\end{align}	
\end{subequations}

\section{Failure of standard grad-div stabilizations for acoustics}
\label{sec:standardstab}

\subsection{SUPG stabilization}
\label{sec:supgstandard}

The stabilized variational form of the SUPG method by Hughes and collaborators \cite{brooks82,hughes86} can be written as 
\begin{equation}\label{eq:supg0}
\int \phi  ( \partial_t q +  J^x\partial_x q + J^x\partial_y q ) dx
+ \int  \alpha h (J^x\partial_x   \phi+ J^y\partial_y  \phi ) ( \partial_t q +  J^x\partial_x q + J^y\partial_y q )= \mathsf{B.C.s}
\end{equation}
with $\alpha$ a stabilization constant/matrix and $h$ a reference mesh size. 
The stability of the method can be shown by
 replacing the test function $ \phi$ by $q + \alpha h\,q_t$ for constant $\alpha$, (neglecting boundary condition terms, see also \cite{burman_supg}). The natural energy norm of SUPG given by
$$
E_{\text{SUPG}} := \int  \left\{ \frac{ q^Tq}{2} + (\alpha\,h)^2  (J^x\partial_x q + J^y\partial_y q)^T(J^x\partial_x q + J^y\partial_y q )\right\}dx\,
$$
such that
$$
 \partial_tE_{\text{SUPG}}
= -\int\alpha\, h ( \partial_t q +  J^x\partial_x q + J^x\partial_y q )^T( \partial_t q +  J^x\partial_x q + J^y\partial_y q ) dx \le 0.
$$
As said in Section~\ref{sec:intro}, the method naturally includes a grad-div structure in the stabilization. It thus seems to fit exactly the framework of stationarity preserving methods. However, it actually fails
to retain such a property. 
To show it we exploit the finite element/differences bridge presented in the previous sections,
and follow the spectral analysis of   \cite{barsukow17a}.

%

Recall the notation of difference operators in \eqref{eq:FEM1Dmatrices} and that $D_x = -D^x$ up to boundary conditions.
We can now write \eqref{eq:supg0} as  
\begin{equation}\label{eqs:standard_SUPG}
\begin{pmatrix}
 M_x \otimes M_y & 0 &  \alpha\,h D^x \otimes M_y \\ 0 & M_x \otimes M_y & \alpha\,h M_x \otimes D^y\\ \alpha\,h D^x \otimes M_y & \alpha\,h M_x \otimes D^y & M_x \otimes M_y  \end{pmatrix}
 \frac{\dd}{\dd t} \veccc{u}{v}{p}  
+\mathcal{E}_{SUPG}  \veccc{u}{v}{p}  =0  
\end{equation}
having introduced the evolution operator
\begin{equation}\label{eq:standardsupgfindiff}
\mathcal{E}_{\textrm{SUPG}}:= \begin{pmatrix}\alpha\,h D^x_x \otimes M_y &\alpha\,h D^x \otimes D_y & D_x \otimes M_y \\\alpha\,h D_x \otimes D^y &\alpha\,h M_x \otimes D^y_y & M_x \otimes D_y \\ D_x \otimes M_y & M_x \otimes D_y &\alpha\,h D^x_x \otimes M_y +\alpha\,h M_x \otimes D^y_y \end{pmatrix} .
\end{equation}

Let us split the matrices defined above into the central discretization and the stabilization (streamline upwinding) denoting by $\mathcal{A}_{\textrm{SUPG}}:=\mathcal{A}_C + \mathcal{A}_{\textrm{SU}}$ the matrix in front of the time derivative term, with
\begin{subequations}
\begin{align}
	\mathcal{A}_C &:= \begin{pmatrix}
		M_x \otimes M_y & 0 & 0\\ 0 & \!\!\!\!\!\!M_x \otimes M_y &0\\ 0 & 0 & \!\!\!\!\!\! M_x \otimes M_y  \end{pmatrix}
	,\,
	\mathcal{A}_{\textrm{SU}}  :=\alpha\,h \begin{pmatrix}
		0 & 0 &   \!\!\!\!\!\!D^x \otimes M_y \\ 0 & 0 &    \!\!\!\!\!\! M_x \otimes D^y\\  D^x \otimes M_y &    M_x \otimes D^y & 0  \end{pmatrix}, 
\end{align}
as well as for the matrix $\mathcal{E}_{\textrm{SUPG}}=\mathcal{E}_C + \mathcal{E}_{\textrm{SU}}$ with
\begin{align}
	\mathcal{E}_C &: =\begin{pmatrix}
		0 & 0 & D_x \otimes M_y \\ 0 & 0 & M_x \otimes D_y \\ D_x \otimes M_y & M_x \otimes D_y & 0 
	\end{pmatrix},\\
\mathcal{E}_{\textrm{SU}}&:=\alpha\,h
\begin{pmatrix}
	  D^x_x \otimes M_y &    D^x \otimes D_y & 0 \\  D_x \otimes D^y &   M_x \otimes D^y_y &0 \\ 0 &0 &    D^x_x \otimes M_y +   M_x \otimes D^y_y \end{pmatrix}.
\end{align}
This way, \eqref{eqs:standard_SUPG}  can be rewritten for $q = (u,v,p)^T$ as 
\begin{equation}\label{eq:standardsupgfindiff_compact}
	0=\mathcal{A}_{\textrm{SUPG}} \frac{\dd}{\dd t} {q} + \mathcal{E}_{\textrm{SUPG}} {q} =(\mathcal{A}_C + \mathcal{A}_{\textrm{SU}}) \frac{\dd}{\dd t} {q} + (\mathcal{E}_C+\mathcal{E}_{\textrm{SU}} ){q}.
\end{equation}
To be noted that $\mathcal{A}_C$ and $\mathcal{E}_{SU}$ are symmetric positive (semi-)definite, while $\mathcal{A}_{SU}$ and $\mathcal{E}_{C}$ are anti-symmetric matrices.

\end{subequations}

Consider now the lowest-order SUPG with $\mathbb Q^1$ basis functions. As there is only one degree of freedom per cell, SUPG can be immediately interpreted as a finite difference method, and its properties can be analyzed using techniques from \cite{barsukow17a}. Assuming for simplicity that $\Delta x = \Delta y = h$ and recalling that the quadrature formula and the Lagrangian basis functions are defined with the same Gauss-Lobatto points, we have
\begin{align}
 \mathbb F_{t_x}(M_x) &=1, &\mathbb F_{t_x}(D_x) = - \mathbb F_{t_x} (D^x)&= \frac{t_x^2 - 1}{2 t_x \,h },  &   \mathbb F_{t_x}(D^x_{x})= -\frac{(t_x - 1)^2}{t_x  \,h^2} .
\end{align}
Then, 
\begin{equation}\begin{split}
		\mathbb F_{t_x,t_y}(\mathcal E_{\textrm{SUPG}}) = \begin{pmatrix}
				- \alpha \frac{(t_x - 1)^2}{h  t_x  }  &
				- \alpha \frac{t_x^2 - 1}{2 t_x } \frac{t_y^2 - 1}{2 h t_y } & 
				\frac{t_x^2 - 1}{2h t_x }\\ 
				- \alpha \frac{t_x^2 - 1}{2 h t_x } \frac{t_y^2 - 1}{2 t_y } & 
				- \alpha \frac{(t_y - 1)^2}{h t_y } & 
				\frac{t_y^2 - 1}{2 h t_y } \\ 
				\frac{t_x^2 - 1}{2 h t_x }  &
				 \frac{t_y^2 - 1}{2 h t_y  } & 
				- \alpha \frac{(t_x - 1)^2}{h t_x }  - \alpha  \frac{(t_y - 1)^2}{h t_y } 
			\end{pmatrix}
			.
	\end{split}\end{equation}
As shown in \cite{barsukow17a}, all non-trivial stationary  states are given as the right kernel of $\mathcal E$, while its left kernel 
gives the corresponding involutions (if any). However note now that
\begin{align}\label{eq:det_Q1_SUPG}
 \det \mathbb F_{t_x,t_y}(\mathcal E_{\textrm{SUPG}}) =  
 \alpha (t_x-1)^2 (t_y-1)^2 (\ldots) \neq 0,
\end{align}
i.e., its kernel is trivial unless $t_x = 1$ or $t_y = 1$ (functions are constant in $x$ or $y$) or $\alpha = 0$ (no stabilization). This method does not have non-trivial stationary states, and for the same reason also no discrete involutions. Without proof we note that the same result holds for the $\mathbb Q^2$ case.

A more general characterization, still for the $\mathbb Q^1$-case, can be obtained observing that 
\begin{align}
	\begin{split}
 \det \mathbb F_{t_x,t_y}(\mathcal E_{\textrm{SUPG}}) &= \alpha^3 h^3 \Big(\mathbb F_{t_y}(D_{y}^y)\mathbb F_{t_x}( M_x) + \mathbb F_{t_x}(D_{x}^x)\mathbb F_{t_y}( M_y)\Big) 
 \times\\
 & \Big(\mathbb F_{t_x}(D_{x}^x)\mathbb F_{t_y}( D_{y}^y)\mathbb F_{t_x}( M_x) \mathbb F_{t_y}(M_y) - \mathbb F_{t_x}(D_x)^2 \mathbb F_{t_y}(D_y)^2 \Big) \\
 &- \alpha h \mathbb F_{t_x}(M_x)\mathbb F_{t_y}( M_y) \Big(\Big(\mathbb F_{t_x}(D_{x}^x) \mathbb F_{t_x}( M_x) + \mathbb F_{t_x}(D_x)^2\Big) \mathbb F_{t_y}(D_y)^2 
 +\\
 & \Big(\mathbb F_{t_y}(D_{y}^y)\mathbb F_{t_y}( M_y)+  \mathbb F_{t_y}(D_y)^2\Big)\mathbb F_{t_x}(D_x)^2\Big),
	\end{split}
\end{align}
which vanishes if
\begin{align}
 \mathbb F_{t_x}(D_{x}^x)\mathbb F_{t_x}( M_x) =- \mathbb F_{t_x}(D_x)^2  \quad \text{ and }\quad \mathbb F_{t_y}(D_{y}^y)\mathbb F_{t_y}( M_y) = -\mathbb F_{t_y}(D_y)^2,
\end{align}
i.e.,
\begin{align}
 D_{x}^x M_x =- D_x^2  \quad \text{ and }\quad D_{y}^yM_y =- D_y^2.
\end{align}
Note that for $\mathbb Q^1$ FEM discussed here, composition of difference operators is commutative. For general FEM, we have the following result.

\begin{proposition}\label{th:supgfail}
 Define the two operators
 \begin{align}
  \del_x \mathrm{DIV} \vec v &:= -(D_{x}^x \otimes M_y) u - (D^x \otimes D_y) v ,\qquad 
  \mathrm{DIV} \vec v := (D_x \otimes M_y) u + (M_x \otimes D_y) v .
 \end{align}
 Then the following two statements are equivalent:
 \begin{enumerate}
  \item 
  \begin{align}
  D_{x}^x = D^x  M_x^{-1}  D_x \label{eq:massmatrixsecondderivcorrespondence}
 \end{align}
 \item For all $u,v$ such that $\mathrm{DIV} \vec v \equiv 0$, $\del_x \mathrm{DIV} \vec v=0$ holds.
 
 \end{enumerate}
\end{proposition}
\begin{proof}
Assume first \eqref{eq:massmatrixsecondderivcorrespondence} and $\mathrm{DIV} \vec v = 0$:
 \begin{align*}
  \del_x \mathrm{DIV} \vec v &= -(D_{x}^x \otimes M_y) u - (D^x \otimes D_y) v \\
   &= -(\underbrace{D^x  M_x^{-1}  D_x}_{D_{x}^x} \otimes M_y) u - (D^x  \underbrace{M_x^{-1}  M_x}_{\id} \otimes D_y) v \\
   &= -(D^x \otimes M_y)  (M_x^{-1} \otimes M_y^{-1})  (D_x \otimes M_y) u - (D^x \otimes M_y)  (M_x^{-1}  \otimes M_y^{-1})  (M_x \otimes D_y) v \\
   &= -(D^x  M_x^{-1} \otimes \id_y)  \Big ( (D_x \otimes M_y) u +  (M_x \otimes D_y) v\Big) = 0.
 \end{align*} 
 Conversely, 
 \begin{align}
  0 = \del_x \mathrm{DIV} \vec v &= - (D_{x}^x \otimes M_y) u - (D^x \otimes D_y) v \\
   &= -(D_x^x \otimes M_y) u - (D^x \otimes D_y) v + (D^x  M_x^{-1} \otimes \id_y)  \underbrace{\Big ( (D_x \otimes M_y) u +  (M_x \otimes D_y) v\Big)}_{=\mathrm{DIV} \vec{v}=0}\\
   &= -(D_{x}^x \otimes M_y) u - (D^x \otimes D_y) v+ (D^x  M_x^{-1}  D_x \otimes M_y) u +  (D^x \otimes D_y) v\\
   &= -(D_{x}^x \otimes M_y) u + (D^x  M_x^{-1}  D_x \otimes M_y) u . \label{eq:resultdiv}
 \end{align}
 The set of $(u,v)$ that satisfy $\textrm{DIV}(u,v)=0$ is very large, and e.g. $u$ can be considered unconstrained. As \eqref{eq:resultdiv} shall be true for all those $u$ one concludes $D_{x}^x = D^x  M_x^{-1}  D_x$.
\end{proof}

\begin{proposition}
 The stabilization terms of standard SUPG do not vanish when the Galerkin approximation of the divergence vanishes, because for $\mathbb Q^k$ FEM \eqref{eq:massmatrixsecondderivcorrespondence} is never true.
\end{proposition}
\begin{proof}
 We show the proof for the Gauss-Lobatto basis functions with Gauss-Lobatto quadrature that we will use in the numerical section. In this configuration the mass matrix is diagonal and all diagonal terms are different from zero.
 
 Now, we want to show that at least one element of $D^xM_x^{-1}D_x$ is different from the one of $D^x_x$. We consider the entry $(i-1,0;i,K)$ for any $i$ and we show that this entry is 0 in $D^x_x$ but not in the other matrix.
 \begin{subequations}
\begin{align}
	(D^x(M_x)^{-1}D_x)_{i-1,0; i,K} :=&\sum_{j\in \mathbb Z, r \in [0,K-1]}\int \varphi_{i-1,0}' \varphi_{j,r}  \dd x  \frac{1}{(M_x)_{j,r;j,r}}\int \varphi_{j,r} \varphi_{i,K}'  \dd x \\
	=&\frac{1}{(M_x)_{i,0;i,0}}\int \varphi_{i-1,0}' \varphi_{i,0}  \dd x\int \varphi_{i,0} \varphi_{i,K}' \dd x \label{eq:mult_diff} \\
	=&\frac{1}{(M_x)_{i,0;i,0}}\underbrace{\int \varphi_{i-1,0}' \varphi_{i-1,K}  dx}_{\neq 0} \underbrace{ \int \varphi_{i,0} \varphi_{i,K}' dx}_{\neq 0} \neq 0, \label{eq:mult_diff_neq}\\
	(D^x_x)_{i-1,0; i,K} : =& \int \varphi_{i-1,0}'\varphi_{i,K}' \dd x =0. \label{eq:diff_term_zero}
\end{align}
\end{subequations}
In \eqref{eq:mult_diff}, we have used the fact the only basis function that has support both on the support of $\varphi_{i-1,0}$ and $\varphi_{i,K-1}$ is $\varphi_{i,0}=\varphi_{i-1,K}$. Then, explicitly using the quadrature formula, we observe that $\int_{E_i^x} \varphi_{i,s}' \varphi_{i,r} \dd x = \dx w_{r}  \varphi_{i,s}'(x_{i,r})$, with $ w_r = 1/\dx\int_{E_i^x} \varphi_{i,r}$ being the $r$-th quadrature weight of the Gauss--Lobatto formula. Now, $\varphi_{i,s}'(x_{i,r})\neq 0$ for $r\neq s$ otherwise $x_{i,r}$ would have been both a zero of $\varphi_{i,s}$ and a local extremum. In this case, this zero of $\varphi_{i,s}$ would have multiplicity higher than one, but, by definition, it has multiplicity one. Hence, it must be different from 0.
On the other hand, in \eqref{eq:diff_term_zero} $\varphi_{i-1,0}'$ has support only in $E^x_{i-2}$ and $E^x_{i-1}$ and $\varphi_{i,K}$ has support only in $E^x_i$ and $E^x_{i+1}$, so the integral is zero. This concludes the proof.
\end{proof}

This Proposition does not allow to conclude that SUPG fails to be stationarity preserving because it might have some other non-trivial discrete stationary states, which are not governed by the Galerkin approximation $\mathrm{DIV} \vec v$ of the divergence. However, at least for $\mathbb Q^1$ (Equation \eqref{eq:det_Q1_SUPG}) and $\mathbb Q^2$ (without proof) SUPG does not possess nontrivial discrete stationary states.
The approach proposed later in the paper  allows to side-step the limitations highlighted here without  imposing the constraints of Theorem   \ref{th:supgfail}.

\subsection{grad-div Orthogonal Subscale Stabilization (OSS)}
\label{sec:ossstandard}


The Orthogonal Subscale Stabilization (OSS) is a stabilization technique introduced originally for Stokes equations \cite{CODINA1997373} and then extended for other problems, including convection--diffusion--reaction problems \cite{CODINA20001579,OSSCodinaBadia}. For hyperbolic equations, it has been studied in  \cite{michel2021spectral,michel2022spectral}, in particular its fully discrete Fourier  stability
when coupled with explicit time integration methods. 
The OSS stabilization technique  allows to use any dissipative operator by introducing  a   penalization term   composed
by the variational approximation of the dissipative operator  minus the same quantity evaluated using an $L^2({\Omega})$ projection
of the appropriate operator. For a Laplacian stabilization, for example, this gives a term of the form $\int_{{\Omega}} \varphi (\nabla q -w) =0$
with $w$ the projection of the gradient on the global approximation space.

For the acoustic system, we aim at constructing a grad-div based operator. We thus propose to study the  following stabilized variational form  
\begin{equation}\label{eq:OSS1}
\begin{split}
\int  \varphi  ( \partial_t u + \partial_x p ) dx  +& \int \alpha \, h \partial_x \varphi(\nabla \cdot \vec{u} - w^{\nabla \cdot \vec{u}} ) dx =0\\
\int  \varphi  ( \partial_t v + \partial_y p ) dx  +& \int \alpha \, h  \partial_y \varphi(\nabla \cdot \vec{u} - w^{\nabla \cdot \vec{u}} ) dx =0\\
\int  \varphi  ( \partial_t p + \partial_x u+ \partial_y v ) dx  +& \int \alpha \, h  \partial_x \varphi   (\partial_x p - w^p_x )dx  +\int \alpha \, h  \partial_y \varphi  (\partial_y p - w^p_y ) dx  =0\\
\end{split}
\end{equation}
with the projections $w^{\nabla \cdot \vec{u}},\, w^p_x$ and $w^p_y$ defined by
\begin{equation}\label{eq:OSS_projections}
	\begin{cases}
		\int  \varphi \left( \nabla\cdot \vec{u} - w^{\nabla\cdot \vec{u} } \right)dx=0, &\forall \varphi \in V_h,\\
		\int  \varphi \left( \partial_x p - w^p_x \right)dx=0, &\forall \varphi \in V_h,\\
		\int  \varphi \left( \partial_y p - w^p_y \right)dx=0, &\forall \varphi \in V_h.
	\end{cases}
\end{equation}
The stability of the method can be shown classically by replacing $\varphi$ by the velocities and pressure in the main system, summing up the results and removing 
from it the expression obtained by testing projections with  $\alpha\,h   (w^{\nabla\cdot \vec{u}},\;w^p_x ,\;w^p_y )^T$.

After some algebra, one shows that the energy stability  of the scheme is characterized by (neglecting boundary conditions, see also \cite{OSSCodinaBadia,michel2021spectral,michel2022spectral}):
$$
 \partial_t\int \frac{q^Tq}{2} dx 
= -\int\alpha\, h (   J^x\partial_x q + J^x\partial_y q - \tilde w)^T(J^x\partial_x q + J^x\partial_y q - \tilde w ) dx \le 0
$$
The above stabilized formulation seems a good candidate for being stationary preserving, as it involves the approximation of the proper differential terms, namely the grad-div  Laplacian for the velocity equations. Unfortunately, as for SUPG, a standard discretization of the operators involved fails to be stationarity preserving.
To show this, we proceed as done in the previous subsection and consider the semi-discrete version of the scheme.  We first consider the 
 projection which can be written as 
\begin{equation}
	\begin{cases}
		w^{\nabla \cdot \vec{u}} = (M_x \otimes M_y)^{-1}  ((D_x \otimes M_y)u + (M_x\otimes D_y) v),\\
		w^{p}_x = (M_x \otimes M_y)^{-1}  (D_x \otimes M_y)p,\\
		w^{p}_y  = (M_x \otimes M_y)^{-1} (M_x\otimes D_y) p,
	\end{cases}
\end{equation}
Then, inserting these definitions into the stabilization terms  we can show the following  for the horizontal velocity 
	\begin{equation}
	\begin{split}
	s^u = &		\alpha h \left[ (D^x_x \otimes M_y) u +(D^x \otimes D_y) v - (D^x\otimes M_y) (M_x \otimes M_y)^{-1}  ((D_x \otimes M_y)u + (M_x\otimes D_y) v) \right] \\
	=& 	\alpha h \left[(D^x_x \otimes M_y) u +(D^x \otimes D_y) v - (D^xM_x^{-1}D_x\otimes M_y) u- (D^x\otimes D_y) v\right] \\
	=&	\alpha h \left[ (D^x_x \otimes M_y) u  - (D^xM_x^{-1}D_x\otimes M_y) u\right]  = ((D^x_x-D^xM_x^{-1}D_x)\otimes M_y) u ,
		\end{split} 
	\end{equation}
which provides a coupled matrix representation of the stabilization term.
Introducing the matrices $Z_x:={D}^x_x-D^xM_x^{-1} D_x$ and $Z_y:=D^y_y-D^yM_y^{-1}D_y$, the  OSS stabilization terms can
be written in semi-discrete form as 
\begin{subequations}\label{eqs:OSS_stabs}
%
	\begin{align}
	s^u = &	\alpha hM_y  \otimes  Z_x u, \label{eq:OSS_u}\\
		s^v = &	\alpha h  M_x\otimes Z_y v , \label{eq:OSS_v}\\
		s^p=&  	\alpha h  (M_y \otimes  Z_x + M_x\otimes Z_y  ) p .\label{eq:OSS_p}
	\end{align}
%
The stabilization matrix is 
	\begin{align}
		\mathcal{E}_{\textrm{OSS}} := \alpha h \begin{pmatrix}
			Z_x \otimes M_y &0 &0 \\
			0 & M_x\otimes Z_y &0\\
			0&0&Z_x \otimes M_y + M_x\otimes Z_y
		\end{pmatrix} 
	\end{align}
\end{subequations}
and the OSS formulation can be succinctly written as
\begin{equation}\label{eq:OSS_with_matrix}
	\mathcal{A} \frac{\dd}{\dd t} \vec{q} + \mathcal{E} \vec{q}=0,\qquad  \text{with } \mathcal{A}=\mathcal{A}_C,\,\mathcal{E} = \mathcal{E}_C + \mathcal{E}_{\textrm{OSS}} .
\end{equation} 
Similarly to SUPG, we observe that 
\begin{equation}\begin{split}
		\mathbb F_{t_x,t_y}(\mathcal E) = \begin{pmatrix}
				\alpha h F_{t_x}(Z_x)&
				0 & 
				\frac{t_x^2 - 1}{2h t_x }\\ 
				0 & 
				\alpha h F_{t_y}(Z_y)& 
				\frac{t_y^2 - 1}{2 h t_y } \\ 
				\frac{t_x^2 - 1}{2 h t_x }  &
				\frac{t_y^2 - 1}{2 h t_y  } & 
				F_{t_x}(Z_x) + F_{t_y}(Z_y)
			\end{pmatrix}
			.
	\end{split}\end{equation}
As $F_{t_x}(Z_x)= -\frac{(t_x-1)^2}{t_xh^2}+\frac{(t_x^2-1)^2}{4t_x^2h^2}$ factors out a $(t_x-1)^2$ term and $F_{t_y}(Z_y)$ factors out a $(t_y-1)^2$ term, 
\begin{equation}
	\det \mathbb F_{t_x,t_y}(\mathcal E) = \alpha (t_x-1)^2(t_y-1)^2(\dots).
\end{equation}
This means that the kernel is only non-trivial ($\det\mathbb F_{t_x,t_y}(\mathcal E) = 0$) when $t_x=1$ or $t_y=1$ (functions constant in $x$ or $y$), i.e. the method is not stationarity preserving.


\newcommand{\factorTmp}{}
\newcommand{\factorTmpy}{}

\section{Global Flux quadrature and continuous Finite Elements} \label{sec:globalflux}

\subsection{Global flux quadrature in multi-D: the GFq divergence operator}\label{ssec:globalfluxgen}

The classical  Galerkin approximation of the divergence $\int \varphi (\partial_x u_h +\partial_y u_h) \dd x \dd y$   gives 

\begin{equation}\label{eq:div-sem}
\mathrm{DIV} \vec{v} =  D_x\otimes M_y u + M_x\otimes D_y v   \,.
\end{equation}
with the discrete operators defined in~\eqref{eq:FEM1Dmatrices}.
To generalize the 1D Global Flux idea to multiple dimensions, we use the symmetric  approximation  \eqref{eq:GF2Didea}, introducing the new notion of divergence
\begin{equation}\label{eq:div_UpV}
\partial_x u + \partial_y v \equiv \partial_{xy}(U+V),
\end{equation}
where 
\begin{align}
U  :=  U_0 + \int_{y_0}^y u\dd y \Longleftrightarrow \partial_yU =u, \qquad
V := V_0 + \int_{x_0}^x v \dd x \Longleftrightarrow \partial_xV =v.
\end{align}

As in one dimension, we construct discrete approximations of $U_h$ and $V_h$ 
in the same polynomial space of $u_h$ and $v_h$. 
To achieve this, we provide a line-by-line definition of $U_h$ and $V_h$ whose nodal values can be constructed as 
\begin{equation}\label{eq:bigU_V}
U  = \id_x \otimes I_y u\;,\quad   V = I_x \otimes \id_y v,
\end{equation}
with integration  operators defined for 1D FEM as 
\begin{equation}\label{eq:definition_integrator}
	\begin{split}
	&(I_x)_{i,s;k,p}:= \int_{x_{i,0}}^{x_{i,s}} \varphi_{k,p}^x (x ) \dd x\,\quad \text{for }s=1,\dots,K \quad\text{ and }\\
	&(I_y)_{i,s;k,p}:= \int_{y_{i,0}}^{y_{i,s}} \varphi_{k,p}^y (y ) \dd y\,\quad \text{for }s=1,\dots,K.
	\end{split}
\end{equation}
The modified weak divergence $\int \varphi \partial_{xy}(U_h+V_h) \dd x$ 
then reads
\begin{equation}\label{eq:div-GF}
	\begin{split}
\mathrm{DIV} \vec{v} &= 
D_x \otimes D_y (U + V) =  D_x \otimes D_yI_y u + D_xI_x\otimes D_y v = D_x \otimes D_y \left( \id_x\otimes I_y u + I_x \otimes \id_y v \right)   .
	\end{split}
\end{equation}
As in 1D, compared to \eqref{eq:div-sem} the latter formula essentially involves modifications of the mass matrices: they are replaced by the operators $D_xI_x$ and $D_yI_y$.   The GFq divergence operator \eqref{eq:div-GF} obtained with this modification
has a clear characterization of its kernel.

\begin{proposition}[Physically relevant part of the kernel of the GFq divergence]\label{th:div_kernel}
The  global flux quadrature divergence operator \eqref{eq:div-GF} vanishes identically for 
	\begin{align}
		U_{i,k;j,s}+V_{i,k;j,s}=f(i,k)+g(j,s). \label{eq:discreteUplusV}
	\end{align}
\end{proposition}
\begin{proof}
	By construction $D_x\otimes D_y (f+g)=0$ since $D_y f=0$ and  $D_xg=0$ which immediately yields the result.
\end{proof}

A neat way of writing the above property is obtained by using the local assembly of \eqref{eq:div-GF},
which reads using a FEM notation
$$
[\mathrm{DIV} \vec{v}]_{\alpha;\beta} =\sum_{E\ni (\alpha;\beta) }  [  D_x^E \otimes D_y^E  (\id_x^E \otimes I_y^E  \, u^E) ]_{\alpha;\beta} + 
\sum_{E\ni (\alpha;\beta) }     [D_x^E \otimes D_y^E   (I_x^E \otimes \id_y^E  \, v^E)]_{\alpha;\beta} \,.  
$$
where the superscript $^E$ denotes the local entries of operators and arrays inside the element $E$.
On the  element $E=E_{ij}$, consider now the local arrays
$$
[u_{0}^E]_{i,s;j,p}:= u_{i,0;j,p}\;\forall s=0,\dots,K\,, \quad [v_0^E]_{i,s;j,p}:= v_{i,s;j,0}\;\forall p=0,\dots,K\,.
$$
Define now the elemental array of integrated divergences on each element $E=E_{ij}$
\begin{equation}\label{eq:div-residual}
\begin{split}
\Phi^E :=   & (\id_x^E\otimes I_y^E)   (u^E -u_0^E)    + (I_x^E\otimes \id_y^E)   (v^E -v^E_0) , \\
\Phi^E_{i,s;j,p} = &\int_{y_{j,0}}^{y_{j,p}}( u_h(x_{i,s},y) - u_h(x_{i,0},y))\dd y +  \int_{x_{i,0}}^{x_{i,s}}( v_h(x,y_{j,p}) - v_h(x,y_{j,0}))\dd x .
\end{split}
\end{equation}
 Using the fact that $D_x^E \otimes D_y^Eu_0^E=D_x^E \otimes D_y^Ev_0^E=0$, we can readily see that 
\begin{equation}\label{eq:div-phi}
[\mathrm{DIV} \vec{v}]_{\alpha;\beta} =\sum_{E\ni (\alpha;\beta)} [(D_x^E \otimes D_y^E)\Phi^E]_{\alpha;\beta}.
\end{equation}

\begin{proposition}[GFq  divergence and vanishing subcell integrals]\label{th:div_kernel-phi}
The  global flux quadrature divergence operator \eqref{eq:div-GF} vanishes identically whenever  $\forall E_{ij}$
and $\forall  \,s,p\in E_{ij}$ the integrated divergence on the subcell $[x_{i,0},x_{i,s}] \times[y_{j,0},y_{j,p}]$ vanishes: 
\begin{equation*}
\begin{split}
\Phi^E_{i,s;j,p} = 0\; 
 \forall s,p  \text{ and } \,\forall E_{i,j} \Rightarrow \mathrm{DIV} \vec{v} =0.
\end{split}
\end{equation*}
\end{proposition}
The proof follows directly the previous computations.

The last proposition shows two important properties:
\begin{enumerate}
\item {\it  the approach introduced allows to  define, at steady state, as many linearly independent zero divergences 
as the number of nodes in the mesh;}
\item  {\it the  GFq approach  allows to  naturally pass from nodal  to face integrated quantities}.
Indeed,   $U$ and $V$ contain   integrated values of the  velocities in the directions
normal to the faces of the element sub-cells.  These are natural objects to express the integrals 
\begin{align}
 \iint (\del_x u + \del_y v) \dd x \dd y = \int [u]_x \dd y + \int [v]_y \dd x = \left[ \int u \dd y \right]_x + \left[ \int v \dd x \right ]_y .
\end{align}
This establishes a loose link to mimetic schemes  using  face averages of normal components.
\end{enumerate}

\subsection{Construction of the integrators in multi-D} \label{sec:supggfintegrators}

Finally, we give a general recipe how to construct integrators $I_x, I_y$ that allow to discretize
\begin{align}
	\partial_y U =u, &
	&\partial_x V=v\label{eq:UV_def}
\end{align}
in such a way that all the discrete spatial derivatives are compact differences, once the method is expressed in terms of $u$ and $v$.

%
%
%

To ensure a local nature of the method, a condition on $I_x, I_y$ is that $\mathcal E$, and in particular $D_x I_x, D^x_x I_x$, etc. have to be compact difference operators. For $\mathbb Q^1$ FEM this is ensured by choosing 
\begin{align}
I_x = \Delta x \frac{t_x+1}{2(t_x-1)} \label{eq:integratorsimpleP1}
\end{align}
because the characteristic polynomial of any finite difference formula that discretizes a derivative can always be divided by $t_x-1$ (see \cite{barsukow18low}). Observe the identities
\begin{align}
 D_x I_x &= \frac{t_x^2 + 2 t_x +1}{4t_x }, &
  D_{x}^x I_x =- D_x,
\end{align}
which follow from
\begin{align}
 \mathbb F_{t_x}(D_x I_x) &= \frac{(t_x+1)^2}{4 t_x}, & \mathbb F_{t_x}(D_{x}^x I_x) &= -\frac{(t_x-1)(t_x+1)}{2t_x \Delta x} =- \mathbb F_{t_x}(D_x).
\end{align}

The idea in the context of FEM is to take $u,v \in V_h^K $ and integrate them in $y$ and $x$, respectively, inside each cell $E_{ij}$. 
This would give piecewise $U(x,y)=\int^y u(x,s)\dd s \in \mathbb Q^{K}(E_i^x)\times \mathbb Q^{K+1}(E^y_{j})$ and $V(x,y) =\int^x v (s,y)\dd s\in \mathbb Q^{K+1}(E^x_i) \times \mathbb Q^{K}(E^y_{j})$. We then choose to project $U,V$ pointwise back onto $\mathbb Q^K(E_{ij})$. In particular, we write
%
\begin{subequations}
\begin{align}
	u(x,y)  &= \sum_{(i,j) \in \mathbb Z^2}\sum_{z,w = 1}^K u_{i,z;j,w}\phi^x_{i,z}(x) \phi^y_{j,w}(y)  = \sum_{(i,j) \in \mathbb Z^2}\sum_{z,w = 0}^K u_{i,z;j,w}\phi_{i,z}(x)\Big |_{E^x_{i}} \phi_{j,w}(y)\Big |_{E^y_{j}}, \\
	v(x,y)  &= \sum_{(i,j) \in \mathbb Z^2}\sum_{z,w = 1}^K  v_{i,z;j,w}\phi^x_{i,z}(x) \phi^y_{j,w}(y)= \sum_{(i,j) \in \mathbb Z^2}\sum_{z,w = 0}^K  v_{i,z;j,w}\phi^x_{i,z}(x)\Big |_{E^x_{i}} \phi^y_{j,w}(y)\Big |_{E^y_{j}},
\end{align}
\end{subequations}
see Figure~\ref{fig:FEM_DOFs}.
Integrating $u$ (or $v$) with respect to $y$ (or $x$), we get in the cell $E_{i,j}$:
\begin{subequations}
\begin{align}
	\begin{split}\label{eq:def_U_hat}
	\hat U(x,y)  := &\int_{y_\text{0}}^y u(x, y') \dd y' = \hat U(x, y_{j}) + \int_{y_{j}}^y u(x, y') \dd y'=\\
	& \hat U(x, y_{j}) 
	+ \sum_{z',w'=0}^{K} u_{i,z';j,w'} \int_{y_{j}}^y  \varphi^x_{i,z'}(x) \varphi^y_{j,w'}(y')   dy',
	\end{split}\\
	\begin{split}
	\hat V(x,y)  :=& \int_{x_0}^x v(x',y) \dd x' = \hat V(x_{i}, y) + \int_{x_{i}}^x v(x',y) \dd x'=\\
	& \hat V(x_{i}, y) + \sum_{z',w'=0}^{K} v_{i,z';j,w'} \int_{x_{i}}^x  \varphi^x_{i,z'}(x') \varphi^y_{j,w'}(y)   dx' .
	\end{split}
\end{align}
\end{subequations}

We will show below that the step-functions $\hat U(x, y_{j}), \hat V(x_{i}, y)$ are of no importance for the final form of the method, which will allow us to eventually drop them.
We also want to highlight that the restriction of $u$ on cell $E_{ij}$
\begin{align}
 U \Big |_{E_{ij}} = \sum_{z',w'=0}^{K} u_{i,z';j,w'} \int_{y_{j}}^y  \varphi^x_{i,z'}(x) \varphi^y_{j,w'}(y')   dy'
\end{align}
includes the degrees of freedom associated to $z'=0$ and $w'=0$, which belong also to the previous cell. Recall their definitions:
\begin{subequations}
\begin{align}
 u_{i,0;j,w'} &:= u_{i-1,K;j,w'}, \qquad \forall w' = 1, \ldots, K, \\
 u_{i,z';j,0} &:= u_{i,z';j-1,K}, \qquad \forall w' = 1, \ldots, K, \\
 u_{i,0;j,0} &:= u_{i-1,K;j-1,K}.
\end{align}
\end{subequations}

\begin{proposition}[Differentiation of integrals is independent on the starting value]
	Consider $U(x,y) : = \sum_{z,w} \varphi^x_{i,z}(x) \varphi^y_{j,w}(y) U_{i,z;j}$ in the cell $E_{ij}$ with some $U_{i,z;j}$ only depending on the cell $j$, not on $w$, the DoF in $y$. Then, $\partial_y U(x,y) = 0$ in the cell $E_{ij}$.
\end{proposition}
\begin{proof}
	Let us compute $\partial_y U(x,y)$ in the cell $E_{ij}$. To this end we need to include the degrees of freedom associated to $z,w=0$:
	\begin{align}
		\partial_y U(x,y) \Big|_{E_{ij}} 
		&= \partial_y \left(  \sum_{z,w=0}^{K} \varphi^x_{i,z}(x) \varphi^y_{j,w}(y) 
		U_{i,z;j} 
		\right)
		= \sum_{z=0}^{K} \varphi^x_{i,z}(x) 
		U_{i,z;j} 
		\partial_y \underbrace{\left(  \sum_{w=0}^{K} \varphi^y_{j,w}(y)\right)}_{\equiv 1} 
		= 0 . 
	\end{align}
	
\end{proof}

It this does not matter which constant we choose to define $U$ in each cell. 
As in the construction of the method only $\del_y U$ appears, the term $\hat U(x, y_{j})$ in \eqref{eq:def_U_hat} can be dropped straight away. We thus define
\begin{subequations}\label{eq:def_U_V}
\begin{align}
	U_{i,z;j,w} &:= \sum_{w'=0}^{K} u_{i,z;j,w'} \int_{y_{j}}^{y_{j,w}}   \varphi_{j,w'}(y')   dy', \\
	V_{i,z;j,w} &:= \sum_{z'=0}^{K} v_{i,z';j,w} \int_{x_{i}}^{x_{i,z}}  \varphi_{i,z'}(x')   dx', \qquad z,w = 1, \ldots, K,
\end{align}
\end{subequations}
and pass from $U$ and $V$ to $u$ and $v$ with the matrix multiplications
\begin{align}
 	U &= \id_x \otimes I_y u &  V &= I_x \otimes \id_y v
\end{align}
with the integrator $I_x$ defined by
\begin{align}
 (I_x\otimes \id_y v)_{i,z;j,w} = \sum_{z'=0}^{K} v_{i,z';j,w} \int_{x_{i}}^{x_{i,z}}  \varphi_{i,z'}(x')   dx'. \label{eq:defintegrator}
\end{align}
Notice that doing so, we are implicitly defining $U(x,y):= \sum_{\alpha,\beta} \varphi_{\alpha;\beta}(x,y) U_{\alpha;\beta} \in V_h^K$, while formally $\hat U$ in \eqref{eq:def_U_hat} was belonging to a different functional space with higher degree of polynomials in $y$ direction. This step is a projection onto the space $V_h$.
The two polynomials, nevertheless, coincide on the degrees of freedom, i.e.,
\begin{equation}
	\hat U(x_{\alpha},y_{\beta}) = U_{\alpha;\beta} = U(x_\alpha;y_\beta).
\end{equation}

%

\begin{example}
In the case of $\mathbb Q^1$ FEM ($K=1$) one finds
\begin{subequations}
\begin{align}
	(I_x v)_{i,1;j,w} 
        &=  v_{i,0;j,w} \int_{x_{i,0}}^{x_{i+1,0}}  \varphi_{i,0}(x')   dx' + v_{i,1;j,w} \int_{x_{i,0}}^{x_{i,1}}  \varphi_{i,1}(x')   dx' = \Delta x \frac{ v_{i,0;j,w}  + v_{i+1,0;j,w}}{2}
\end{align}
i.e.
\begin{align}
 \mathbb F_{t_x}(I_x) = \Delta x \frac{t_x+1}{2t_x} = \Delta x \left( \frac{t_x+1}{2(t_x-1)} - \frac{t_x+1}{2t_x(t_x-1)}  \right),
\end{align}
\end{subequations}
which is similar to, but not exactly, the factor $\frac{t_x+1}{2(t_x-1)}$ used in Equation \eqref{eq:integratorsimpleP1}.
\end{example}

\subsection{Nodal projection, nodal consistency, and super-convergence}\label{sec:well_prepared_ic}
We have now a new definition of the discrete divergence, and we have shown that the desired physical
equilibria are part of  its kernel. 
We now set out to study the following two questions:
\begin{itemize}
 \item Given an element of the space of discrete divergence-free
solutions of Propositions \ref{th:div_kernel} and \ref{th:div_kernel-phi},  what is   its formal consistency with respect to exact analytical solutions?
 \item Given a divergence free vector field, how to devise a projection  onto the space of discrete divergence free solutions?
\end{itemize} 
The first question is  covered by the following

\begin{proposition}[GFq divergence: consistency estimate]\label{th:div-consistency}Consider a $C^{P}(\Omega)$ solenoidal vector field $(u_e,v_e)$ 
with $P\ge 1$, such that the solenoidal condition $\partial_x u_e(x,y)=-\partial_y v_e(x,y)$ is true in every point and in particular at all collocation points.
Given an ODE $U'(t)=F(U,t)$, let the integrators
$$
U_p-U_0 = (I_xF)_p\;, \quad U_p-U_0 =  (I_yF)_p\;, 
$$
be exact when $F$ is a polynomial of degree  $M$.  Then,  for $P\ge M$,  the global flux divergence \eqref{eq:div-GF} admits exact
discrete kernels    verifying Propositions \ref{th:div_kernel} or \ref{th:div_kernel-phi},  and such that $u=u_e$ and $v=v_e$ on $\partial\Omega_h$  
which also verify the consistency estimates $u=u_e + \mathcal{O}(h^{M})$, and  $v=v_e + \mathcal{O}(h^{M})$.
\end{proposition}
\begin{proof}Since $\partial_xu_e +\partial_y v_e=0$ is true pointwise and in particular  at all collocation points $(x_\alpha,y_\beta)$, we can
 remove from the expression of $\mathrm{DIV}$ operators applied to pointwise values of the derivatives
   $\partial_xu_e(x_\alpha,y_\beta)$ and $\partial_yv_e(x_\alpha,y_\beta)$. In particular note that
   $$
   \partial_xu_e(x_\alpha,y_\beta)+\partial_yv_e(x_\alpha,y_\beta) =0\Rightarrow (\partial_xu_e)_h + (\partial_yv_e)_h =0
   $$
 We thus  start from  \eqref{eq:div-GF} and add or remove interpolated values of the above zero divergence tested against all $\varphi^x\varphi^y \in V_h$
 \begin{equation}
 	\begin{split}
 \mathrm{DIV} \vec{v}=& (D_x \otimes D_y I_y)( u_h  + (I_x\otimes \id_y)((\partial_xu_e)_h + (\partial_yv_e)_h)) + \\&
  (D_x I_x \otimes D_y) (  v_h +(\id_x \otimes I_y)((\partial_xu_e)_h+(\partial_yv_e)_h)) .
 \end{split}
 \end{equation}
Simple manipulations show that the previous expression is equivalent to 
\begin{equation}\label{eq:proj0}
\begin{split}
	\mathrm{DIV} \vec{v}=& (D_x \otimes D_y I_y)( u_h  +(I_x\otimes \id_y)(\partial_yv_e)_h)+\\
	& 
	(D_x I_x\otimes D_y )(  v_h +(\id_x \otimes I_y)(\partial_xu_e)_h ) +(D_xI_x \otimes D_y I_y)\cancelto{0}{((\partial_x u_e)_h+(\partial_yv_e)_h)}.
\end{split}
\end{equation}
Since $u=u_e$ and $v=v_e$ on  $\partial\Omega$,  we consider discrete states  defined by marching along
gridlines using the ODE integrator defined by $I_x$ and $I_y$ by integrating for every $x_\alpha$ and every $y_\beta$ the ODEs implicitly appearing in \eqref{eq:proj0}, i.e.,
$$
\dfrac{du(x,y_\beta)}{dx} = - \partial_yv_e(x,y_\beta)\;,\quad 
\dfrac{dv(x_\alpha,y)}{dy} = - \partial_xu_e(x_\alpha,y).
$$
By construction the resulting nodal values  $u(x_\alpha,y_\beta)$ and $v(x_\alpha,y_\beta)$ verify
\begin{equation}\label{eq:proj1}
u_h+ (I_x\otimes \id_y)(\partial_y v_e)_h = \textrm{c}_x(y)\;,\quad
v_h+ (\id_x \otimes I_y)(\partial_xu_e)_h = \textrm{c}_y(x) 
\end{equation}
and by virtue of \eqref{eq:proj0} are thus   exact discrete solutions of  $\mathrm{DIV} \vec{v}=0$. 
Moreover, if integration is started from the boundary points $(x_0,y_\beta)$ and $(x_\alpha,y_0)$,
the hypotheses on the exactness of integration tables  and on the regularity of $(u_e,v_e)$ lead to 
the  local nodal consistency estimates
\begin{equation*}
\begin{split}
u(x_\alpha,y_\beta) =  u_e(x_0,y_\beta) - \int_{x_0}^{x_s}\partial_y v_e(x,y_\beta)dx + \mathcal{O}(h^{M+1})=u_e(x_\alpha, y_\beta)+ \mathcal{O}(h^{M+1})\;, \\
v(x_\alpha,y_\beta) =  v_e(x_\alpha,y_0) - \int_{y_0}^{y_\beta}\partial_xu_e(x_\alpha, y)dy + \mathcal{O}(h^{M+1})=v_e(x_\alpha, y_\beta)+ \mathcal{O}(h^{M+1}).
\end{split}
\end{equation*}
The global consistency is obtained  classically by considering the space marching on the whole domain, on a number of cells of order $h^{-1}$ which  
leads to the sough $h^M$ estimate.
\end{proof}

Concerning the initialization, the proof of Theorem \ref{th:div_kernel-phi}, and in particular \eqref{eq:div-residual}, provides 
an idea how to  construct discrete projections  on the kernel of \eqref{eq:div-GF}. In particular,  we define hereafter the
following quadrature-based projection.

\begin{definition}[Line-by-line quadrature  projection]\label{def:quad-proj} Let $(u_e,v_e)$  be a smooth enough   vector field. 
Let $u(0,y_\beta) =  u_e(0,y_\beta)$ and $v(x_\alpha,0) =  v_e(x_\alpha,0)$   on the bottom and left of the domain $x=0$ and $y=0$. Given these values, we define recursively  over line/row elements  the  $\vec v$ fulfilling 
\begin{equation}\label{eq:proj-q}
\begin{split}
[I_y^{E^y_{j}}u^{E^y_{j}}(x_{i,s})]_p:=\int_{y_{j,0}}^{y_{j,p}} u_e(x_{i,s},y)dy\;,\quad  
[I_x^{E^x_{i}}v^{E^x_{i}}(y_{j,p})]_s:=\int_{x_{i,0}}^{x_{i,s}} v_e(x,y_{j,p})dx 
\end{split}
\end{equation}
with  $I_x^{E^x_{i}}$ and $I_y^{E^y_{j}}$ the local restriction of the integration tables, and with  local initial conditions on each
element $ u_h(x_s,y_{j,0})=u_h(x_s,y_{j-1,K})$ and  $ v_h(x_{i,0},y_p)=v_h(x_{i-1,K},y_{p})$.
\end{definition} 

We can immediately prove the following

\begin{proposition}[Line-by-line quadrature  projection of solenoidal data]\label{prop:ll-rr-poj}
Let  $(u_e,v_e)$  be a given smooth enough solenoidal field, 
 if the quadrature of the components of $(u_e,v_e)$ in \eqref{eq:proj-q} is of order  $M_q$ then
 the line by line/row by row quadrature  projection is equivalent to \eqref{eq:proj1}  within   $\max(h^{M_q},h^{M})$,
and it is in the kernel of \eqref{eq:div-GF} for  exact integration of the right hand sides in \eqref{eq:proj-q}.
Moreover,  the projected data has a pointwise consistency  w.r.t. $(u_e,v_e)$
 of order $\max(h^{M_q},h^{M})$.
\end{proposition}
\begin{proof}
We prove the result for the $u$ component, the proof for the $v$ component is similar.
We can write that
\begin{equation*} 
\begin{split}
\int_{y_{j,0}}^{y_{j,p}} u_h(x_{i,s},y)dy =& \int_{y_{j,0}}^{y_{j,p}} u_e(x_{i,s},y)dy + \mathcal{O}(h^{M_q+1}) \\= &\int_{y_{j,0}}^{y_{j,p}} u_e(x_{i,0},y)dy -
\int_{y_{j,0}}^{y_{j,p}}\int_{x_{i,0}}^{x_{i,s}}\partial_yu_e(x,y)dx\,dy + \mathcal{O}(h^{M_q+1})\\
= &\int_{y_{j,0}}^{y_{j,p}} u_e(x_{i,0},y)dy -
\int_{y_{j,0}}^{y_{j,p}}  [I_x^{E_i^x}(\partial_y u_e)_h^{E_i^x} (y)]_s\,dy + \mathcal{O}(h^{M_q+1})+ \mathcal{O}(h^{M+1})
\end{split}
\end{equation*}
having used the hypotheses of the local initial conditions.  The latter is equivalent to  
$$
u_h(x_{i,s},y_{j,p}) + (I_x\otimes \id_y)(\partial_y v_e)_h = u_h(x_{i,0},y_{j,p}) + \mathcal{O}(h^{M_q+1})+ \mathcal{O}(h^{M+1})
$$
which shows that the equivalence with \eqref{eq:proj-q} within $\max(h^{M_q},h^{M})$.  
Using the solenoidal condition on $(u_e,v_e)$, one can also check now that by definition
\begin{equation*} 
\begin{split}
[\Phi^E]_{i,s;j,p} =& \int_{y_0}^{y_p}( u_h(x_s,y) - u_h(x_0,y))dy +  \int_{x_0}^{x_s}( v_h(x,y_p) - v_h(x,y_0))dx\\ =&
 \int_{y_0}^{y_p}( u_e(x_s,y) - u_e(x_0,y))dy +  \int_{x_0}^{x_s}( v_e(x,y_p) - v_e(x,y_0))dx +\mathcal{O}(h^{M_q+1})=\mathcal{O}(h^{M_q+1})
\end{split}
\end{equation*}
From  Proposition \ref{th:div_kernel-phi}  for exact integration   the projected data is in the kernel of \eqref{eq:div-GF}.
\end{proof}

The above directional initialization introduces some apparent dependence  on the initial integration point and direction of marching.
However,   the scheme  is in reality  symmetric with respect to the above choices.
This can be seen from   the following property.

\begin{proposition}[Reversibility of global flux quadrature SEM\label{prop:symmetric-GFq}] The  projection   operator $D_xI_x   $ obtained
with definitions  \eqref{eq:definition_integrator} is independent on the orientation of integration.
\end{proposition}
\begin{proof}
The reversed table $\tilde{I}_{x}$ verifies
$$
(I_x)_{i,j;i,l}= \int_{x_{i,0}}^{x_{i,l}}\varphi_{i,l}(x)dx= \int_{x_{i,0}}^{x_{i,K}}\varphi_{i,l}(x)dx +\int_{x_{i,K}}^{x_{i,j}}\varphi_{i,l}(x)dx = \int_{x_{i,0}}^{x_{i,K}}\varphi_l(x)dx+(\tilde{I}_x)_{i,j;i,l},
$$
with $(\tilde{I}_x)_{i,j;i,l}:=-\int_{x_{i,j}}^{x_{i,K}}\varphi_{i,l}(x)dx$.
This implies $ I_x{S} = \overline{{S}} + \tilde{I}_x{S}  $, with $ \overline{{S}}$   containing constant entries given by the average of the source. 
As a consequence  $D_xI_x{S}=D_x\tilde{I}_x{S}$.
\end{proof}
The above property shows that locally both directions of integration provide the same global flux quadrature scheme at the end. 
In practice, it is the boundary conditions that will define the actual steady solution of the scheme.
The line by line/row by row quadrature  projection is in practice simple to implement but unless corrected with several sweeps
in different vertical/horizontal directions, or of some combinations of the projection in different directions,
it is affected by the accumulation of the error along the mesh, as any ODE integrator. To avoid this drawback we have 
considered the direct optimization based  projection defined below. 

\begin{definition}[Optimization based projection]\label{def:opt-proj} 
Consider  initial data obtained by nodally sampling a given solenoidal vector field: $(u_e(x_s,y_p),v_e(x_s,y_p))$. 
The optimization based projection consists in  looking for perturbed nodal data $(\tilde{u}(x_s,y_p),\, \tilde{v}(x_s,y_p))$ whose
error w.r.t. the initial sample data is minimized, under the constraint that the discrete divergence should vanish:
\begin{equation}
	(\tilde u, \tilde v) = \argmin_{u,v \in V_h\,:\, {\normalfont \text{DIV}  \vec{v}} =0} 
		||u-u_e||_2^2+ ||v-v_e||_2^2.
\end{equation}
\end{definition} 
The above definition requires solving a  linearly constraint optimization problem with a very simple quadratic functional to be minimized. 
This solution can be obtained e.g. using a \href{https://docs.scipy.org/doc/scipy/tutorial/optimize.html\#id36}{Trust-Region Constrained Algorithm}\footnote{an open source implementation can be found in \texttt{scipy.optimize}}. For the above method we cannot prove
any consistency estimate, but in practice  we obtain data with the nodal consistency of 
Theorem~\ref{th:div-consistency} and Proposition~\ref{prop:ll-rr-poj}  within 2 iterations of the algorithm.\\

When initializing the solution with a given  solenoidal vector field, we thus have three possibilities:
sampling  at collocation points;  line by line/row by row quadrature  projection;
optimization based projection. These will be  evaluated and compared in detail  in the results section.

\section{GFq based grad-div compatible stabilization}

A high-order approximation of the divergence is not enough, because the system \eqref{eq:acoustic} under consideration is hyperbolic and stabilization is required. It was shown in Section~\ref{sec:FDstructure_preserving} that appropriate stabilization must be used in order to preserve the stationary states. We integrate in this section the global Flux technique into the SUPG and OSS stabilizations.

\subsection{SUPG stabilization with GFq} \label{ssec:globalfluxsupg}


We  construct the GFq  variant of the SUPG method 
by evaluating the integrals 
\begin{equation}\label{eq:supg-gfq1}
\begin{split}
\int  \varphi_h  ( \partial_t u_h + \partial_x p_h ) dx  +& \int \alpha \, h  \partial_x \varphi  ( \partial_t p_h +\partial_{xy}(U_h+V_h)) dx =0\\
\int  \varphi_h  ( \partial_t v_h + \partial_y p_h ) dx  +& \int \alpha \, h  \partial_y \varphi  ( \partial_t p_h +\partial_{xy}(U_h+V_h) ) dx =0\\
\int  \varphi_h  ( \partial_t p_h + \partial_{xy}(U_h+V_h)  dx  +& \int \alpha \, h  \partial_x \varphi  (  \partial_t u_h + \partial_x p_h ) dx  + \int \alpha \, h  \partial_y \varphi_h  (  \partial_t v_h + \partial_y p_h ) dx  =0\\
\end{split}
\end{equation}
We then use the definitions \eqref{eq:def_U_V} of the nodal values of $U_h$ and $V_h$ and end up with  
the modified version of the SUPG scheme:
%
%
\begin{align}\label{eq:supg_GF}
(\mathcal{A}_C + \mathcal{A}_{\textrm{SU}})\dfrac{d\mathbf{q}}{dt} + \mathcal{E}_{\textrm{C-GFq}}\mathbf{q} + \mathcal{E}_{\textrm{SU-GFq}} \mathbf{q} =0
\end{align}
with $\mathcal{A}_C$ and  $\mathcal{A}_{\textrm{SU}}$ the same as in Section \ref{sec:supgstandard} and $\mathcal E_\textrm{SUPG-GFq} = \mathcal{E}_{\textrm{C-GFq}} + \mathcal{E}_{\textrm{SU-GFq}}$ with
\begin{subequations}\label{eq:evolutionmatrixsupggf_all}
\begin{align}
  \mathcal{E}_{\textrm{C-GFq}} = \begin{pmatrix}
	0 & 0 & D_x \otimes M_y \\ 0 & 0 & M_x \otimes D_y  \\ D_x \otimes (D_y I_y) & (D_xI_x) \otimes D_y &  0  \end{pmatrix}
\end{align}
and
\begin{align}
  \mathcal{E}_{\textrm{SU-GFq}} = \alpha\, h \begin{pmatrix}
   \factorTmp D^x_x \otimes (D_yI_y) &   \factorTmp (D^x_x I_x) \otimes D_y & 0\\    \factorTmpy D_x \otimes (D^y_y I_y) &   \factorTmpy (D_xI_x) \otimes D^y_y &0 \\ 0 & 0 &  \!\!\!\!\!\!\! \factorTmp D^x_x \otimes M_y  +   \factorTmpy M_x \otimes D^y_y  \end{pmatrix} \label{eq:evolutionmatrixsupggf}  
\end{align}
\end{subequations}

The form of the above operators allows to prove the following result.

\begin{proposition}[Equilibria of SUPG-GFq]\label{th:SUPG_equilibrium} Any state with constant pressure $p$
and velocities in  the kernel of the global flux divergence operator, as 
characterized by Proposition \ref{th:div_kernel} or equivalently Proposition \ref{th:div_kernel-phi}, is a steady equilibrium
of the GFq based  SUPG.
\end{proposition}
\begin{proof}
	If \eqref{eq:discreteUplusV} is true, then both the divergence and the stabilization terms in the $u$ and $v$ equations vanish:
$$
	 D_x \otimes (D_y I_y) u + (D_x I_x) \otimes D_y v = (D_x \otimes D_y) ( \id_x \otimes I_y u +  I_x \otimes \id_y v ) = 
	 (D_x \otimes D_y) ( U+V)= 0
$$
	and also
	\begin{align*}
		\alpha \,h \factorTmp \Big ( D^x_x \otimes (D_y I_y) u +  (D^x_x I_x) \otimes D_y v\Big ) 
		&=\alpha \,h \factorTmp (D^x_x \otimes D_y )\Big ( \id_x \otimes I_y u +   I_x \otimes \id_y v\Big )  \\
		&= \alpha \,h \factorTmp D^x_x \otimes D_y ( U+V)= 0
	\end{align*}
	and similarly for the $v$ equation.
\end{proof}
Using the characteristic polynomial / Fourier representation of the scheme, as done in Section~\ref{sec:supgstandard}, one immediately confirms for the case of $\mathbb Q^1$ FEM that, independently on further choices
\begin{align}
 \det \mathbb F_{t_x,t_y}(\mathcal E) = 0.
\end{align}
and that the right kernel of $\mathbb F_{t_x,t_y}(\mathcal E) $ (related stationary states) is parallel to
\begin{align}\label{eq:right_kernel_SUPG}
 (-\mathbb F_{t_x}(I_x), \mathbb F_{t_y}(I_y), 0)^\text{T}.
\end{align}

The above property shows one of the key differences with respect to the scheme of Section~\ref{sec:supgstandard}:
the kernel of the stabilization includes that of the divergence, in particular it contains equilibria which have physical meaning,
and are characterized by Theorem~\ref{th:div_kernel} and Proposition~\ref{th:div_kernel-phi}. 
The kernel of the divergence, however, may also contain non-physical modes. This aspect is discussed in more detail in Section~\ref{sec:kernels}.

\subsection{GFq stabilization operators: OSS} \label{ssec:globalfluxoss}
Proceeding in a similar manner we construct a GFq version of the OSS stabilization, whose discrete equations are obtained from, $\forall \phi_h \in V^K_h,$ 
\begin{equation}\label{eq:OSS-gfq1}
\begin{split}
\int  \varphi_h  ( \partial_t u_h + \partial_x p_h ) dx  +& \int \alpha \, h \partial_x \varphi_h(\partial_{xy}(U_h+V_h) - w_h^{\nabla \cdot \vec{u}} ) dx =0,\\
\int  \varphi_h  ( \partial_t v_h + \partial_y p_h ) dx  +& \int \alpha \, h  \partial_y \varphi_h(\partial_{xy}(U_h+V_h) - w_h^{\nabla \cdot \vec{u}} ) dx =0,\\
\int  \varphi_h  ( \partial_t p_h + \partial_{xy}(U_h+V_h) ) dx  +& \int \alpha \, h  \partial_x \varphi_h   (\partial_x p_h - w_h^{p_x} )dx  + \int \alpha \, h  \partial_y \varphi_h  (\partial_y p - w^{p_y} ) dx  =0,  \\
\end{split}
\end{equation}
with the projections $w_h^{\nabla \cdot \vec{u}},\, w^p_x$ and $w^p_y$ defined by
\begin{equation}\label{eq:OSS_projections_gfq}
	\begin{cases}
		\int  \varphi_h \left(\partial_{xy}(U_h+V_h) - w_h^{\nabla\cdot \vec{u} } \right)dx=0,\\
		\int  \varphi_h \left( \partial_x p_h - w_h^{p_x} \right)dx=0,\\
		\int  \varphi_h \left( \partial_y p_h - w_h^{p_y} \right)dx=0,
	\end{cases}
\end{equation}
With the notation of Sections~\ref{sec:ossstandard} and~\ref{ssec:globalfluxsupg}, we can show that 
the evaluation of the above integrals leads to the following stabilization terms for the velocity equations (compare with the the OSS ones in~\eqref{eqs:OSS_stabs})
	\begin{equation}
		\label{eq:OSS_GF_uv}
		\begin{split}
			s^u = &	
			\alpha h \left\{ (Z_x\otimes D_yI_y) u +(Z_x I_x \otimes {D}_y) v \right\},\\[5pt]
			s^v =&\alpha h \left\{ (D_x \otimes Z_yI_y)u + (D_xI_x \otimes Z_y)v\right\}.
		\end{split}
	\end{equation}
The pressure stabilization is identical to the standard case. This leads to the final GFq form of the OSS stabilized scheme:
\begin{subequations}  
\begin{align}
\mathcal{A}_C \dfrac{d\mathbf{q}}{dt} + \mathcal{E}_{\textrm{C-GFq}}\mathbf{q} + \mathcal{E}_{\textrm{OSS-GFq}} \mathbf{q} =0
\end{align}
where
	\begin{align}
		\mathcal{E}_{\textrm{OSS-GFq}} := \alpha h \begin{pmatrix}
			Z_x \otimes D_yI_y &Z_xI_x \otimes {D}_y &0 \\
			{D}_x \otimes Z_yI_y & D_xI_x\otimes Z_y&0\\
			0&0&Z_x \otimes M_y + M_x\otimes Z_y
		\end{pmatrix} .\label{eq:ossgfqE}
	\end{align}
As for SUPG with GFq, we can readily characterize the steady states  of this method.
\end{subequations}

\begin{proposition}[Known  equilibria of OSS-GFq]\label{th:OSS_equilibrium}
Any state with constant pressure $p$
and velocities in  the kernel of the global flux divergence operator, as 
characterized by Theorem \ref{th:div_kernel} or equivalently Proposition \ref{th:div_kernel-phi}, is a steady equilibrium
of the GFq based  OSS method.
\end{proposition}
\begin{proof}
The only thing we need to check is that the stabilization  terms \eqref{eq:OSS_GF_uv} vanish. 
This is a trivial consequence of the hypotheses made, which imply that not only all the terms in the pressure stabilization vanish,
but also  $(D_x\otimes D_y )(U_h+V_h)$ and $w_h^{\nabla\cdot \vec{u}}$.
\end{proof}
For $\mathbb Q^1$, the right kernel is the same as for the SUPG~\eqref{eq:right_kernel_SUPG}.

Again, the above property is a  major    difference  with respect to the scheme of Section~\ref{sec:ossstandard}:
the stabilization shares with the non-stabilized scheme its kernel of equilibria,
characterized  by Theorem~\ref{th:div_kernel} and  Proposition~\ref{th:div_kernel-phi}. 
As already remarked for the SUPG, these may not be just the physically relevant ones, but the kernel might contain non-physical modes. 
This aspect is discussed in more detail in Section~\ref{sec:kernels}.\\


\subsection{Kernel of derivative operators}\label{sec:kernels}
Elements of the kernel of the GFq divergence operator are also contained in the kernel of the full discretization schemes proposed, including the stabilization, and that among them one can identify representatives of continuous stationary states. 
However, in general, other steady states, i.e. spurious modes, may exist in the kernel of the divergence operator.
In this case, it is essential that such unphysical modes are not in the kernel of the stabilization,  so that they will be dissipated, if present.

One of the interesting features of the tensor based GFq method is that   everything boils down to studying the $(U+V)$ variable instead of the two variables $u$ and $v$ independently, and information on one-dimensional kernels can be very easily applied to the multi-dimensional case.
This section is devoted to the study of the impact of the  SUPG and OSS stabilization on spurious modes contained in the kernel of the discrete divergence.

\subsubsection{One dimensional kernels}

Before dealing with two dimensions consider one-dimensional operators.
To do so, let us introduce another simplified notation for the following FEM spaces
\begin{subequations}
\begin{align}
	V^K_{\Delta x}(\Omega^x_{\Delta x}) &= \left\{ q \in \mathcal C^0(\Omega^x_{\Delta x}) : q|_{E} \in \mathbb P^K(E), \, \forall E \in \Omega^x_{\Delta x}  \right\},\\
	V_{\Delta x,0}^K(\Omega^x_{\Delta x}) &= \left\{ q \in \mathcal C^0(\Omega^x_{\Delta x}) : q|_{E} \in \mathbb P^K(E), \, \forall E \in \Omega^x_{\Delta x} \text{ and } q(x)=0\, \forall x \in \partial \Omega^x_{\Delta x} \right\},\\
	V_{\Delta x,b}^{K-1}(\Omega^x_{\Delta x}) &= \left\{ q \in L^2(\Omega^x_{\Delta x}) : q|_{E} \in \mathbb P^{K-1}(E), \, \forall E \in \Omega^x_{\Delta x}  \right\}.
\end{align}
\end{subequations}

In order to study the derivative operators, we do not impose any boundary conditions that could introduce further constraints. We are looking for the kernel of the following operators defined on $V^K_{\Delta x}(\Omega^x_{\Delta x}) \ni u_h$
\begin{align}
	D_x u := \int_{{\Omega^x_{\Delta x}}} \varphi^x(x) \partial_x u_h(x) , \qquad \forall \varphi^x \in V^K_0(\Omega^x_{\Delta x}),\\
	D^x_x u := \int_{{\Omega^x_{\Delta x}}} \partial_x \varphi^x(x) \partial_x u_h(x) , \qquad \forall \varphi^x \in V^K_0(\Omega^x_{\Delta x}). 
\end{align}
We observe that $V^K_{\Delta x}(\Omega^x_{\Delta x}) \sim \mathbb R^{N_x \times K +1}$, $V_{\Delta x,0}^K(\Omega^x_{\Delta x}) \sim\mathbb R^{N_x \times K -1}$, so the linear operators can be equivalently seen as $D_x, D_x^x: \mathbb R^{N_x\times K +1}\to\mathbb R^{N_x\times K -1}$.
Then it is clear that there is a non-empty kernel of dimension at least 2 for these operators. The trivial constant state is of course part of both kernels.
This gives us a hint to study a simplified form of these operators. Observe that $z_h :=\partial_x u_h \in V_{\Delta x,b}^{K-1}(\Omega^x_{\Delta x})$. Clearly, all the constant states $u_h\equiv u_0 \in \mathbb R$  vanish upon differentiation, as $\partial_x:V^{K}_h(\Omega^x_{\Delta x})\sim \mathbb R ^{N_x \times K+1} \to V^{K-1}_{\Delta x, b}(\Omega^x_{\Delta x})\sim \mathbb R ^{N_x \times K}$ has a one-dimensional kernel generated by $u_h\equiv 1$. This can be easily seen looking at each cell for functions that are $\partial_x u_h=0$. Indeed, being 0 polynomials in each cell, it must be that $u_h$ is constant in every cell, and, being continuous, it must be a constant over the whole domain.

Consider therefore new operators $\tilde{D}_x, \tilde{D}^x_x : V^{K-1}_{\dx, b}(\Omega^x_{\Delta x})\sim \mathbb R^{N_x \times K}\to V^{K}_{\dx,0}(\Omega^x_{\Delta x})\sim \mathbb R^{N_x \times K -1}$ defined as
\begin{align}
	\tilde{D}_x z := \int_{{\Omega^x_{\Delta x}}} \varphi(x) z_h(x) , \qquad \forall \varphi \in V^K_{\dx,0}(\Omega^x_{\Delta x}),\\
	\tilde{D}^x_x z := \int_{{\Omega^x_{\Delta x}}} \partial_x \varphi(x) z_h(x) , \qquad \forall \varphi \in V^K_{\dx,0}(\Omega^x_{\Delta x}). 
\end{align}
\begin{proposition}[Kernel characterization] \label{th:kernel_char}
	$\tilde{D}_x, \tilde{D}^x_x: \mathbb R^{N_x\times K }\to \mathbb R^{N_x \times K -1}$ have kernels of dimension one. The kernel of $\tilde{D}_x$ is generated by a function that is discontinuous at each cell interface, while the kernel of $\tilde{D}_x^x$ is generated by the constant function $1$.
\end{proposition}
The proof can be found in Appendix~\ref{app:kernel_one_dim}.
\begin{corollary}[Kernel characterization of $D_x^x,\, Z_x$ and $D_x$]
	Consider $D_x, D_x^x, Z_x: \mathbb R^{N_x\times K+1 } \to \mathbb R^{N_x \times K -1}$ defined with test functions in $V_{\dx,0}^K$ and trial functions in $V^K_\dx$. The kernel of $D_x^x$ is $\langle 1, x \rangle$, the kernel of $Z_x$ contains $\langle 1, x \rangle$, while the kernel of $D_x = \langle 1, w \rangle$ with $w$ a non-constant function with discontinuities in the first derivative at each cell interface. Moreover, the kernel of $Z_x$ does not contain $w$.
\end{corollary}
The proof is trivial for $D_x^x$ and $D_x$, while the proof for $Z_x$ can be found in Appendix~\ref{app:kernel_one_dim}.

\begin{example}[Analysis of the one-dimensional operator kernels for $\mathbb P^2$]
 In the $\mathbb P^1$ case, i.e. for Finite Differences, one usually associates spurious modes with the checkerboard mode $t_x = -1$.  One can show that $D_x$ for $\mathbb P^2$ Finite Elements has a non-trivial kernel iff $t_x = 1$. This does not mean, though, that it is checkerboard-free, because 
 \begin{align}
  \ker \mathbb F_{t_x}(D_x) = \mathrm{span}\left(  \vecc{1}{0}, \vecc{0}{1} \right )
 \end{align}
 and thus the values $q_{i,0}$ and $q_{i,1}$ can be specified independently. The opposite case is exemplified by $D_x^x$ whose kernel is also nontrivial iff $t_x = 1$, but
 \begin{align}
  \ker \mathbb F_{t_x}(D^x_x) = \mathrm{span}\left(  \vecc{1}{1} \right ),
 \end{align}
 such that it contains only uniform constants, no checkerboards.

\end{example}

\subsubsection{Global flux operator kernels in two dimensions}

When we consider the two-dimensional extension of these operators, we can still focus on similar function spaces, namely $V_h^K(\Omega_h),V^{K}_{h,0}(\Omega_h), V^{K-1}_{h,b}(\Omega_h) $, obtaining for $(U+V)_h \in V^K_h(\Omega_h)$

\begin{align}
	D_x\otimes D_y (U+V) &= \int_{\Omega_h} \varphi( x,y) \partial_x\partial_y (U+V)_h( x,y) \dd  x \dd y, \qquad \forall \varphi \in V^{K}_{h,0}(\Omega_h),\\
	D^x_x\otimes D_y (U+V) &= \int_{\Omega_h} \partial_x \varphi( x,y) \partial_x\partial_y (U+V)_h( x,y) \dd  x \dd y, \qquad \forall \varphi \in V^{K}_{h,0}(\Omega_h),\\
	D_x\otimes D^y_y (U+V) &= \int_{\Omega_h} \partial_y \varphi( x,y) \partial_x\partial_y (U+V)_h( x,y) \dd  x \dd y, \qquad \forall \varphi \in V^{K}_{h,0}(\Omega_h).
\end{align}
Again, we notice that $V^K_h(\Omega_h) \sim \mathbb R^{(N_xK+1)\times (N_yK+1)}$, $V^K_{h,0}(\Omega_h) \sim \mathbb R^{(N_xK-1)\times (N_yK-1)}$ and $V^{K-1}_{h,0}(\Omega_h) \sim \mathbb R^{(N_xK)\times (N_yK)}$, and
that the operator $\partial_x \partial_y : V^K_h(\Omega_h) \to V^{K-1}_{h,0}(\Omega_h)$ has a kernel of dimension $N_x K+N_yK+1$.
Kernel bases trivially include $\lbrace (0,\dots,0,1,0,\dots,0)\otimes (1,\dots,1)\rbrace$ (these are $N_xK+1$), and $\lbrace  (1,\dots,1)\otimes (0,\dots,0,1,0,\dots,0)\rbrace$ (these are $N_yK+1$). The spaces generated by these two sets have an intersection which is of dimension one and generated by $(1,\dots,1)\otimes(1,\dots,1)$. Hence, we have a clear description of the kernel of $\partial_x\partial_y$ which consists of the constant-by-line solutions: the interesting steady states that are automatically in the kernel of all GF operators. We can now focus on the spurious modes.

Define $\widetilde{D_x\otimes D_y},\widetilde{D^x_x\otimes D_y},\widetilde{D_x\otimes D^y_y}: V^{K-1}_{h,b}(\Omega_h) \to V^{K}_{h,0}(\Omega_h)$ by their actions on $z_h\in V^{K-1}_{h,b}(\Omega_h)$ as
\begin{subequations}
\begin{align}
	\widetilde{D_x\otimes D_y} z_h &= \int_{\Omega_h} \varphi(x,y)z_h (x,y) \dd x \dd y, \qquad \forall \varphi \in V^{K}_{h,0}(\Omega_h),\\
		\widetilde{D^x_x\otimes D_y} z_h &= \int_{\Omega_h} \partial_x \varphi(x,y) z_h(x,y) \dd x \dd y, \qquad \forall \varphi \in V^{K}_{h,0}(\Omega_h),\\
		\widetilde{D_x\otimes D^y_y} z_h &= \int_{\Omega_h} \partial_y \varphi(x,y) z_h(x,y)\dd x \dd y, \qquad \forall \varphi \in V^{K}_{h,0}(\Omega_h).
\end{align} 
\end{subequations}
Note that $\widetilde{D_x\otimes D_y} = \tilde{D}_x \otimes \tilde{D}_y$, $\widetilde{D^x_x\otimes D_y} = \tilde{D}^x_x \otimes \tilde{D}_y$ and $\widetilde{D_x\otimes D^y_y} = \tilde{D}_x \otimes \tilde{D}^y_y$, and that the kernel of the Kronecker product of two such operators is given by the space 
\begin{align}
	\ker (A_x\otimes B_y) = \ker(A_x) \otimes V_{\Delta y}^K(\Omega_{\Delta y}^y) +V_{\Delta x}^K(\Omega_{\Delta x}^x)\otimes \ker(B_y).
\end{align}
Then, using Theorem~\ref{th:kernel_char}, we can infer many things on the kernel of the two-dimensional operators. First of all we notice that 
$$D_x \otimes D_y \Phi = 0 \Longleftrightarrow \Phi = \Phi_x +\Phi_y$$ with 
$ \Phi_x \in \ker(D_x)\otimes \mathbb R^{N_yK+1}$ and 
$ \Phi_y \in \mathbb R^{N_xK+1}\otimes \ker(D_y)$. This means that
\begin{equation}
	\begin{split}
		&\Phi_x = {1}\otimes g_1 +  {w} \otimes g_2 \text{ with  }g_1,g_2\in\mathbb R^{N_yK+1},\\
		&\Phi_y = f_1\otimes {1} + f_2\otimes {w} \text{ with  }f_1,f_2\in\mathbb R^{N_xK+1}.		
	\end{split}
\end{equation}
Let us focus on $\Phi_y$, as the same holds for the $x$ component.
We can distinguish between the desired equilibria $f_1\otimes {1}$ and the checkerboard spurious modes $f_2\otimes {w}$.
Clearly, $f_1\otimes {1}$ belongs also to the kernel of the stabilization $D_x\otimes D^y_y$ and $D^x_x\otimes D_y$ as $1\in \ker(D_y)$ and $1\in\ker (D_y^y)$; so the desired equilibria are preserved.
At the same time,
$$( D_x \otimes D_y^y) (f_2\otimes {w}) \neq 0  \Longleftrightarrow f_2 \not \in \ker (D_x).$$
So, the spurious modes of $D_x\otimes D_y$ are diffused away by the stabilization operators,  except for those generated by $1\otimes w, w \otimes 1, $ and $w\otimes w$. In principle, these modes could be filtered out by boundary conditions or initial conditions.

\begin{example}[Fourier analysis of the operator kernels for $\mathbb Q^1$]
 Recall that for $\mathbb Q^1$ Finite Elements, the characteristic polynomials of the relevant operators are
 \begin{subequations}
 \begin{align}
  \ker (D_x \otimes D_y) = \left\{ \Phi : \frac{(t_x^2 -1)}{2 t_x \Delta x }\frac{(t_y^2 -1)}{2 t_y \Delta y} \Phi = 0\right \}, \\
  \ker (D_x^x \otimes D_y) = \left\{ \Phi : \frac{(t_x -1)^2}{t_x \Delta x }\frac{(t_y^2 -1)}{2 t_y \Delta y} \Phi= 0\right \}, \\
  \ker (D_x \otimes D^y_y) = \left\{ \Phi : \frac{(t_x^2 -1)}{2 t_x \Delta x }\frac{(t_y -1)^2}{t_y \Delta y} \Phi= 0 \right \}.
 \end{align}
 \end{subequations}
 The kernel of $D_x \otimes D_y$ contains functions that are either constant in one of the directions, or are checkerboards $\mathfrak c \in \mathfrak C$ in one of the directions:
 \begin{align}
  \mathfrak C_x &= \{ \phi \neq 0 : (t_x+1)\phi = 0 \}, & \mathfrak C_y &= \{ \phi \neq 0 : (t_y+1)\phi = 0 \}, & \mathfrak C &:= \mathfrak C_x \cup \mathfrak C_y.
 \end{align}
One easily can verify the inclusions
 \begin{align}
  \ker (D_x^x \otimes D_y) &\subset \ker (D_x \otimes D_y), & \ker (D_x \otimes D^y_y) &\subset \ker (D_x \otimes D_y) .
 \end{align}

However, not all the checkerboards are damped by the numerical diffusion:
\begin{align}
 \mathfrak a \in \mathfrak C_x \Rightarrow \begin{cases} a \not \in \ker (D_x^x \otimes D_y) & \text{if } t_y^2 \neq 1, \\
 	 a \in \ker (D_x^x \otimes D_y) &  \text{if } t_y^2 = 1.
 \end{cases}
\end{align}
The checkerboards not dissipated are
\begin{align}
 \{ \phi \neq 0 : (t_x + 1)(t_y^2 - 1)\phi = 0 \} \cup \{ \phi \neq 0 : (t_y + 1)(t_x^2 - 1)\phi = 0  \}.
\end{align}

\end{example}

\subsection{Discrete involutions} \label{ssec:discreteinvolutions}

Involutions of the SUPG-GFq matrix $\mathcal E_{\text{SUPG-GFq}}$ \eqref{eq:evolutionmatrixsupggf_all} can be found by using the characteristic polynomial / Fourier representation of the scheme, following  \cite{barsukow17a} as done in Section \ref{sec:supgstandard}.

The left kernel (the kernel of the transpose, related to the stationary involution) of $\mathcal E_{\textrm{SUPG-GFq}}$ in the $\mathbb Q^1$ case is parallel to
{
\begin{align}
\!\!\!\!\begin{pmatrix}
\mathbb F_{t_x}(D_x) \Big(\mathbb F_{t_y}(D_y)^2 \mathbb F_{t_x}(M_x)-\alpha^2 h^2    \mathbb F_{t_y}(D_{y}^y) \factorTmpy \Big(   \mathbb F_{t_y}(D_{y}^y) \factorTmpy
\mathbb F_{t_x}(M_x)+   \mathbb F_{t_x}(D_{x}^x) \factorTmp \mathbb F_{t_y}(M_y)\Big)\Big)
\\[10pt]
\mathbb F_{t_y}(D_y) \Big(-\mathbb F_{t_x}(D_x)^2 \mathbb F_{t_y}(M_y)+\alpha^2 h^2   \mathbb F_{t_x}(D_{x}^x) \factorTmp \Big(   \mathbb F_{t_y}(D_{y}^y)
\factorTmpy \mathbb F_{t_x}(M_x)+  \mathbb F_{t_x}(D_{x}^x) \factorTmp \mathbb F_{t_y}(M_y)\Big)\Big)
\\[10pt]
\alpha h  \left(-  \mathbb F_{t_x}(D_{x}^x) \mathbb F_{t_y}(D_y)^2 \factorTmp \mathbb F_{t_x}(M_x)+\mathbb F_{t_x}(D_x)^2  
\mathbb F_{t_y}(D_{y}^y) \factorTmpy \mathbb F_{t_y}(M_y)\right)
\end{pmatrix}^T
\end{align}}
Observe that it does not depend on $I_x, I_y$.

\begin{proposition}[Involutions of SUPG-GFq] \label{thm:involutionsupggf}
 There exist $K^2$ discrete involutions remaining stationary for any initial data subject to evolution according to SUPG-GFq with $\mathbb Q^K$ FEM.
\end{proposition}
\begin{proof}
 We use the fact that the SUPG-GFq method for linear acoustics can be written as a function of $U+V$ and $p$, instead of $u,v,p$ individually. This is obvious from Equation \eqref{eq:supg-gfq1}. In particular the $3\times 3$ (block) matrix $\mathcal E_{\mathrm{SUPG-GFq}}$ from \eqref{eq:evolutionmatrixsupggf} can be written as
 \begin{align}
  \mathcal E_{\mathrm{SUPG-GFq}} \veccc{u}{v}{p} = \underbrace{\left( \begin{array}{cc} \alpha h D_x^x \otimes D_y & D_x \otimes M_y \\ \alpha h D_x \otimes D_y^y & M_x \otimes D_y \\ D_x \otimes D_y & \alpha h D_x^x \otimes M_y + \alpha  h M_x \otimes D_y^y \end{array} \right )}_{=: \hat {\mathcal E}_{\mathrm{SUPG-GFq}}} \vecc{U+V}{p} 
 \end{align}
 Matrix $\hat {\mathcal E}_{\mathrm{SUPG-GFq}}$, for $\mathbb Q^k$ FEM, has $3 K^2$ rows and $2 K^2$ columns. Its left kernel therefore is at least $K^2$-dimensional.
\end{proof}

However, in practice it is rather difficult to explicitly express the preserved involution. In particular, one cannot simply extend the result of $\mathbb Q^1$. Defining 
\begin{align}
\mathcal K :=  D_{x}^x \otimes \factorTmp M_y +    
M_x \otimes D_{y}^y \factorTmpy,\quad
\omega := \veccc{D_x \Big(M_x \otimes D_y^2 -\alpha^2 h^2  D_{y}^y \factorTmpy \mathcal K\Big)
}{%
D_y \Big(-D_x^2 \otimes M_y+\alpha^2  h^2 D_{x}^x \factorTmp \mathcal K\Big)
}{%
\alpha h  \left(-  M_xD_{x}^x \otimes D_y^2 \factorTmp  +   D_x^2 \otimes 
M_yD_{y}^y \factorTmpy \right)  }^\text{T} \label{eq:involutionQ1}
\end{align}
we find 
$$
(\omega \hat {\mathcal E})  = ( v_1\,\;v_2)^T
$$
with
{\footnotesize
$$
\begin{aligned}
v_1 &=  
 \alpha  h \Big\{[D_x,M_xD_x^x] \otimes D_y^3 +   D_x^3 \otimes[M_yD_y^y, D_y] + \alpha^2 h^2 \left( (D_x^x \otimes D_y) [\mathcal K, D_x \otimes D_y^y] + [D_x^x \otimes D_y, (D_x \otimes D_y^y) \mathcal K]  \right) \Big\}, \\[10pt]
v_2&=D_x[M_x,D_x] \otimes D_y[D_y,M_y] + \alpha^2 h^2  [(D_x^x \otimes D_y)\mathcal K, M_x \otimes D_y]  + \alpha^2 h^2 [D_x \otimes M_y, (D_x \otimes D_y^y)\mathcal K] .
\end{aligned}
$$}
The appearance of commutators in each term in general prevent $\omega$ from being an involution. This is due to the fact that one cannot perform the same computations with block matrices as with usual matrices if the blocks do not commute. It is only for $K=1$ that the matrices reduce to scalars and \eqref{eq:involutionQ1} is indeed the involution.

The left kernel of $\mathcal E_\textrm{OSS-GFq}$ in~\eqref{eq:ossgfqE} is, for $\mathbb Q^1$,
\begin{equation}
	\begin{pmatrix}
\mathbb F_{t_x,t_y}(D_y) \mathbb F_{t_x,t_y}(M_x)+ \alpha^2h^2 \mathcal{K}_u \\
	-\mathbb F_{t_x,t_y}(D_x) \mathbb F_{t_x,t_y}(M_y)+\alpha^2 h^2 	\mathcal{K}_v \\
	\alpha h \left(-\frac{\mathbb F_{t_x,t_y}(D_x^x) \mathbb F_{t_x,t_y}(D_y) \mathbb F_{t_x,t_y}(M_x)}{\mathbb F_{t_x,t_y}(D_x)}+\frac{\mathbb F_{t_x,t_y}(D_x)
	\mathbb F_{t_x,t_y}(D_y^y) \mathbb F_{t_x,t_y}(M_y)}{\mathbb F_{t_x,t_y}(D_y)}\right)
	\end{pmatrix},
\end{equation}
where $\mathcal{K}_u$ and $\mathcal{K}_v$ are reported in Appendix~\ref{sec:curl_inv_OSS}.
It is a consistent discretization of $\left( \partial_y, - \partial_x,0 \right)^\text{T}$.

\section{Time stepping via Deferred Correction}  \label{sec:time}

In Equations~\eqref{eqs:standard_SUPG}, we have described the classical SUPG discretization of the linear acoustic system within a time-continuous framework, while in \eqref{eq:supg_GF} we have modified the spatial discretization to obtain a GF formulation that is vorticity preserving. Similarly, in Section~\ref{ssec:globalfluxoss} we have introduced the spatial discretization of the OSS and its GF version.
The goal of this section is to introduce an arbitrarily high-order time discretization.

The deferred correction is a class of arbitrarily high-order time integrations that are based on a space--time residual formulation \cite{fox1949some,dutt2000spectral,minion2003semi}. This residual is then approximated with a certain level of accuracy that matches the order of the space--time discretization.
We start by introducing Lagrangian basis functions in time (we will use Gauss-Lobatto Lagrangian basis functions). 
Then, we define 2 time operators, following \cite{abgrall2017high}: a high-order one that corresponds to an implicit collocation RK, in this case to the Lobatto IIIA, and a low-order one that corresponds to an explicit RK, in our case the explicit Euler method. 
The implicit high-order operator will be based on the residual that we want to minimize, while the explicit one is a simple aid to set up an iterative process.

The implicit operator can be seen as a high-order FEM discretization in space and time. We focus, for simplicity, on the ODE \begin{equation}\label{eq:ODE}
	u'+F(u)=0
\end{equation} 
inside one timestep $[t_n, t_{n+1}]$, as for every one-step method. 
In this timestep, we define $M$ sub-timesteps through $M+1$ sub-timenodes $t_n=t^0<\dots <t^m < \dots <t^M=t_{n+1}$ and the Lagrangian interpolating polynomials $\gamma_r(t)$ for $r=0,\dots,M$. We then denote with the superscript $r$ the approximation of the solution $q$ at the sub-timenode $t^r$, i.e., $q^r\approx q(t^r)$.
Now, for each sub-timenode $m=1,\dots,M$, the high-order time discretization operator reads
\begin{equation}
	T^{2,m}(\underline{q}) = \frac{q^m-q^0}{\dt} + \frac{1}{\Delta t} \int_{t^0}^{t^m} F\left(  \sum_{r=0}^{M} \gamma_r(t) q^r \right)\, dt.
\end{equation}
Then, using as quadrature points the same Lagrangian subtimenodes, we have
\begin{equation}
	T^{2,m}(\underline{q}) = \frac{q^m-q^0}{\dt} + \sum_{r=0}^{M} \frac{1}{\Delta t}\int_{t^0}^{t^m}\gamma_r(t) \, dt\, F\left(    q^r \right) = \frac{q^m-q^0}{\dt} + \sum_{r=0}^{M} \theta_r^m F\left(    q^r \right),
\end{equation}
with $\theta_r^m =\frac{1}{\Delta t} \int_{t^0}^{t^m} \gamma_r (t) dt$, which are independent on $\Delta t$.
The simple explicit Euler operator, instead, will read 
\begin{equation}
	T^{1,m}(\underline{q}) = \frac{q^m-q^0}{\dt} + \sum_{r=0}^{M} \frac{1}{\Delta t}\int_{t^0}^{t^m}\gamma_r(t) \, dt\, F\left(    q^0 \right) = \frac{q^m-q^0}{\dt} + \beta^m F\left(    q^0 \right),
\end{equation}
with $\beta^m = \frac{t^m-t^0}{\Delta t}$.
Moreover, as in our case, mass matrices or more complex spatial discretizations can be included in both operators. One can further simplify the $T^1$ operator by using a lumped first-order approximation version of the mass matrix in front of the term $\frac{q^m-q^0}{\dt}$, leading to an explicit matrix-free solver for $T^1(\underline{q})=\underline{r}$. 

The DeC iterative method is then defined as 
\begin{equation}\label{eq:dec_iterations}
	\begin{cases}
		\underline{q}^{(0)} = q(t^0)=q_n,\\
		T^1(\underline{q}^{(p)}) = T^1(\underline{q}^{(p-1)})-T^2(\underline{q}^{(p-1)}),\qquad \text{for }p=1,\dots, P,\\
		q_{n+1} = q^{(P)}(t^M),
	\end{cases}
\end{equation}
with $P$ being the order of accuracy of the scheme. Notice that each iteration $p$ implies the solution of $M$ systems that are explicit and matrix--free. Hence, for each time step, in total, we need to compute around $MP$ equivalent RK stages, more precisely $M(P-1)+1$.
\begin{proposition}[Order of accuracy of DeC time integrator \cite{micalizzi2022new,han2021dec}]
	The order of accuracy of the DeC is the minimum between the number of iterations $P$ and the order of the time discretization $Q$. For Gauss--Lobatto nodes it is $\min(P,2M)$, with $M$ the number of the subtimesteps.
\end{proposition}

Now, for clarity, we discuss the SUPG/OSS spatial discretization in the matrix formulation introduced in \eqref{eq:standardsupgfindiff_compact} and in \eqref{eq:OSS_with_matrix}, respectively. We will give in Appendix~\ref{app:fully_discrete_L2} the expansion of the space-time discretization in each equation, for ease in reproducibility (only for the SUPG case).
The $T^2$ operator encompasses the SUPG/OSS residual, so, differently from the ODE case, we have to insert also the mass matrix term. 
It is defined for $m=1,\dots, M$ as
\begin{equation}\label{eq:L2SUPG}
	T^{2,m}(\underline{{q}})= \mathcal{A} \frac{{q}^m-{q}^0}{\Delta t} +  \sum_{r=0}^M \theta_r^m \mathcal{E} {q}^r.
\end{equation}
On the other side, the $T^1$ low order operator is a simplified version of $T^2$, in particular, the mass matrix is a simple lumped version of the mass matrix, i.e., $(L_x)_{\alpha , \beta} = \delta_{\alpha,\beta }\int_{\Omega_{\Delta x}^x} \varphi_{\alpha}^x(x)\dd x$, which we define as 
\begin{equation}
	\mathcal{L}:= \begin{pmatrix}
		L_x \otimes L_y & 0&0\\
		0&L_x \otimes L_y & 0\\
	0&0&	L_x \otimes L_y \\
	\end{pmatrix}.
\end{equation}
Hence, the $T^1$ operator can be defined for $m=1,\dots, M$ as
\begin{equation}
	T^{1,m}(\underline{{q}})= \mathcal{L} \frac{{q}^m-{q}^0}{\Delta t} +  \beta^m \mathcal{E} {q}^0.
\end{equation}

The update formula in \eqref{eq:dec_iterations}, after the simplification of the terms in the $T^1$ operators reads for every $p=1,\dots,P$ and every $m=1,\dots,M$ 
\begin{align}
	0=\mathcal{L} \frac{{q}^{(p),m}-{q}^{(p-1),m}}{\Delta t}  + \mathcal{A} \frac{{q}^{(p-1),m}-{q}^0}{\Delta t} +  \sum_{r=0}^M \theta_r^m \mathcal{E} {q}^{(p-1),r},
\end{align}
where ${q}^{(p),m}$ is the only unknown term and $\mathcal L$ is a diagonal matrix.

\begin{proposition}[Order of accuracy of the space--time DeC]
	The SUPG--DeC/OSS--DeC space time discretization presented above is of order  $\min(K+1,2M,P)$. 
\end{proposition}
\begin{proof}
	The proof follows the ones of \cite{abgrall2017high,micalizzi2022new} with the details on the SUPG/OSS discretization defined in \cite{michel2021spectral,michel2022spectral}.
\end{proof}
\begin{proposition}[Steady states of the DeC]
	If $\bar{{q}}$ is such that $\mathcal{E} \bar{{q}}=0$, then $\bar{{q}}$ is a steady state of the DeC, i.e., if ${q}_{n}=\bar{{q}}$ then ${q}_{n+1}=\bar{{q}}$.
\end{proposition}
\begin{proof}
	We proceed by induction on $p$, for all $m$, showing that ${q}^{(p),m}={q}_{n}=\bar{{q}}$. By definition of the DeC, for $p=0$ we set all ${q}^{(0),m}=\bar{{q}}$ and for $m=0$ we set all ${q}^{(p),0}=\bar{{q}}$.
	Then, for $p=1,\dots,P$ and $m=1,\dots,M$ we have that
	\begin{equation}
		 {q}^{(p),m} = {q}^{(p-1),m}  -\mathcal{L}^{-1}  \mathcal{A} ({q}^{(p-1),m}-{q}^0) -\Delta t \mathcal{L}^{-1} \sum_{r=0}^M \theta_r^m \mathcal{E} {q}^{(p-1),r}={q}^{(p-1),m} = \bar{{q}},
	\end{equation}
as for $p-1$ we have that $\mathcal{E} {q}^{(p-1),r}= 0$ for all $r=0,\dots,M$ and that ${q}^{(p-1),m}={q}^{0}=\bar{{q}}$ for all $m=1,\dots,M$. Hence, also ${q}_{n+1}={q}^{(P),M} = \bar{{q}}$.
\end{proof}

The consequence of this theorem is that for the SUPG-GFq and OSS-GFq discretizations $\mathcal{E}$, we know a class of steady states, hence, the DeC time integration method will preserve them. 
This allows us to easily converge towards the steady states, or to set up initial conditions that verify the condition $\mathcal{E}{q}=0$ and to preserve them.

\section{Numerical results} \label{sec:numerical}

In this section, we show the benefits of the proposed formulation through various numerical tests. In all computations  we have used Gauss--Lobatto points both for interpolation and  quadrature.
 
\subsection{Convergence analysis on a smooth oblique flow}
To show the arbitrarily high-order property of the SUPG and SUPG-GFq DeC-FEM methods of Section~\ref{sec:globalflux}, we use a smooth two-dimensional problem of an oblique wave on the square $[0,1]^2$ with periodic boundary conditions. Its analytical solution  for the linear acoustic equations \eqref{eq:acoustic} is
\begin{equation}
\begin{cases}
	u(x,y,t) = -\frac{1}{2c} \left(\cos( \alpha \xi(x,y)+ct )-\cos( \alpha \xi(x,y)-ct )\right) \cos(\theta),\\
	v(x,y,t) = -\frac{1}{2c} \left(\cos( \alpha \xi(x,y)+ct )-\cos( \alpha \xi(x,y)-ct )\right) \sin(\theta),\\
	p(x,y,t) = \frac{1}{2} \left(\cos( \alpha \xi(x,y)+ct )+\cos( \alpha \xi(x,y)-ct )\right),
\end{cases}	
\end{equation}
with $\alpha = \frac{2\pi}{\lambda\cos(\theta)} $ with $\theta = \frac{\pi}{4}$ and $\lambda = \frac14$.
We run the simulations up to $T=1$.
\begin{figure}
	\centering
		\begin{tikzpicture}
			\begin{axis}[
				xmode=log, ymode=log,
				xmin=2.6,xmax=1000,
				grid=major,
				xlabel={$N$},
				ylabel={$L^2$ Error of $u$},
				xlabel shift = 1 pt,
				ylabel shift = 1 pt,
				legend pos= north east,
				legend style={nodes={scale=0.6, transform shape}},
				width=.95\textwidth,
				height=.45\textwidth
				]
				
				\addplot[mark=square*,mark size=1.3pt,densely dashdotted, magenta]             table [y=err u, x=N, col sep=comma]{figures/Lin2D_oblique/errors_SUPG_ord2.csv};
				\addlegendentry{SUPG2}
				\addplot[mark=square*,dashed,mark size=1.3pt,magenta]  table [y=err u, x=N, col sep=comma]{figures/Lin2D_oblique/errors_SUPG_GF_ord2.csv};
				\addlegendentry{SUPG-GFq2}
				\addplot[magenta,domain=20:320]{500./x/x};
				\addlegendentry{order 2}		
				
				\addplot[mark=otimes*,mark size=1.3pt,densely dashdotted, blue]             table [y=err u, x=N, col sep=comma]{figures/Lin2D_oblique/errors_SUPG_ord3.csv};
				\addlegendentry{SUPG3}
				\addplot[mark=otimes*,dashed,mark size=1.3pt,blue]  table [y=err u, x=N, col sep=comma]{figures/Lin2D_oblique/errors_SUPG_GF_ord3.csv};
				\addlegendentry{SUPG-GFq3}
				\addplot[blue,domain=10:160]{500./x/x/x};
				\addlegendentry{order 3}		
				
				\addplot[mark=triangle*,mark size=1.3pt,densely dashdotted, violet]             table [y=err u, x=N, col sep=comma]{figures/Lin2D_oblique/errors_SUPG_ord4.csv};
				\addlegendentry{SUPG4}
				\addplot[mark=triangle*,dashed,mark size=1.3pt,violet]  table [y=err u, x=N, col sep=comma]{figures/Lin2D_oblique/errors_SUPG_GF_ord4.csv};
				\addlegendentry{SUPG-GFq4}
				\addplot[violet,domain=6:106]{500./x/x/x/x};
				\addlegendentry{order 4}			
				
				\addplot[mark=star,mark size=1.3pt,densely dashdotted, red]             table [y=err u, x=N, col sep=comma]{figures/Lin2D_oblique/errors_SUPG_ord5.csv};
				\addlegendentry{SUPG5}
				\addplot[mark=star,dashed,mark size=1.3pt,red]  table [y=err u, x=N, col sep=comma]{figures/Lin2D_oblique/errors_SUPG_GF_ord5.csv};
				\addlegendentry{SUPG-GFq5}
				\addplot[red,domain=5:80]{500./x/x/x/x/x};
				\addlegendentry{order 5}	
				
				\addplot[mark=diamond*,mark size=1.3pt,densely dashdotted,olive]             table [y=err u, x=N, col sep=comma]{figures/Lin2D_oblique/errors_SUPG_ord6.csv};
				\addlegendentry{SUPG6}
				\addplot[mark=diamond*,dashed,mark size=1.3pt,olive]  table [y=err u, x=N, col sep=comma]{figures/Lin2D_oblique/errors_SUPG_GF_ord6.csv};
				\addlegendentry{SUPG-GFq6}
				\addplot[olive,domain=4:64]{500./x/x/x/x/x/x};
				\addlegendentry{order 6}

				\addplot[mark=o,mark size=1.3pt,densely dashdotted,brown]             table [y=err u, x=N, col sep=comma]{figures/Lin2D_oblique/errors_SUPG_ord7.csv};
				\addlegendentry{SUPG7}
				\addplot[mark=o,dashed,mark size=1.3pt,brown]  table [y=err u, x=N, col sep=comma]{figures/Lin2D_oblique/errors_SUPG_GF_ord7.csv};
				\addlegendentry{SUPG-GFq7}
				\addplot[brown,domain=3:53]{500./x/x/x/x/x/x/x};
				\addlegendentry{order 7}			
			\end{axis}
	\end{tikzpicture}
	\caption{Oblique flow: convergence of $L^2$ error in $u$ w.r.t. the number of elements in $x$}\label{fig:convergence_oblique}
\end{figure}
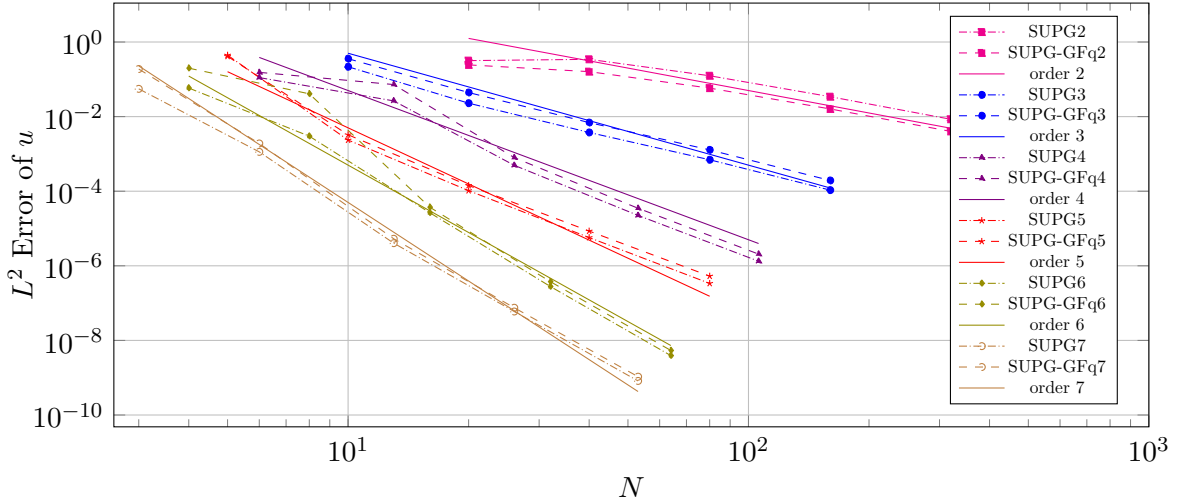

The errors are compared for the two methods in Figure~\ref{fig:convergence_oblique} (for $u$). The SUPG and SUPG--GF methods provide very similar errors and both of them converge with the expected order of accuracy. 


\subsection{Divergence-free solutions}\label{sec:div_free_vortex}
In this section, we will consider two analytical solutions that are divergence-free. Both have the velocity field of a vortex and a constant pressure. Both are be defined on the unit square with Dirichlet boundary conditions and centered in $(x_0,y_0)=(0.5,0.5)$.
The first one is a compactly supported solution in $\mathcal{C}^6$, while the second one is used for convergence purposes and is a $\mathcal{C}^\infty$ compactly supported function.

Both can be written as
\begin{equation}\label{eq:definition_any_vortex}
	\begin{cases}
		u(x,y) = f(\rho(x,y)\cdot(y-y_0)\\
		v(x,y) = -f(\rho(x,y))\cdot(x-x_0)\\
		p(x,y)=1
	\end{cases}
\end{equation}
with $\rho(x,y)=\frac{\sqrt{(x-x_0)^2+(y-y_0)^2}}{r_0}$ with $r_0=0.45$ the radius of the support. 
The first test case is defined by
\begin{equation}\label{eq:definition_C6_vortex}
	f(\rho) = \gamma (1+\cos(\pi \rho))^2,
\end{equation}
with $\gamma = \frac{12\pi\sqrt{0.981}}{r_0\sqrt{315\pi^2-2048}}$, see \cite{ricchiuto2021analytical} for the origin of these solutions.
The second is defined by
\begin{equation}\label{eq:definition_smooth_vortex}
	f(\rho) = 2\gamma e^{-\frac{1}{2(1-\rho)^2}} \sqrt{\frac{g}{r_0(1-\rho)^3}}
\end{equation}
with $g=9.81$, $\gamma = 0.2$ if $\rho<1$, else 0.

For the first vortex \eqref{eq:definition_C6_vortex}, let us consider some qualitative results on a coarse mesh. 
In Figure~\ref{fig:simul_vortex_ord2_comparison}, we compare the solution at time $T=100$ for $\mathbb Q^1$ SUPG and SUPG-GFq methods on a grid of $20\times 20$ cells. On the one hand, for long simulation times the SUPG scheme leads to vertical/horizontal artifacts that are not physical and the pressure does not converge to a constant state. On the other hand, the SUPG-GFq does not show this behavior and converges to the divergence-free state with constant pressure. In Figure~\ref{fig:simul_vortex_ord3_comparison}, we compare the pressures for $\mathbb Q^2$ schemes and observe the same behavior on a smaller scale. SUPG--GFq obtains a constant pressure up to machine precision, while SUPG has oscillations of the order of $10^{-9}$.

\begin{figure}
	\includegraphics[width=\textwidth, trim={0 0 0 30 },clip]{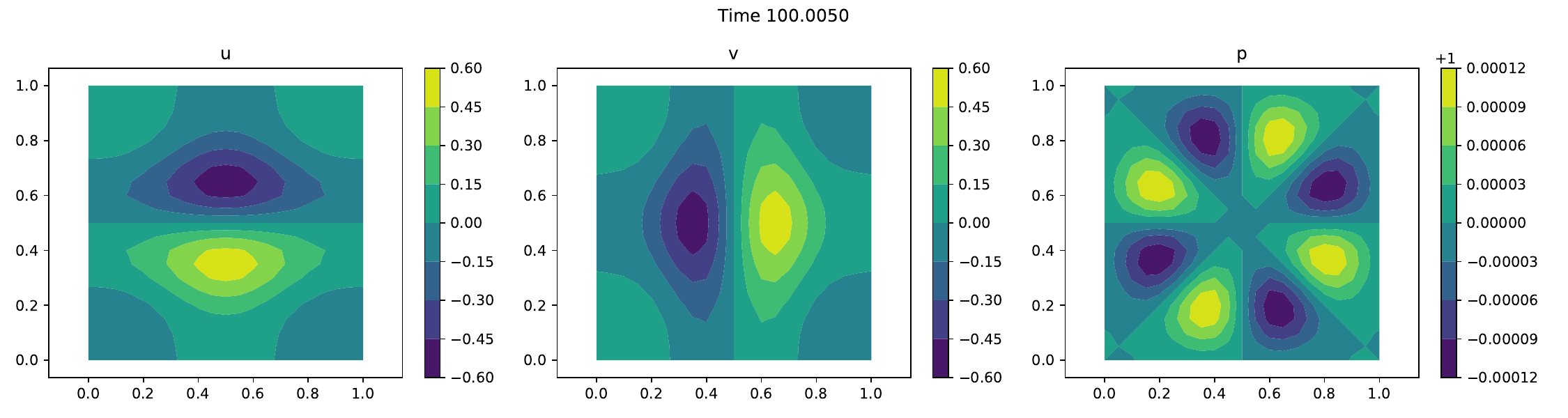}
	\includegraphics[width=\textwidth, trim={0 0 0 30 },clip]{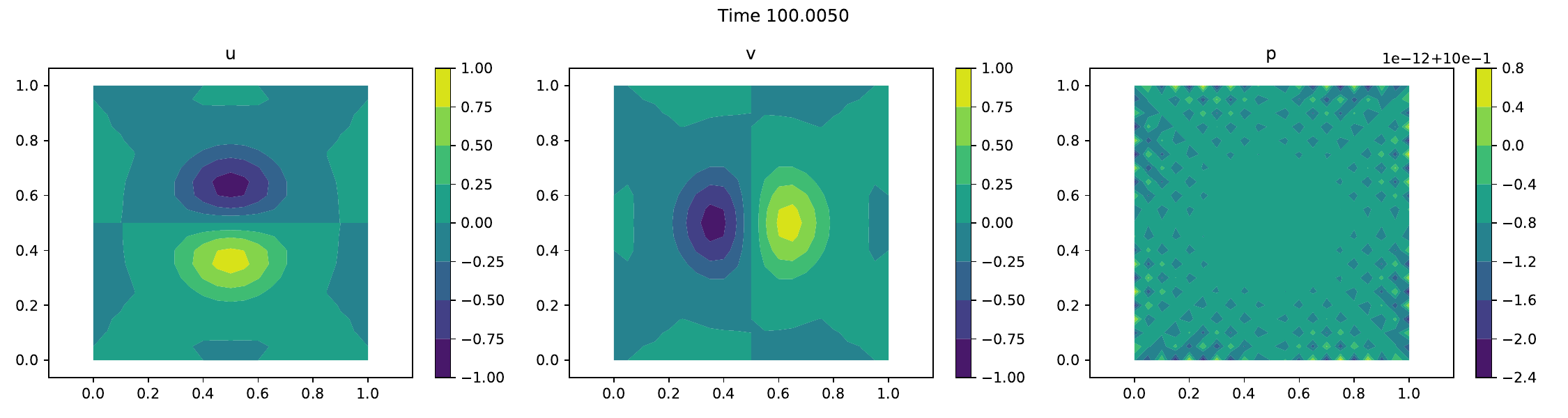}
	\caption{Simulation at time $T=100$ of the vortex \eqref{eq:definition_C6_vortex} with $20\times 20$ cells and $\mathbb P^1$ elements for the SUPG (top) and SUPG--GF (bottom) schemes}\label{fig:simul_vortex_ord2_comparison}
\end{figure}
\begin{figure}
	\centering
	\includegraphics[width=0.35\textwidth, trim={720 0 0 30 },clip]{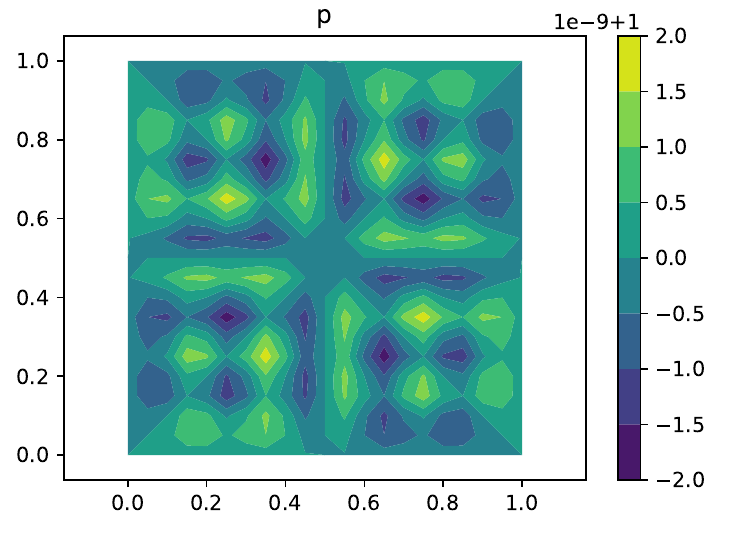}
	\includegraphics[width=0.35\textwidth, trim={720 0 0 30 },clip]{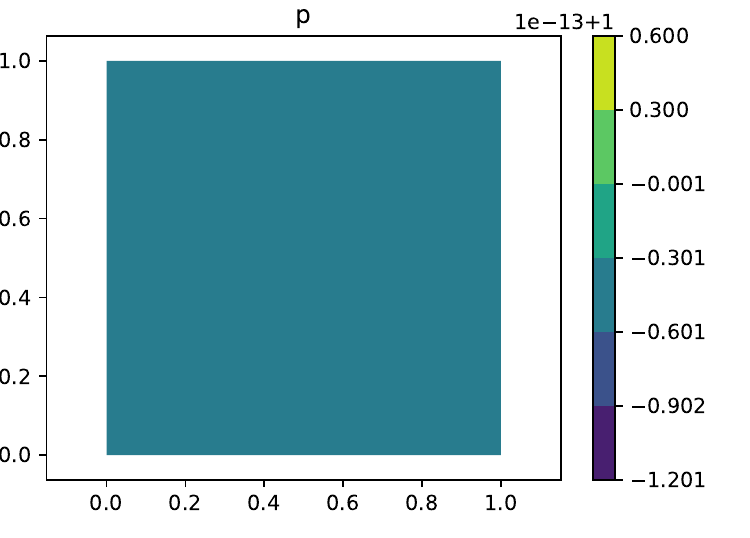}
	\caption{Simulation at time $T=100$ of the vortex \eqref{eq:definition_C6_vortex} (only $p$) with $10\times 10$ cells and $\mathbb P^2$ elements for the SUPG (left) and SUPG--GF (right) schemes}\label{fig:simul_vortex_ord3_comparison}
\end{figure}

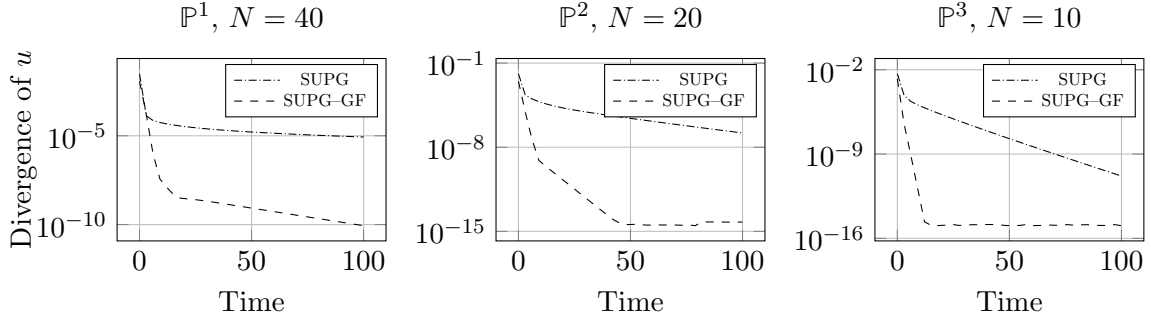
\begin{figure}
	\centering
		\begin{tikzpicture}
	\begin{axis}[
		ymode=log,
		grid=major,
		xlabel={Time},
		ylabel={Divergence of $u$},
		xlabel shift = 1 pt,
		ylabel shift = 1 pt,
		legend pos= north east,
		legend style={nodes={scale=0.6, transform shape}},
		width=.32\textwidth,
		height=.25\textwidth,
		title={$\mathbb P^1,\,N=40$}
		]			
		\addplot[densely dashdotted,black]  table [y=div, x=t, col sep=comma]{figures/LinAc2D_vortex_long/div_err_time_SUPG_ord_2_N_0040.csv};
		\addlegendentry{SUPG}
		\addplot[dashed, black]             table [y=div, x=t, col sep=comma]{figures/LinAc2D_vortex_long/div_err_time_SUPG_GF_ord_2_N_0040.csv};
		\addlegendentry{SUPG--GF}
	\end{axis}
\end{tikzpicture}
		\begin{tikzpicture}
	\begin{axis}[
		ymode=log,
		grid=major,
		xlabel={Time},
		xlabel shift = 1 pt,
		ylabel shift = 1 pt,
		legend pos= north east,
		legend style={nodes={scale=0.6, transform shape}},
		width=.32\textwidth,
		height=.25\textwidth,
		title={$\mathbb P^2,\,N=20$}
		]			
		\addplot[densely dashdotted,black]  table [y=div, x=t, col sep=comma]{figures/LinAc2D_vortex_long/div_err_time_SUPG_ord_3_N_0020.csv};
		\addlegendentry{SUPG}
		\addplot[dashed, black]             table [y=div, x=t, col sep=comma]{figures/LinAc2D_vortex_long/div_err_time_SUPG_GF_ord_3_N_0020.csv};
		\addlegendentry{SUPG--GF}
	\end{axis}
\end{tikzpicture}
		\begin{tikzpicture}
	\begin{axis}[
		ymode=log,
		grid=major,
		xlabel={Time},
		xlabel shift = 1 pt,
		ylabel shift = 1 pt,
		legend pos= north east,
		legend style={nodes={scale=0.6, transform shape}},
		width=.32\textwidth,
		height=.25\textwidth,
		title={$\mathbb P^3,\,N=10$}
		]			
		\addplot[densely dashdotted,black]  table [y=div, x=t, col sep=comma]{figures/LinAc2D_vortex_long/div_err_time_SUPG_ord_4_N_0013.csv};
		\addlegendentry{SUPG}
		\addplot[dashed, black]             table [y=div, x=t, col sep=comma]{figures/LinAc2D_vortex_long/div_err_time_SUPG_GF_ord_4_N_0013.csv};
		\addlegendentry{SUPG--GF}
	\end{axis}
\end{tikzpicture}
	\caption{Norm of the discrete divergence of $\vec{v}$ for SUPG (${D}_x \otimes M_y u + M_x\otimes {D}_y v$) and SUPG-GFq (${D}_x \otimes (D_yI_y) u + (D_x I_x)\otimes {D}_y v$) as a function of time for different orders of accuracy}\label{fig:disc_div_long_vortex}
\end{figure} 

In Figure~\ref{fig:disc_div_long_vortex}, we show the norm of the discrete divergence in time. For the SUPG-GFq schemes it decays exponentially until reaching machine precision (higher orders decay faster than low order schemes). For the SUPG schemes the decay is much slower and it depends heavily on the order of accuracy. In particular, the second order method is very inaccurate, while the fourth order scheme reaches small values of divergence at the final time.

\begin{table}[htbp]
	\centering
	\foreach \n in {2,...,6}{
		\pgfmathtruncatemacro\result{\n-1}
		\caption{Smooth vortex convergence results $\mathbb P^{\result}$: without GF (left) and with GF (right)}\label{tab:smooth_vortex_conv_\n}
		\begin{adjustbox}{max width=0.48\textwidth}
			\begin{tabular}{|c|ccc|ccc|}
				\hline
				$N$ & err $u$ & err $v$ & err $p$ & ord $u$ & ord $v$ & ord $p$ \\
				\hline
				\input{figures/LinAc2D_smooth_vortex/errors_SUPG_ord\n.tex}\\ \hline
			\end{tabular}
		\end{adjustbox}
		\begin{adjustbox}{max width=0.48\textwidth}
			\begin{tabular}{|c|ccc|ccc|}
				\hline
				$N$ & err $u$ & err $v$ & err $p$ & ord $u$ & ord $v$ & ord $p$ \\
				\hline
				\input{figures/LinAc2D_smooth_vortex/errors_SUPG_GF_ord\n.tex}\\ \hline
			\end{tabular}
		\end{adjustbox}
	}
\end{table}

\begin{figure}
	\centering
	\begin{tikzpicture}
		\begin{axis}[
			xmode=log, ymode=log,
			xmin=3,xmax=450,
			grid=major,
			xlabel={$N$},
			title={$L^2$ Error of $u$},
			xlabel shift = 1 pt,
			ylabel shift = 1 pt,
			legend pos= outer north east,
			legend style={nodes={scale=0.6, transform shape}},
			width=.5\textwidth,
			height=.4\textwidth
			]
			
			\addplot[mark=square*,mark size=1.3pt,densely dashdotted,magenta]             table [y=err u, x=N, col sep=comma]{figures/LinAc2D_smooth_vortex/errors_SUPG_ord2.csv};
			\addlegendentry{SUPG2}
			\addplot[mark=square*,dashed,mark size=1.3pt,magenta]  table [y=err u, x=N, col sep=comma]{figures/LinAc2D_smooth_vortex/errors_SUPG_GF_ord2.csv};
			\addlegendentry{SUPG-GFq2}
			\addplot[magenta,domain=70:320]{2./x/x};
			\addlegendentry{order 2}		
			
			\addplot[mark=otimes*,mark size=1.3pt,densely dashdotted,blue]             table [y=err u, x=N, col sep=comma]{figures/LinAc2D_smooth_vortex/errors_SUPG_ord3.csv};
			\addlegendentry{SUPG3}
			\addplot[mark=otimes*,dashed,mark size=1.3pt,blue]  table [y=err u, x=N, col sep=comma]{figures/LinAc2D_smooth_vortex/errors_SUPG_GF_ord3.csv};
			\addlegendentry{SUPG-GFq3}
			\addplot[blue,domain=50:160]{5./x/x/x};
			\addlegendentry{order 3}		
			
			\addplot[mark=triangle*,mark size=1.3pt,densely dashdotted,violet]             table [y=err u, x=N, col sep=comma]{figures/LinAc2D_smooth_vortex/errors_SUPG_ord4.csv};
			\addlegendentry{SUPG4}
			\addplot[mark=triangle*,dashed,mark size=1.3pt,violet]  table [y=err u, x=N, col sep=comma]{figures/LinAc2D_smooth_vortex/errors_SUPG_GF_ord4.csv};
			\addlegendentry{SUPG-GFq4}
			\addplot[violet,domain=40:160]{100./x/x/x/x};
			\addlegendentry{order 4}			
			
			\addplot[mark=star,mark size=1.3pt,densely dashdotted,red]             table [y=err u, x=N, col sep=comma]{figures/LinAc2D_smooth_vortex/errors_SUPG_ord5.csv};
			\addlegendentry{SUPG5}
			\addplot[mark=star,dashed,mark size=1.3pt,red]  table [y=err u, x=N, col sep=comma]{figures/LinAc2D_smooth_vortex/errors_SUPG_GF_ord5.csv};
			\addlegendentry{SUPG-GFq5}
			\addplot[red,domain=30:106]{400./x/x/x/x/x};
			\addlegendentry{order 5}	
			
			\addplot[mark=diamond*,mark size=1.3pt,densely dashdotted, olive]             table [y=err u, x=N, col sep=comma]{figures/LinAc2D_smooth_vortex/errors_SUPG_ord6.csv};
			\addlegendentry{SUPG6}
			\addplot[mark=diamond*,dashed,mark size=1.3pt,olive]  table [y=err u, x=N, col sep=comma]{figures/LinAc2D_smooth_vortex/errors_SUPG_GF_ord6.csv};
			\addlegendentry{SUPG-GFq6}
			\addplot[olive,domain=30:80]{1500./x/x/x/x/x/x};
			\addlegendentry{order 6}		
			\addplot[black,domain=30:64]{8000./x/x/x/x/x/x/x};
			\addlegendentry{order 7}		

		\end{axis}
	\end{tikzpicture}
	\begin{tikzpicture}
	\begin{axis}[
		xmode=log, ymode=log,
		xmin=2.6,xmax=450,
		grid=major,
		xlabel={$N$},
		title={Error ratio},
		xlabel shift = 1 pt,
		ylabel shift = 1 pt,
		width=.42\textwidth,
		height=.4\textwidth
		]
		
		\addplot[mark=square*,mark size=1.3pt,densely dashdotted,magenta] table [y=err u, x=N, col sep=comma]{figures/LinAc2D_smooth_vortex/ratios_ord2.csv};
		
		\addplot[mark=otimes*,mark size=1.3pt,densely dashdotted,blue] table [y=err u, x=N, col sep=comma]{figures/LinAc2D_smooth_vortex/ratios_ord3.csv};
		
		\addplot[mark=triangle*,mark size=1.3pt,densely dashdotted,violet]             table [y=err u, x=N, col sep=comma]{figures/LinAc2D_smooth_vortex/ratios_ord4.csv};
		
		\addplot[mark=star,mark size=1.3pt,densely dashdotted,red]             table [y=err u, x=N, col sep=comma]{figures/LinAc2D_smooth_vortex/ratios_ord5.csv};
		
		\addplot[mark=diamond*,mark size=1.3pt,densely dashdotted, olive]             table [y=err u, x=N, col sep=comma]{figures/LinAc2D_smooth_vortex/ratios_ord6.csv};
		
	\end{axis}
\end{tikzpicture}
	\caption{Smooth vortex: convergence of $L^2$ error of $u$ with respect to the number of elements in $x$ (left) and error ratios between SUPG and SUPG-GFq (right)}\label{fig:convergence_vortex}
\end{figure}
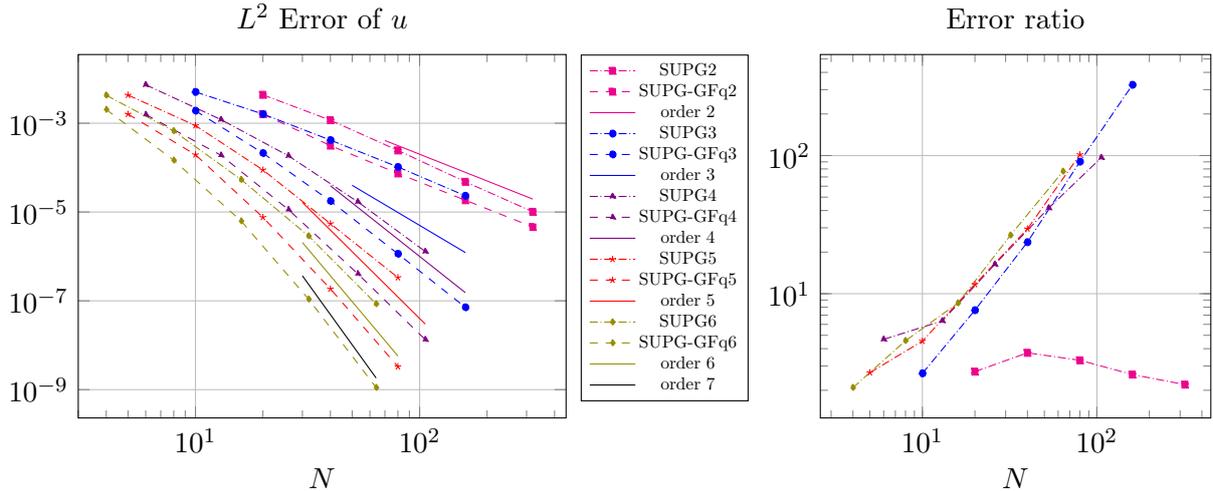

For the smooth vortex \eqref{eq:definition_smooth_vortex}, we test the convergence for arbitrarily high order at final time $T=1$. Although the solution is $\mathcal C^\infty$, the spatial derivatives of the solution are quite steep \cite{ricchiuto2021analytical}. Hence, it is not trivial to observe the right convergence rate for coarse meshes.
Nevertheless, in Tables~\ref{tab:smooth_vortex_conv_2}-\ref{tab:smooth_vortex_conv_6} we see that all SUPG-GFq solutions reach the expected order, while SUPG seems to struggle at this objective.  Moreover, the errors magnitude for the GF formulation are significantly smaller than for the classical one, as much as by two orders of magnitude. In Figure~\ref{fig:convergence_vortex}, we depict the errors for $u$ and the ratio of the SUPG errors and the SUPG-GFq ones. SUPG-GFq methods of lower order can outperform the SUPG method, see for instance how close SUPG-GFq-4 and SUPG-6 are.

\subsection{Perturbation of divergence--free solutions}
In this section, we study the behavior of the schemes when a perturbation is applied to an equilibrium state. 
Typically, one is not aware of the steady state before running the simulation, moreover, the class of steady states of \eqref{eq:acoustic} is quite rich and it is not fully determined by an external datum, like for 1D or 1D-like source-driven equilibria \cite{ciallella2022global}.

Several options are available to run such setups. 
\begin{enumerate}
	\item The first obvious possibility is to just use \textbf{analytical initial conditions} to start the simulations. This, however, might lead to a loss in accuracy in the early stages of the simulations.
	\item A second approach consists in running \textbf{long time simulations}, reaching the equilibrium solution and to add the perturbation afterwards. This approach guarantees that a discrete equilibrium is used as initial condition. This strategy is natural for simulations in which the equilibrium is not known \textit{a priori}. However, it can be expensive to run a long time simulation, just to have a short simulation of the perturbation.
	\item Finally, \textbf{discretely well-prepared initial data} can be first obtained and the perturbation added to those. We propose two approaches to reach such a goal: an optimization of the analytical initial condition constrained to the discrete div-free property and a line-by-line reconstruction of the initial conditions. The description of the two strategies is given in Section~\ref{sec:well_prepared_ic}.
\end{enumerate}
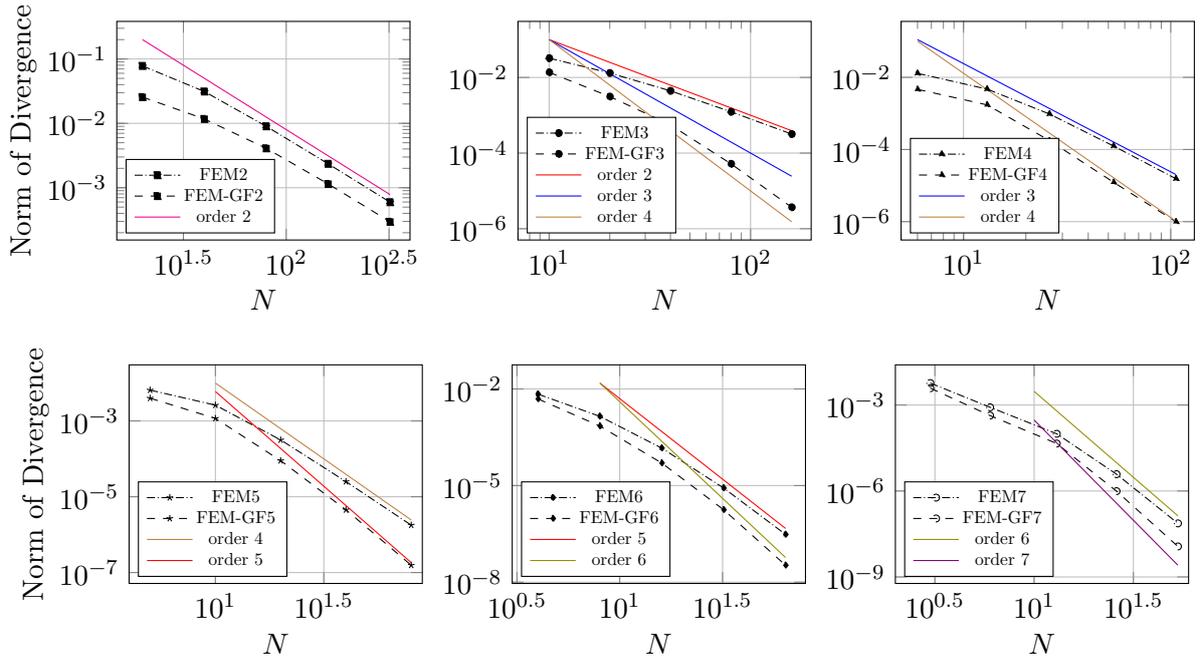
\begin{figure}
	\centering
	\begin{tikzpicture}
		\begin{axis}[
			xmode=log, ymode=log,
			xmin=15,xmax=400,
			grid=major,
			xlabel={$N$},
			ylabel={Norm of Divergence},
			xlabel shift = 1 pt,
			ylabel shift = 1 pt,
			legend pos= south west,
			legend style={nodes={scale=0.6, transform shape}},
			width=.34\textwidth,
			height=.28\textwidth
			]
			\addplot[mark=square*,mark size=1.3pt,densely dashdotted,black]             table [y=err_nbr, x=N, col sep=comma]{figures/LinAc2D_smooth_vortex/divergenceSimple_discretization_error_ord2.csv};
			\addlegendentry{FEM2}
			\addplot[mark=square*,dashed,mark size=1.3pt,black]  table [y=err_nbr, x=N, col sep=comma]{figures/LinAc2D_smooth_vortex/divergenceGF_discretization_error_ord2.csv};
			\addlegendentry{FEM-GF2}
			\addplot[magenta,domain=20:320]{80./x/x};
			\addlegendentry{order 2}
		\end{axis}
	\end{tikzpicture}
	\begin{tikzpicture}
		\begin{axis}[
			xmode=log, ymode=log,
			xmin=7,xmax=200,
			grid=major,
			xlabel={$N$},
			xlabel shift = 1 pt,
			ylabel shift = 1 pt,
			legend pos= south west,
			legend style={nodes={scale=0.6, transform shape}},
			width=.34\textwidth,
			height=.28\textwidth
			]
			\addplot[mark=otimes*,mark size=1.3pt,densely dashdotted,black]             table [y=err_nbr, x=N, col sep=comma]{figures/LinAc2D_smooth_vortex/divergenceSimple_discretization_error_ord3.csv};
			\addlegendentry{FEM3}
			\addplot[mark=otimes*,dashed,mark size=1.3pt,black]  table [y=err_nbr, x=N, col sep=comma]{figures/LinAc2D_smooth_vortex/divergenceGF_discretization_error_ord3.csv};
			\addlegendentry{FEM-GF3}
			\addplot[red,domain=10:160]{10./x/x};
			\addlegendentry{order 2}		
			\addplot[blue,domain=10:160]{100./x/x/x};
			\addlegendentry{order 3}		
			\addplot[brown,domain=10:160]{1000./x/x/x/x};
			\addlegendentry{order 4}			
		\end{axis}
	\end{tikzpicture}
	\begin{tikzpicture}
		\begin{axis}[
			xmode=log, ymode=log,
			xmin=5,xmax=130,
			grid=major,
			xlabel={$N$},
			xlabel shift = 1 pt,
			ylabel shift = 1 pt,
			legend pos= south west,
			legend style={nodes={scale=0.6, transform shape}},
			width=.34\textwidth,
			height=.28\textwidth
			]
			\addplot[mark=triangle*,mark size=1.3pt,densely dashdotted,black]             table [y=err_nbr, x=N, col sep=comma]{figures/LinAc2D_smooth_vortex/divergenceSimple_discretization_error_ord4.csv};
			\addlegendentry{FEM4}
			\addplot[mark=triangle*,dashed,mark size=1.3pt,black]  table [y=err_nbr, x=N, col sep=comma]{figures/LinAc2D_smooth_vortex/divergenceGF_discretization_error_ord4.csv};
			\addlegendentry{FEM-GF4}
			\addplot[blue,domain=6:106]{24./x/x/x};
			\addlegendentry{order 3}		
			\addplot[brown,domain=6:106]{130./x/x/x/x};
			\addlegendentry{order 4}			
		\end{axis}
	\end{tikzpicture}\\[6pt]
	\begin{tikzpicture}
		\begin{axis}[
			xmode=log, ymode=log,
			xmin=4,xmax=90,
			grid=major,
			xlabel={$N$},
			ylabel={Norm of Divergence},
			xlabel shift = 1 pt,
			ylabel shift = 1 pt,
			legend pos= south west,
			legend style={nodes={scale=0.6, transform shape}},
			width=.34\textwidth,
			height=.28\textwidth
			]			
			
			\addplot[mark=star,mark size=1.3pt,densely dashdotted, black]             table [y=err_nbr, x=N, col sep=comma]{figures/LinAc2D_smooth_vortex/divergenceSimple_discretization_error_ord5.csv};
			\addlegendentry{FEM5}
			\addplot[mark=star,dashed,mark size=1.3pt,black]  table [y=err_nbr, x=N, col sep=comma]{figures/LinAc2D_smooth_vortex/divergenceGF_discretization_error_ord5.csv};
			\addlegendentry{FEM-GF5}
			\addplot[brown,domain=10:80]{100./x/x/x/x};
			\addlegendentry{order 4}		
			\addplot[red,domain=10:80]{600./x/x/x/x/x};
			\addlegendentry{order 5}	
		\end{axis}
	\end{tikzpicture}
	\begin{tikzpicture}
		\begin{axis}[
			xmode=log, ymode=log,
			xmin=3,xmax=80,
			grid=major,
			xlabel={$N$},
			xlabel shift = 1 pt,
			ylabel shift = 1 pt,
			legend pos= south west,
			legend style={nodes={scale=0.6, transform shape}},
			width=.34\textwidth,
			height=.28\textwidth
			]
			\addplot[mark=diamond*,mark size=1.3pt, densely dashdotted, black]             table [y=err_nbr, x=N, col sep=comma]{figures/LinAc2D_smooth_vortex/divergenceSimple_discretization_error_ord6.csv};
			\addlegendentry{FEM6}
			\addplot[mark=diamond*,dashed,mark size=1.3pt,black]  table [y=err_nbr, x=N, col sep=comma]{figures/LinAc2D_smooth_vortex/divergenceGF_discretization_error_ord6.csv};
			\addlegendentry{FEM-GF6}
			\addplot[red,domain=8:64]{500./x/x/x/x/x};
			\addlegendentry{order 5}	
			\addplot[olive,domain=8:64]{4000./x/x/x/x/x/x};
			\addlegendentry{order 6}	
			
		\end{axis}
	\end{tikzpicture}
	\begin{tikzpicture}
		\begin{axis}[
			xmode=log, ymode=log,
			xmin=2,xmax=60,
			grid=major,
			xlabel={$N$},
			xlabel shift = 1 pt,
			ylabel shift = 1 pt,
			legend pos= south west,
			legend style={nodes={scale=0.6, transform shape}},
			width=.34\textwidth,
			height=.28\textwidth
			]			
			\addplot[mark=o,mark size=1.3pt,densely dashdotted, black]             table [y=err_nbr, x=N, col sep=comma]{figures/LinAc2D_smooth_vortex/divergenceSimple_discretization_error_ord7.csv};
			\addlegendentry{FEM7}
			\addplot[mark=o,dashed,mark size=1.3pt,black]  table [y=err_nbr, x=N, col sep=comma]{figures/LinAc2D_smooth_vortex/divergenceGF_discretization_error_ord7.csv};
			\addlegendentry{FEM-GF7}
			\addplot[olive,domain=10:53]{3000./x/x/x/x/x/x};
			\addlegendentry{order 6}	
			\addplot[violet,domain=10:53]{3000./x/x/x/x/x/x/x};
			\addlegendentry{order 7}			
		\end{axis}
	\end{tikzpicture}
	\caption{Smooth vortex: convergence of divergence operator on exact IC with respect to the number of elements in $x$}\label{fig:divergence_ic}
\end{figure}
To start the discussion with the analytical initial condition, we want to first check the amount of discrete divergence carried by the analytical IC. 
We display in Figure~\ref{fig:divergence_ic} the norm of the discrete divergence ($\tilde{D}_x u + \tilde{D}_y v$) of the analytical initial conditions of the $\mathcal{C}^\infty$ vortex \eqref{eq:definition_smooth_vortex}. We have used Gauss--Legendre polynomial $\mathbb P ^p$ and clearly see some super-convergence pattern. The divergence should decay with order $p$, but we observe in most of the cases $p+1$ convergence and for $\mathbb P^2$  we observe order $4$.
Comparing it with the classical FEM divergence, in Figure~\ref{fig:divergence_ic}, we observe that the GF divergence is much more accurate and gains one extra order of accuracy with respect to the classical method (except for $\mathbb P^1$ where also FEM gains it and for $\mathbb P^3$ where GF--FEM gains 2 order of accuracy).
Nevertheless, such an error is still too high to preserve a perturbation of an equilibrium.

Let us add a perturbation to the pressure in the form of a Gaussian centered in ${x}_p = (0.4,0.43)$ with scaling coefficient $r_0=0.1$ and radius defined as $\rho({x}) = \sqrt{\|{x} - {x}_p\|}/r_0$
\begin{equation}\label{eq:pressure_perturbation}
	\delta_p({x}) = \varepsilon  e^{-\frac{1}{2(1-\rho({x}))^2} + \frac12},
\end{equation}
until a final time $T=0.35$, obtaining the solution $({u}_p({x},t),p_p({x},t))$.

In Figure~\ref{fig:perturbation_analytical}, we do not apply any preprocessing, but we just run the analytical perturbed solution as initial condition. We compare different meshes and orders to understand how the schemes behave. The plot for order 4 and $26\times 26$ cells with SUPG-GFq shows the most accurate scheme we tested. For lower resolutions, we observe a clear advantage of the SUPG-GFq scheme over the SUPG only scheme, even if, it is easy to observe numerical noise also for SUPG-GFq in the slightly coarser mesh configurations.
\begin{figure}
	\centering
	\includegraphics[width=0.325\textwidth, trim={0 0 360 0}, clip]{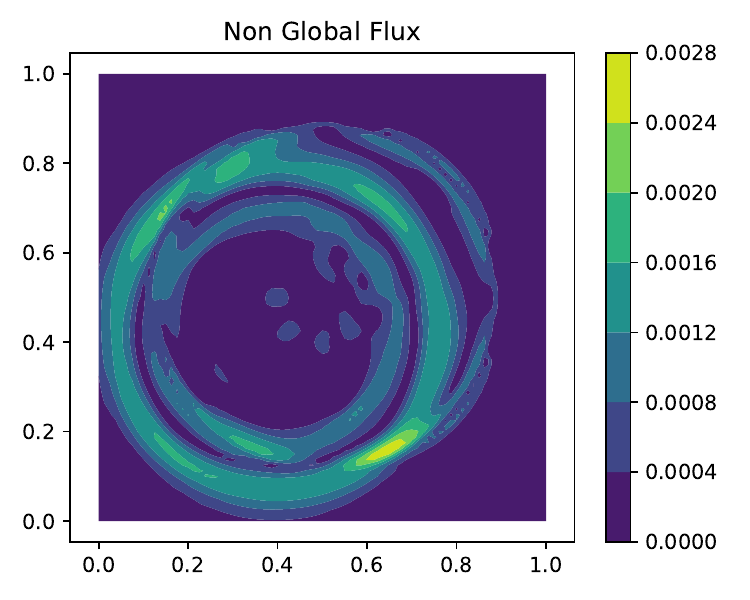}
	\includegraphics[width=0.325\textwidth, trim={0 0 360 0}, clip]{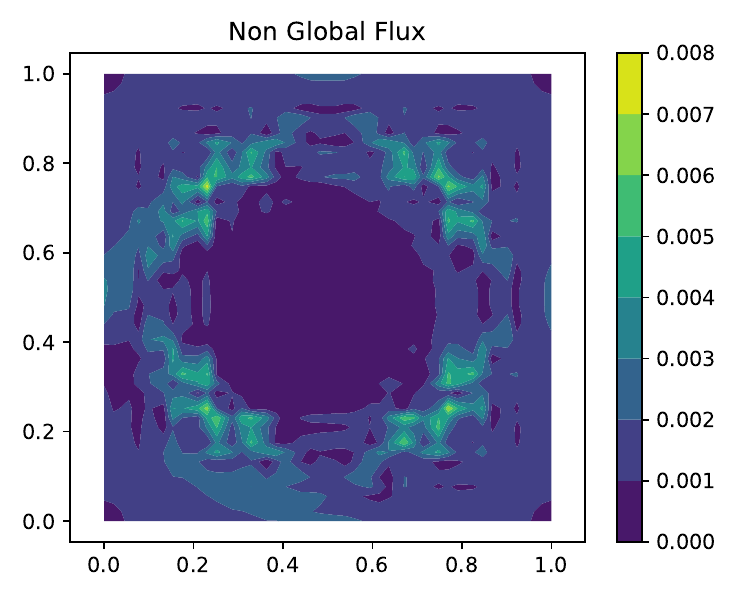}
	\includegraphics[width=0.325\textwidth, trim={0 0 360 0}, clip]{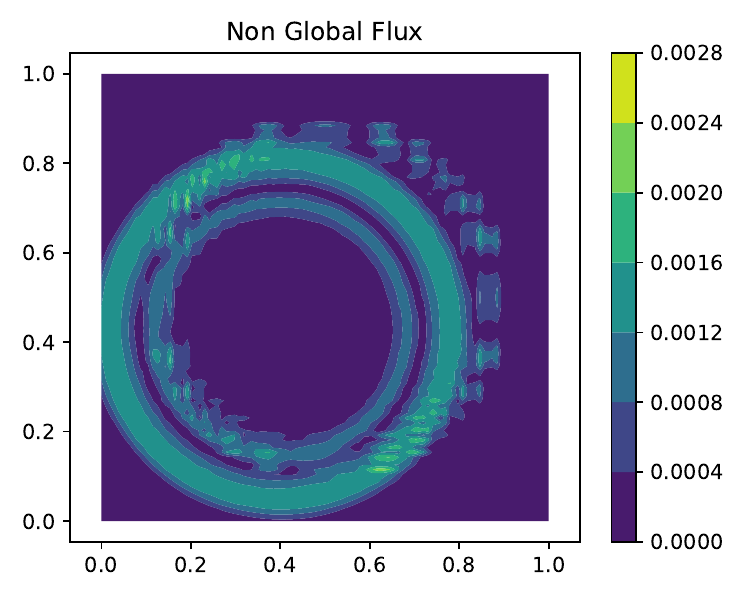}\\
	\includegraphics[width=0.325\textwidth, trim={360 0 0 0}, clip]{{figures/LinAc2D_smooth_vortex_an_perturbation/diff_vel_norm_ord2_N0080_pert_1.000e-02}.pdf}
	\includegraphics[width=0.325\textwidth, trim={360 0 0 0}, clip]{{figures/LinAc2D_smooth_vortex_an_perturbation/diff_vel_norm_ord4_N0013_pert_1.000e-02}.pdf}
	\includegraphics[width=0.325\textwidth, trim={360 0 0 0}, clip]{{figures/LinAc2D_smooth_vortex_an_perturbation/diff_vel_norm_ord4_N0026_pert_1.000e-02}.pdf}
	\caption{Perturbation($\varepsilon=10^{-2}$) test: analytical solution. Plot of $\lVert\vec{u}_{eq}-\vec{u}_p\rVert$, with $\vec{u}_{eq}$ the analytical equilibrium \eqref{eq:definition_smooth_vortex}. SUPG (top), SUPG-GFq (bottom). Left $\mathbb P^1$ with $80 \times 80$ cells, center $\mathbb P^3$ with 13 cells, right $\mathbb P^3$ with 26 cells}\label{fig:perturbation_analytical}
\end{figure}

\begin{figure}
	\centering
	\includegraphics[width=0.325\textwidth, trim={0 0 360 0}, clip]{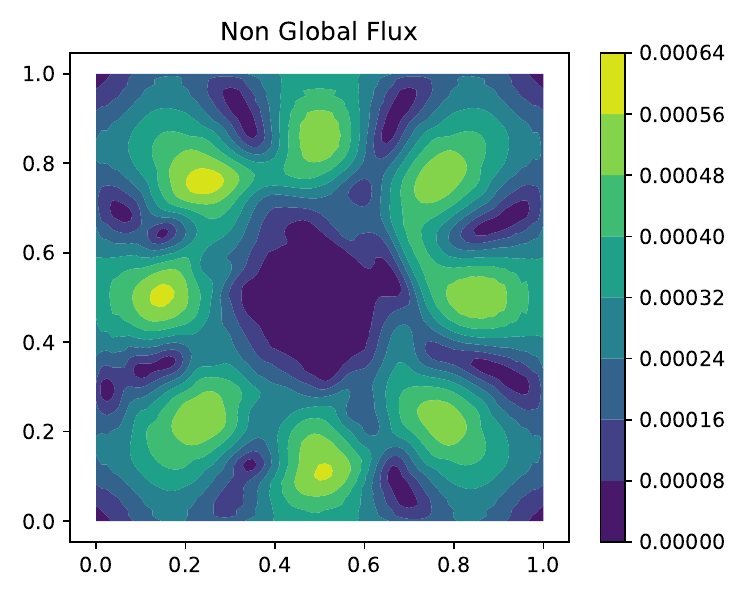}
	\includegraphics[width=0.325\textwidth, trim={0 0 360 0}, clip]{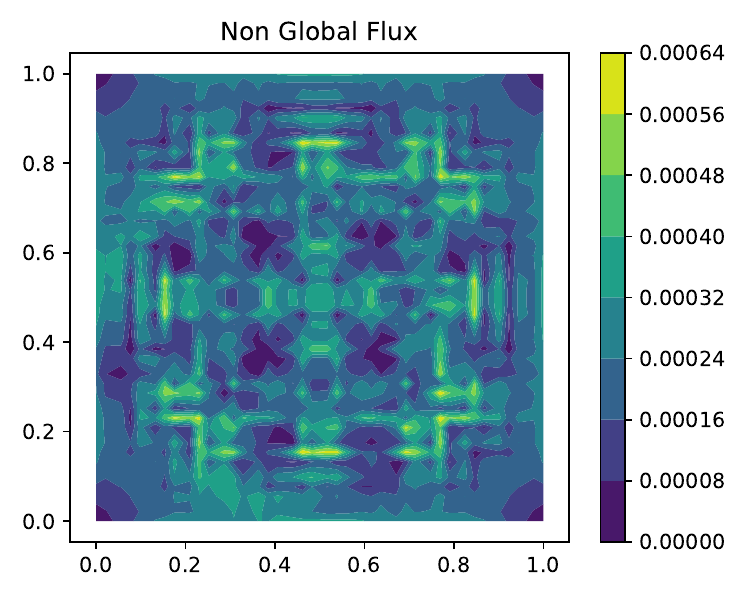}
	\includegraphics[width=0.325\textwidth, trim={0 0 360 0}, clip]{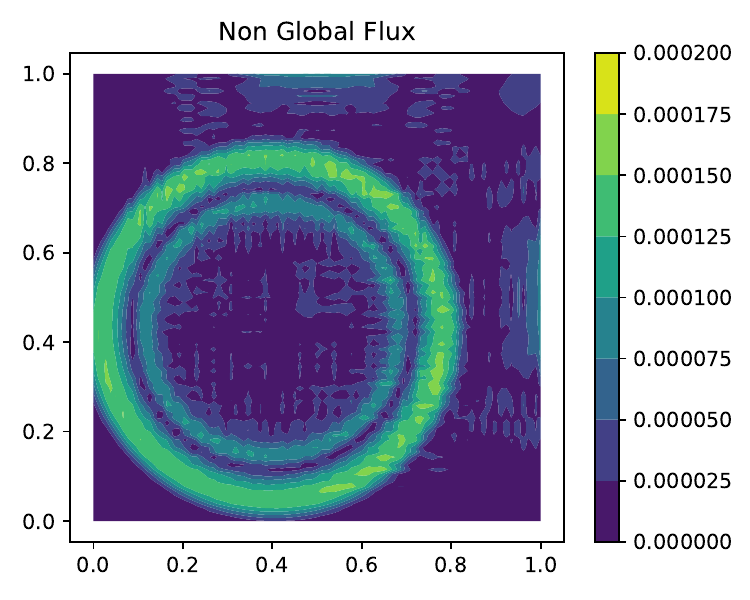}\\
	\includegraphics[width=0.325\textwidth, trim={360 0 0 0}, clip]{{figures/LinAc2D_smooth_vortex_int_perturbation/diff_vel_norm_ord2_N0080_pert_1.000e-03}.pdf}
	\includegraphics[width=0.325\textwidth, trim={360 0 0 0}, clip]{{figures/LinAc2D_smooth_vortex_int_perturbation/diff_vel_norm_ord4_N0013_pert_1.000e-03}.pdf}
	\includegraphics[width=0.325\textwidth, trim={360 0 0 0}, clip]{{figures/LinAc2D_smooth_vortex_int_perturbation/diff_vel_norm_ord4_N0026_pert_1.000e-03}.pdf}
	\caption{Perturbation($\varepsilon=10^{-3}$) test: line-by-line equilibrium solution, see Section~\ref{sec:well_prepared_ic}. Plot of $\lVert\vec{u}_{eq}-\vec{u}_p\rVert$, with $\vec{u}_{eq}$ the analytical equilibrium \eqref{eq:definition_smooth_vortex}. SUPG (top), SUPG-GFq (bottom). Left $\mathbb P^1$ with $80 \times 80$ cells, center $\mathbb P^3$ with 13 cells, right $\mathbb P^3$ with 26 cells}\label{fig:perturbation_integral}
\end{figure}
\begin{figure}
	\centering	
	\includegraphics[width=0.325\textwidth, trim={0 0 360 0}, clip ]{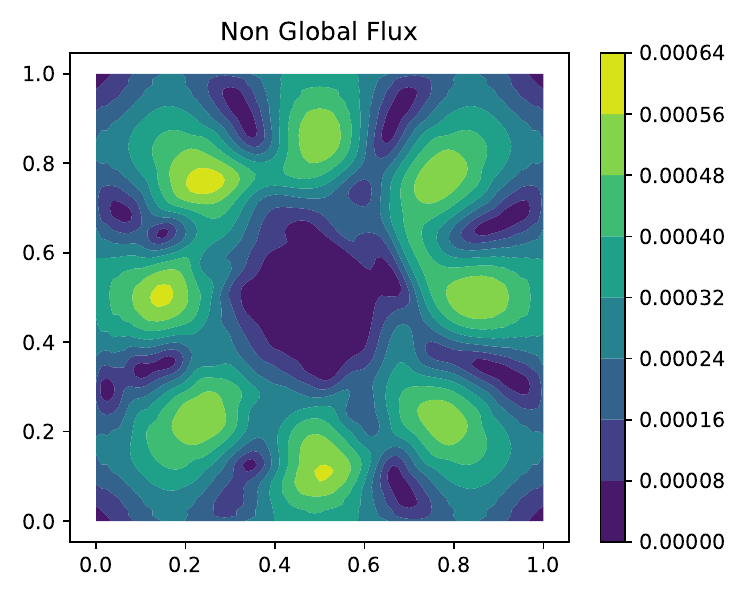}
	\includegraphics[width=0.325\textwidth, trim={0 0 360 0}, clip]{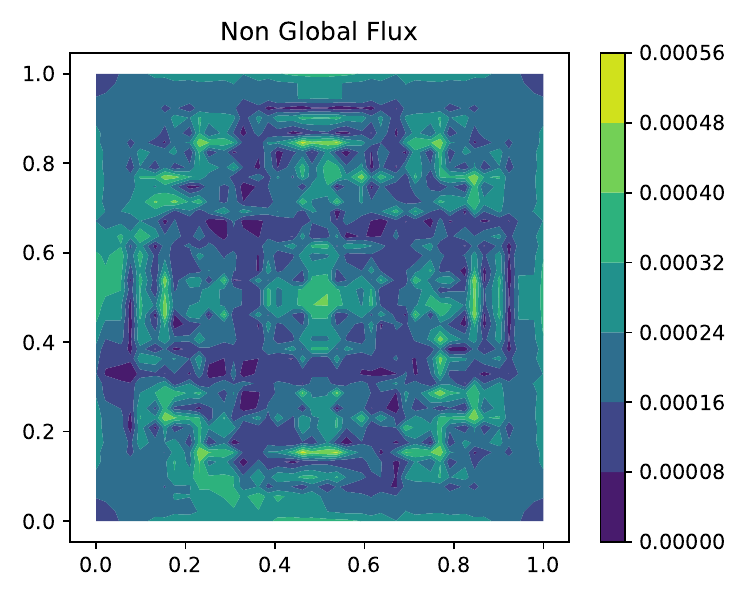}
	\includegraphics[width=0.325\textwidth, trim={0 0 360 0}, clip]{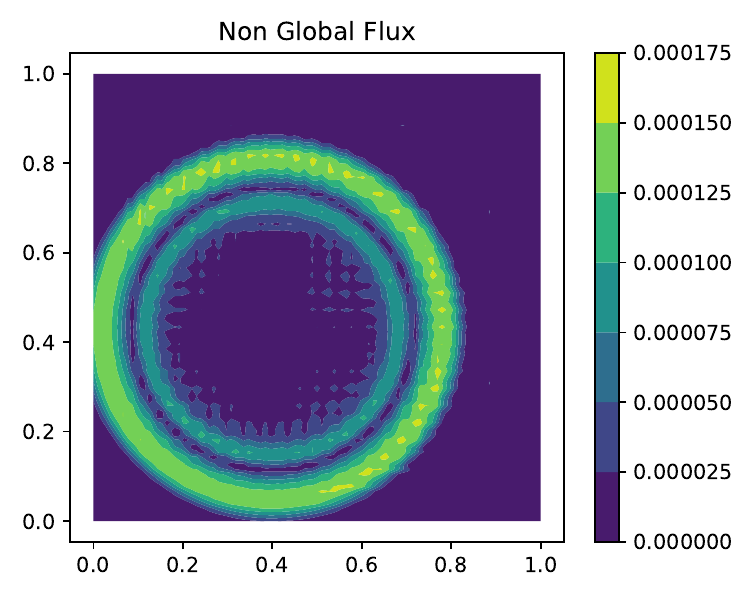}\\
	\includegraphics[width=0.325\textwidth, trim={360 0 0 0}, clip ]{{figures/LinAc2D_smooth_vortex_opt_perturbation/diff_vel_norm_ord2_N0080_pert_1.000e-03}.pdf}
	\includegraphics[width=0.325\textwidth, trim={360 0 0 0}, clip]{{figures/LinAc2D_smooth_vortex_opt_perturbation/diff_vel_norm_ord4_N0013_pert_1.000e-03}.pdf}
	\includegraphics[width=0.325\textwidth, trim={360 0 0 0}, clip]{{figures/LinAc2D_smooth_vortex_opt_perturbation/diff_vel_norm_ord4_N0026_pert_1.000e-03}.pdf}
	\caption{Perturbation($\varepsilon=10^{-3}$) test: optimal equilibrium solution, see Section~\ref{sec:well_prepared_ic}. Plot of $\lVert\vec{u}_{eq}-\vec{u}_p\rVert$, with $\vec{u}_{eq}$ the analytical equilibrium \eqref{eq:definition_smooth_vortex}. SUPG (top), SUPG-GFq (bottom). Left $\mathbb P^1$ with $80 \times 80$ cells, center $\mathbb P^3$ with 13 cells, right $\mathbb P^3$ with 26 cells}\label{fig:perturbation_optimal}
\end{figure}
\begin{figure}
\centering
\includegraphics[width=0.325\textwidth, trim={0 0 360 0}, clip]{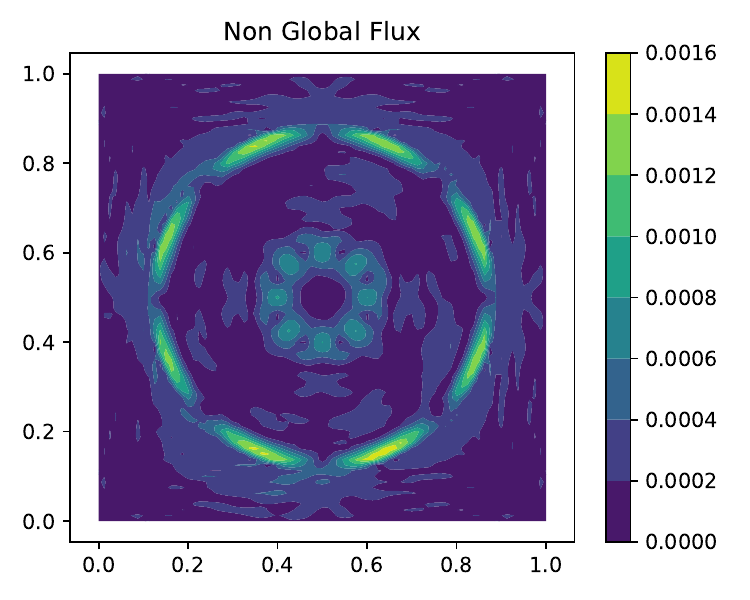}
\includegraphics[width=0.325\textwidth, trim={0 0 360 0}, clip]{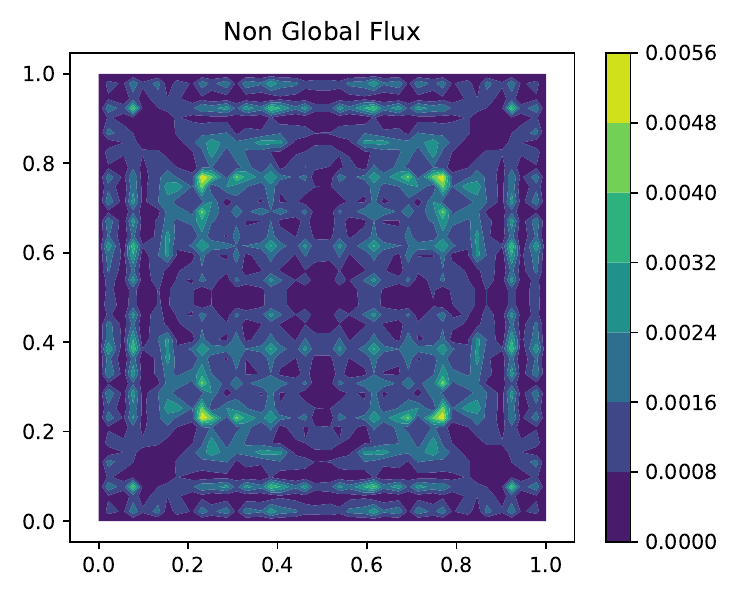}
\includegraphics[width=0.325\textwidth, trim={0 0 360 0}, clip]{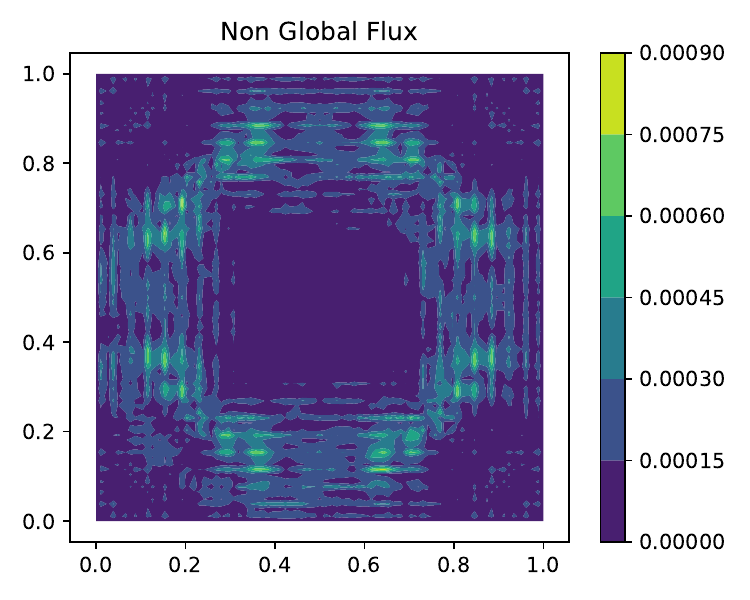}\\
\includegraphics[width=0.325\textwidth, trim={360 0 0 0}, clip]{{figures/LinAc2D_smooth_vortex_opt_perturbation/diff_vel_norm_OSS_ord2_N0080_pert_1.000e-03}.pdf}
\includegraphics[width=0.325\textwidth, trim={360 0 0 0}, clip]{{figures/LinAc2D_smooth_vortex_opt_perturbation/diff_vel_norm_OSS_ord4_N0013_pert_1.000e-03}.pdf}
\includegraphics[width=0.325\textwidth, trim={360 0 0 0}, clip]{{figures/LinAc2D_smooth_vortex_opt_perturbation/diff_vel_norm_OSS_ord4_N0026_pert_1.000e-03}.pdf}
\caption{Perturbation($\varepsilon=10^{-3}$) test: optimal equilibrium solution, see Section~\ref{sec:well_prepared_ic}. Plot of $\lVert\vec{u}_{eq}-\vec{u}_p\rVert$, with $\vec{u}_{eq}$ the analytical equilibrium \eqref{eq:definition_smooth_vortex}. OSS (top), OSS-GFq (bottom). Left $\mathbb Q^1$ with $80 \times 80$ cells, center $\mathbb Q^3$ with 13 cells, right $\mathbb Q^3$ with 26 cell}\label{fig:perturbation_optimal_OSS}
\end{figure}
\begin{figure}
	\centering
	\includegraphics[width=0.325\textwidth, trim={0 0 360 0}, clip]{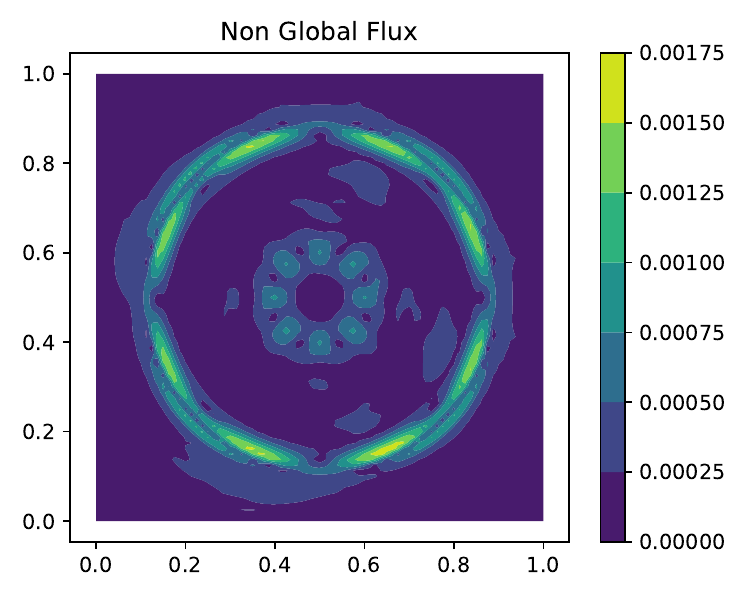}
	\includegraphics[width=0.325\textwidth, trim={0 0 360 0}, clip]{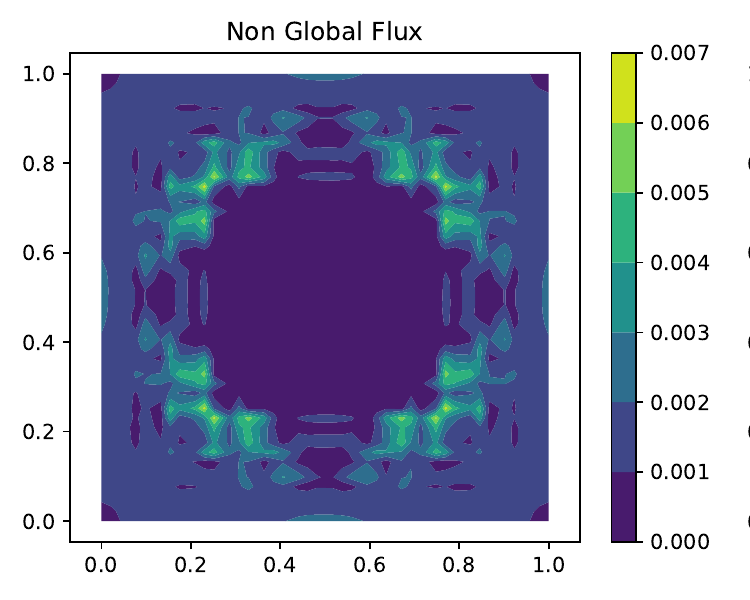}
	\includegraphics[width=0.325\textwidth, trim={0 0 360 0}, clip]{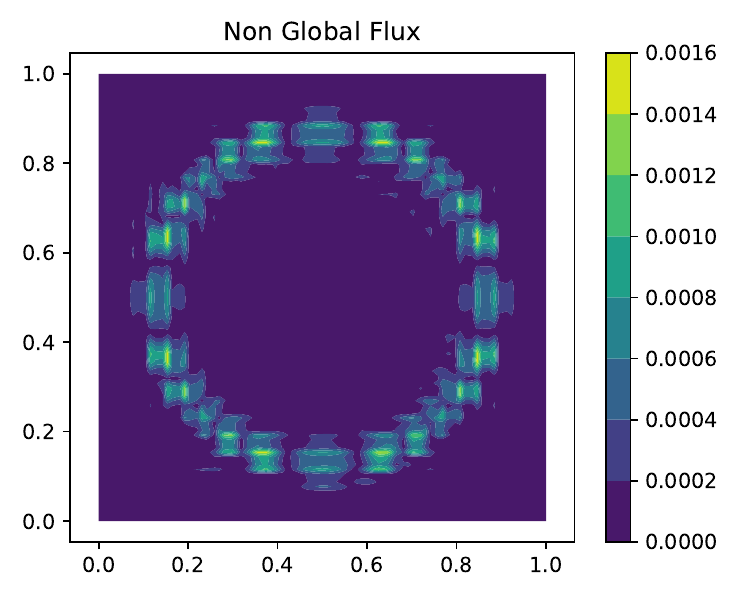}\\
	\includegraphics[width=0.325\textwidth, trim={360 0 0 0}, clip ]{{figures/LinAc2D_smooth_vortex_num_perturbation/diff_vel_norm_ord2_N0080_pert_1.000e-03}.pdf}
	\includegraphics[width=0.325\textwidth, trim={360 0 0 0}, clip]{{figures/LinAc2D_smooth_vortex_num_perturbation/diff_vel_norm_ord4_N0013_pert_1.000e-03}.pdf}
	\includegraphics[width=0.325\textwidth, trim={360 0 0 0}, clip]{{figures/LinAc2D_smooth_vortex_num_perturbation/diff_vel_norm_ord4_N0026_pert_1.000e-03}.pdf}
	\caption{Perturbation test ($\varepsilon=10^{-3}$): long time equilibrium solution. Plot of $\lVert{u}_{eq}-{u}_p\rVert$, with ${u}_{eq}$ the analytical equilibrium \eqref{eq:definition_smooth_vortex}. SUPG (top), SUPG-GFq (bottom). Left $\mathbb Q^1$ with $80 \times 80$ cells, center $\mathbb Q^3$ with 13 cells, right $\mathbb Q^3$ with 26 cells}\label{fig:perturbation_numerical}
\end{figure}
In the other setups studied, we can reduce the size of the perturbation and still be able to see its evolution. In Figure~\ref{fig:perturbation_integral}, we use the integral procedure to preprocess the data and find the equilibrium, see Section~\ref{sec:well_prepared_ic}, while in Figure~\ref{fig:perturbation_optimal}, we use the optimization procedure as described above. Once the equilibrium is obtained, we add the perturbation $\delta_p$ and let the simulation run. We observe that the results, even with a much smaller perturbation, are very accurate for SUPG-GFq even for very coarse grids. To achieve comparable results, the standard SUPG scheme requires a very fine mesh and high resolution to capture the motion of the perturbation.

In Figure~\ref{fig:perturbation_optimal_OSS}, we start from the optimal equilibrium obtained according to the procedure outlines in Definition \ref{def:opt-proj}, but we use the OSS stabilization technique. There are less precise results, with respect to SUPG, as boundaries in this case influence more the solutions, but the GF version is still very accurate for all the presented tests. We only use OSS in this test for brevity.

Finally, in Figure~\ref{fig:perturbation_numerical}, we use as initial solution the long-time simulation of the previous test ($T=100$) with the SUPG-GFq scheme that reaches convergence. We observe that adding a perturbation to such an initial condition leads to very clear results for the SUPG-GFq scheme.

\subsection{Riemann Problem}
\begin{figure}
	\includegraphics[width=\textwidth, trim={100 0 80 0}, clip]{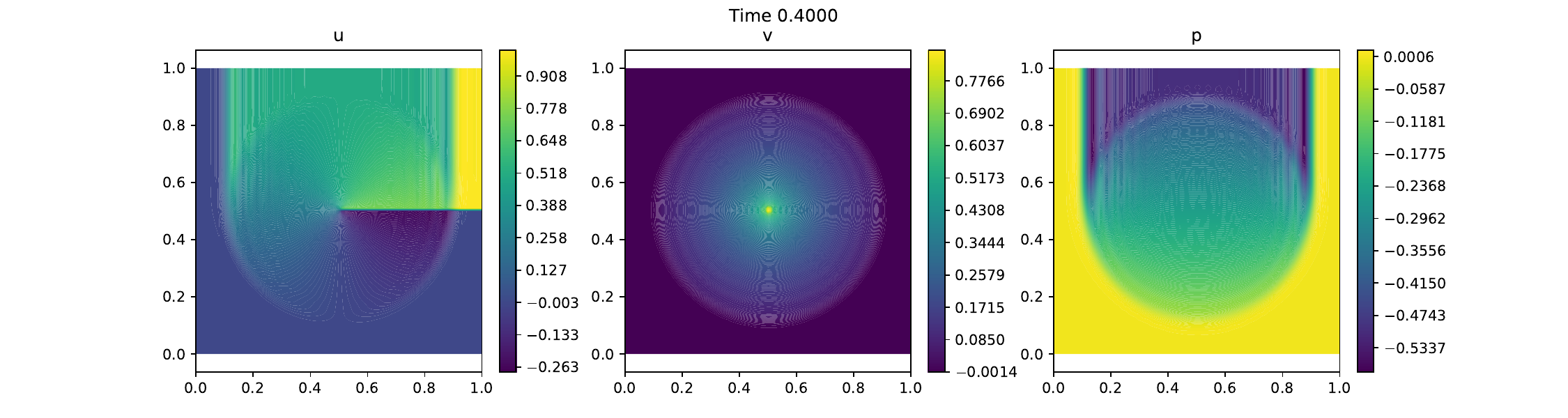}
	\caption{Riemann Problem. Simulation at time $T=0.4$ with $\mathbb P^2$ elements and $50\times 50$ cells with SUPG-GFq scheme}\label{fig:RP_simul}
\end{figure}
\begin{figure}
	\includegraphics[width=0.48\textwidth]{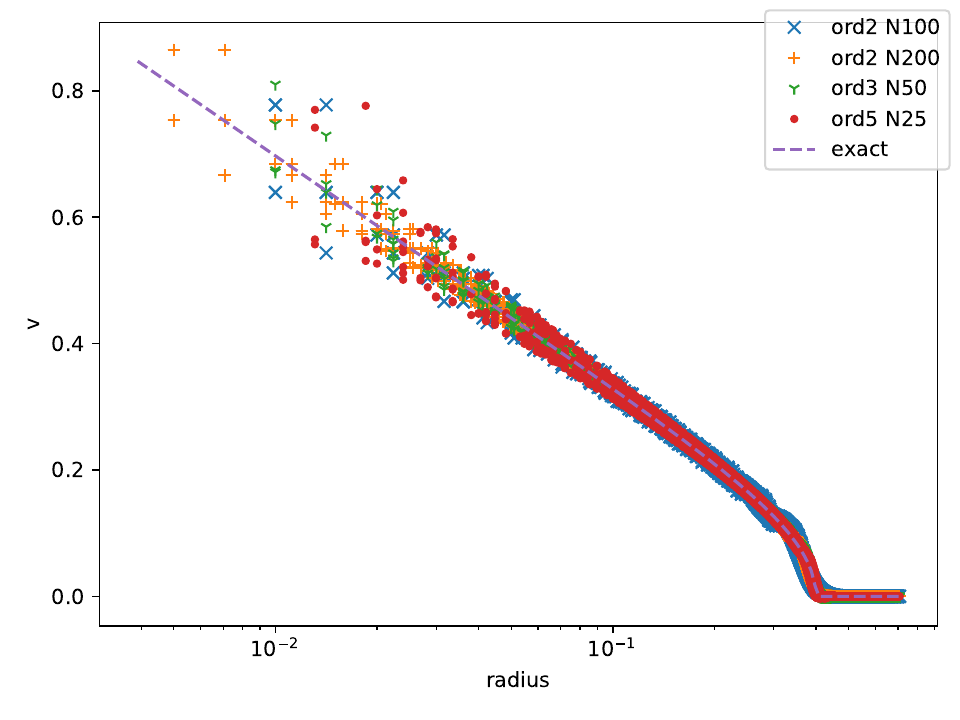}
	\includegraphics[width=0.48\textwidth]{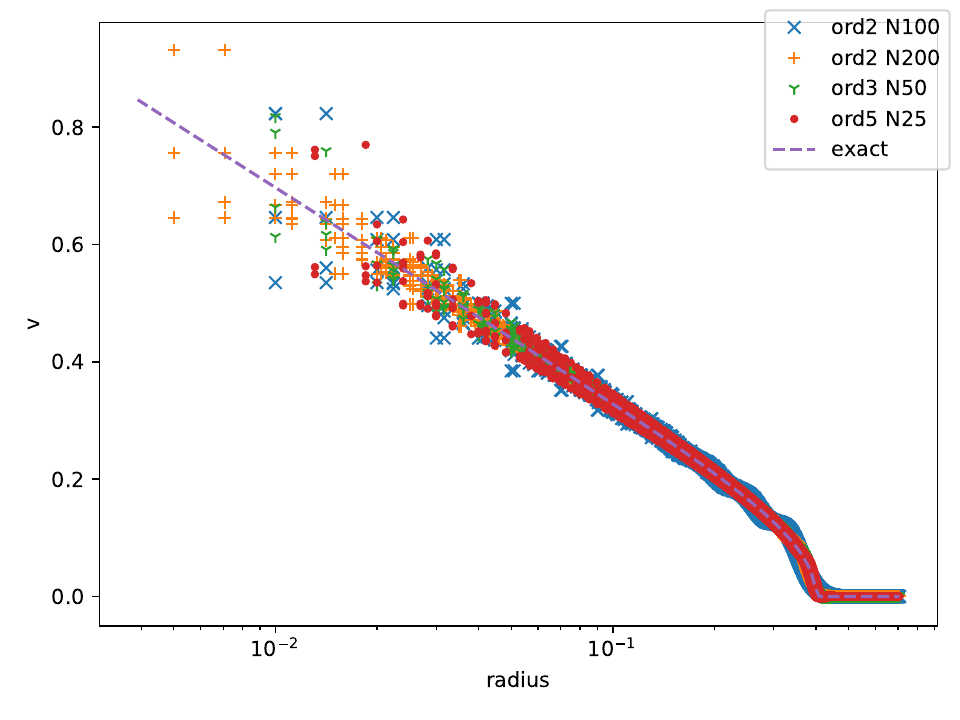}
	\caption{Riemann Problem. Distribution of the solution $v$ for different elements and meshes. Left SUPG scheme, right SUPG-GFq scheme}\label{fig:RP_distribution}
\end{figure}
We present in this section a two-dimensional Riemann problem (RP) centered in ${x}_0  =(0.5,0.5)$ on the unit square $\Omega=[0,1]^2$ \cite{barsukow2022exact}. The initial conditions are
\begin{equation}
	u({x})=
	\begin{cases}
		1, &\text{if }x>0.5\text{ and }y>0.5,\\
		0, &\text{else},
	\end{cases}\qquad
	v({x})=0,\qquad p({x})=0.
\end{equation}
It has been shown \cite{barsukow2022exact} that the perpendicular component $v$ has a logarithmic singularity in the center of the RP for $t>0$:
\begin{equation}
	\begin{split}
	&v(x,y,t)= \frac{1}{2\pi} \mathcal{L} \left(\frac{\sqrt{(x-x_0)^2+(y-y_0)^2}}{ct}\right),\\
	& \mathcal{L} (s):= \log \left(\frac{1+\sqrt{1-s^2}}{s}\right)  = - \log\left(\frac{s}{2}\right) -\frac{s^2}{4} +\mathcal{O}(s^4).
	\end{split}
\end{equation}
In Figure~\ref{fig:RP_simul}, we plot the solution for $\mathbb Q^2 $ elements with the SUPG-GFq method. Other resolutions and orders give qualitatively similar results. 
We can observe that some oscillations appear due to the Gibbs phenomenon at the discontinuities of the solution.

This can also be seen in Figure~\ref{fig:RP_distribution}, where the solution $v$ is plotted against the radius. We do not observe many differences between SUPG and SUPG-GFq.

\section{Conclusions and perspectives}

In this paper, we have studied the preservation of steady states for acoustics when using stabilized continuous finite elements.
We have shown that, despite their structure,  classical grad-div stabilizations as SUPG and OSS  are in general not stationarity preserving
due to an incompatibility between  the stabilizing term and of the Galerkin  one. Borrowing ideas  from the so-called Global Flux quadrature, 
 we have proposed  a new  framework allowing   to construct constraint-compatible stabilization operators, which additionally are non-vanishing for at least some of the unwanted spurious modes.
  We have characterized  rigorously   the discrete kernel of the schemes obtained in terms of stability and consistency, and  
  characterized the corresponding curl involutions   using Fourier symbols. 
Numerical results confirm the theoretical developments. Ongoing work involves the extension  to non-homogeneous systems, e.g. the linearized shallow water equations
with Coriolis, friction, mass, and other sources, as well as the extension to discontinuous polynomial approximations.
The application to non-linear systems (shallow water and Euler equations with gravity) is also under development.

\section*{Acknowledgments}
D.T. was supported by the Ateneo Sapienza projects 2022 “Approssimazione numerica
di modelli differenziali e applicazioni” and 2023 “Modeling, numerical treatment of hyperbolic
equations and optimal control problems”. DT is member of the INdAM Research
National Group of Scientific Computing (INdAM-GNCS).
M.R. is a member of the Cardamom team, Inria at University of Bordeaux.

\appendix

\section{Additional definitions and proofs}

\subsection{Proof of Theorem \ref{thm:charpolcomposition}} \label{sec:proofcharpolcomposition}
 The characteristic polynomials of $A$ and $B$ are
 \begin{align}
  (\mathbb F_{t_x}(A))_{r,s} &= \sum_{k \in \mathbb Z} \alpha^r_{k,s} t_x^{k} &
  (\mathbb F_{t_x}(B))_{r,s} &= \sum_{k \in \mathbb Z} \beta^r_{k,s} t_x^{k} 
 \end{align}
 while that of $AB$ is
 \begin{align}
  (\mathbb F_{t_x}(AB))_{r,s'} &= \sum_{k \in \mathbb Z} \sum_{s = 1}^K \sum_{k' \in \mathbb Z} \alpha^r_{k,s}  \beta^s_{k',s'} t_x^{k+k'} = \sum_{s = 1}^K  \sum_{k \in \mathbb Z}  \alpha^r_{k,s} t_x^{k}      \sum_{k' \in \mathbb Z}  \beta^s_{k',s'} t_x^{k'} =\sum_{s = 1}^K (\mathbb F_{t_x}(A))_{r,s}  (\mathbb F_{t_x}(B))_{s,s'}.
 \end{align}

\subsection{Proof of Theorem \ref{thm:highordercompositiontensor}} \label{sec:proofhighordercompositiontensor}
 Simply inserting the definitions:
 \begin{align*}
  ((A^x \otimes A^y)  (B^x \otimes B^y) q)_{ij}^{rt}
  &= \left( \sum_{(k,\ell) \in \mathbb Z^2} \sum_{s,p = 1}^K (\alpha^x)^{r}_{k,s}  (\alpha^y)^{t}_{\ell,p}q_{i+k,s;j+\ell,p}  \right ) \\
  &
  \left( \sum_{(k,\ell) \in \mathbb Z^2} \sum_{s,p = 1}^K (\beta^x)^{r,t}_{k,s;\ell,p} (\beta^y)^{t}_{\ell,p}  q_{i+k,s;j+\ell,p} \right )\\
  &\!\!\!\!\!\!\!\!\!\!\!\!\!\!\!\!\!\!\!\!\!\!\!\!\!\!\!\!\!\!\!\!\!\!\!\!\!= \sum_{(k,\ell) \in \mathbb Z^2} \sum_{(k',\ell') \in \mathbb Z^2}
  \sum_{s,p = 1}^K   \sum_{s',p' = 1}^K
  (\alpha^x)^{r}_{k,s} (\alpha^y)^{t}_{\ell,p}  (\beta^x)^{r}_{k',s'} (\beta^y)^{t}_{\ell',p'}  q_{i+k+k',s';j+\ell+\ell',p'} \\
  &\!\!\!\!\!\!\!\!\!\!\!\!\!\!\!\!\!\!\!\!\!\!\!\!\!\!\!\!\!\!\!\!\!\!\!\!\!=\sum_{k \in \mathbb Z} \sum_{k' \in \mathbb Z}
  \sum_{s = 1}^K   \sum_{s' = 1}^K
  (\alpha^x)^{r}_{k,s}  (\beta^x)^{r}_{k',s'}   
  \sum_{\ell \in \mathbb Z} \sum_{\ell' \in \mathbb Z}
  \sum_{p = 1}^K   \sum_{p' = 1}^K
  (\alpha^y)^{t}_{\ell,p}  (\beta^y)^{t}_{\ell',p'}  q_{i+k+k',s';j+\ell+\ell',p'} \\
  &\!\!\!\!\!\!\!\!\!\!\!\!\!\!\!\!\!\!\!\!\!\!\!\!\!\!\!\!\!\!\!\!\!\!\!\!\!= ((A^x  B^x) \otimes (A^y  B^y)q)_{ij}^{rt}.
 \end{align*}

\section{Fully discrete $T^2$ SUPG operators}\label{app:fully_discrete_L2}
As an example, we explicitly list here the $T^2$ operator for the SUPG stabilization.
We recall that indexes $m,r$ are devoted to subtimenodes, indexes $i,j$ are devoted to degrees of freedom, index $d$ is for dimension and $s,w$ for variable index.
We start first with the classical SUPG operator, corresponding to the matrix formulation \eqref{eq:L2SUPG}:
\begin{subequations}
\begin{align}
	\begin{split}
		T^{2,m}_{u}(\underline{{q}})=&M_x\otimes M_y \frac{u^{m}-u^{0}}{\dt} + D_x\otimes M_y  \sum_r \theta^m_r p^{r} +\\& \alpha h\left(  D^x \otimes M_y  \frac{p^{m}-p^{0}}{\dt}+ D_{x}^x\otimes M_y  \sum_r \theta^m_r u^{r} +
			D^x \otimes  D_y\sum_r \theta^m_r v^{r} \right),
		\end{split}\\
	\begin{split}
		T^{2,m}_{v}(\underline{{q}})=&M_x\otimes M_y \frac{v^{m}-v^{0}}{\dt} + M_x\otimes D_y \sum_r \theta^m_r p^{r} + \\&\alpha h\left(  M_x\otimes D^y \frac{p^{m}-p^{0}}{\dt}+ D_{x} \otimes D^y \sum_r \theta^m_r u^{r} +
			 M_x \otimes  D^y_y \sum_r \theta^m_r v^{r}\right),
		\end{split}\\
	\begin{split}
		T^{2,m}_{p}(\underline{{q}})=&M_x\otimes M_y \frac{p^{m}-p^{0}}{\dt} + D_x\otimes M_y  \sum_r \theta^m_r u^{r} +M_x\otimes D_y \sum_r \theta^m_r v^{r} + \\
		&\alpha h \left( D^x\otimes M_y \frac{u^{m}-u^{0}}{\dt}+M_x\otimes D^y \frac{v^{m}-v^{0}}{\dt}+ (D_{x}^x\otimes M_y +M_x \otimes D^y_y)\sum_r \theta^m_r p^{r}\right).
		\end{split}
\end{align}
\end{subequations}
\begin{subequations}
Now, we describe the SUPG-GFq $T^2$ operators, corresponding to the semidiscrete \eqref{eq:supg_GF}
\begin{align}
	\begin{split}
		T^{2,m}_{u}(\underline{{q}})= 
		&M_x\otimes M_y \frac{u^{m}-u^{0}}{\dt} + D_x\otimes M_y  \sum_r \theta^m_r p^{r} +\\& \alpha h \left(  D^x \otimes M_y  \frac{p^{m}-p^{0}}{\dt}+   {D}_{x}^x\otimes D_y I_y  \sum_r \theta^m_r u^{r} +
		   D^x_x I_x \otimes  {D}_y\sum_r \theta^m_r v^{r} \right),
	\end{split}
\end{align}

\begin{align}
	\begin{split}
		T^{2,m}_{v}(\underline{{q}})=
		&M_x\otimes M_y \frac{v^{m}-v^{0}}{\dt} + M_x\otimes D_y \sum_r \theta^m_r p^{r} +\\& \alpha h\left(  M_x\otimes D^y \frac{p^{m}-p^{0}}{\dt}+  {D}_{x} \otimes D^y_y I_y \sum_r \theta^m_r u^{r} +
		 D_x I_x \otimes {D}^y_y \sum_r \theta^m_r v^{r} \right),
	\end{split}
\end{align}
\begin{align}
	\begin{split}
		T^{2,m}_{p}(\underline{{q}})=
		&M_x\otimes M_y \frac{p^{m}-p^{0}}{\dt} + {D}_x\otimes D_y I_y   \sum_r \theta^m_r u^{r} + D_x I_x \otimes {D}_y \sum_r \theta^m_r v^{r} + \\
		&\alpha h \left(   D^x\otimes M_y  \frac{u^{m}-u^{0}}{\dt}+M_x\otimes D^y \frac{v^{m}-v^{0}}{\dt}+  (D_{x}^x\otimes M +M \otimes D^y_y)\sum_r \theta^m_r p^{r} \right).
	\end{split}
\end{align}
\end{subequations}

\section{Curl involution for OSS: definitions}\label{sec:curl_inv_OSS}
Here, we give implicitly the definition of $\mathcal K_u$ and $\mathcal K_v$.

\begin{equation*}	
	\begin{split}
	(\mathbb F_{t_y}(D_y M_y^2)& \mathbb F_{t_x}(M_x) )	\mathcal{K}_u = -\bigg\lbrace\left(\mathbb F_{t_y}(D_y)^2+\mathbb F_{t_y}(D_y^y M_y)\right) \bigg(\mathbb F_{t_y}(D_y)^2 \mathbb F_{t_x}(M_x)^2 -\mathbb F_{t_y}(M_y)\\
&\left(-\mathbb F_{t_y}(D_y^y) \mathbb F_{t_x}(M_x)^2-\mathbb F_{t_x}(D_x)^2 \mathbb F_{t_y}(M_y)-\mathbb F_{t_x}(D_x^x M_x) \mathbb F_{t_x,t_y}(M_y)\right)\bigg)\bigg\rbrace,\\
		( \mathbb F_{t_x}(D_x M_x^2) &\mathbb F_{t_y}(M_y)) \mathcal{K}_v = \bigg\lbrace \left(\mathbb F_{t_x}(D_x)^2+\mathbb F_{t_x}(D_x^x M_x)\right) \bigg(\mathbb F_{t_y}(D_y)^2 
		\mathbb F_{t_x}(M_x)^2-\mathbb F_{t_y}(M_y) \\& \bigg(-\mathbb F_{t_y}(D_y^y) \mathbb F_{t_x}(M_x)^2-\mathbb F_{t_x}(D_x)^2 \mathbb F_{t_y}(M_y)-\mathbb F_{t_x}(D_x^x M_x) \mathbb F_{t_y}(M_y)\bigg)\bigg)\bigg\rbrace .
	\end{split}
\end{equation*}
\normalsize

\section{One--dimensional Kernel characterization }\label{app:kernel_one_dim}

\begin{lemma}[Invertibility of mixed mass matrix]\label{lem:invertibility_mixed}
	Consider $\lbrace \hat{\varphi}_i \rbrace_{i=0}^K$ the set of Lagrangian polynomials generated by $K+1$ Gauss Lobatto points in $\mathbb P^K([0,1])$ and  $\lbrace \hat{\psi}_i \rbrace_{i=0}^{K-1}$ the set of Lagrangian polynomials generated by $K$ Gauss Lobatto points in $\mathbb P^{K-1}([0,1])$. Consider the matrix $A_{ij}:=\int_0^1 \hat{\varphi}_i \hat{\psi}_j \dd x$ for $i=0,\dots,K$ and $j=0,\dots,K-1$.
	Now, consider the square matrices obtained as the restriction of $A$ that we denote by  $B_{ij}=A_{ij}$ for $i=0,\dots,K-1$, $j=0,\dots,K-1$. Then, $B$ is invertible.
\end{lemma}
\begin{proof}
	Let us denote with $(x_i,w_i)$ for $i=0,\dots,K$ the $K+1$ Gauss--Lobatto quadrature nodes and weights, respectively. This quadrature formula is exact for polynomials of degree $2K-1$, hence,
	\begin{equation}\label{eq:B_quad}
		B_{ij}=\int_0^1 \hat{\varphi}_i(x) \hat{\psi}_j(x)\dd x = w_i \hat{\psi}_j(x_i).
	\end{equation}
	
	We want to show that the system $\sum_{j=1}^{K-1} B_{ij} q_j = r_i$ for $i=1,\dots,K-1$ is invertible. Using \eqref{eq:B_quad}, the system becomes
	\begin{equation}\label{eq:smaller_sys}
		\sum_{j=1}^{K-1} B_{ij} q_j = r_i \Longleftrightarrow \sum_{j=1}^{K-1} \hat{\psi}_j(x_i ) q_j = \frac{r_i}{w_i}.
	\end{equation}
	Now, the system \eqref{eq:smaller_sys} for $q_j$ with $j=0,\dots,K-1$ is equivalent to finding  a polynomial $q_h \in \mathbb P^{K-1}$ that interpolates the $K$ distinct point $\lbrace (0,0) \rbrace \cup \lbrace (x_i, r_i/w_i ) \rbrace_{i=1}^{K-1}$, which has one and only one solution. Hence, the matrix $B$ is invertible.
\end{proof}

\begin{proposition}[Kernel characterization of $\tilde{D}_x$]
	$\tilde{D}_x: \mathbb R^{N_x\times K }\to \mathbb R^{N_x \times K -1}$ has kernel of dimension one and it is generated by a function that is discontinuous at each cell interface, hence, not the constant function.
\end{proposition}
\begin{proof}
	Let us quickly recall the definition and notation of the matrix. Let us order the indexes of the degrees of freedom of $V^K_\dx(\Omega_\dx^x)$ with a unique index. Recall that for $i=0,\dots,N_x-1$ $\varphi_{i,k}|_{E^x_{i}} \in \mathbb P^K (E^x_{i})$ are the Lagrangian basis functions defined through the Gauss--Lobatto quadrature points. We also recall that for the degrees of freedom $r=0,\dots,K$ of the cell $E^x_{i}$ we assign a unique index the  $\alpha := iK+r$ with the equivalence $\varphi_{i,K}=\varphi_{i+1,0}$ for $i=0,\dots,N_x-2$.
	For $V_{\dx,0}^K(\Omega_\dx^x)$ the first and the last degrees of freedom are neglected, and we can define its basis as $\lbrace \varphi_\alpha \rbrace_{\alpha = 1}^{N_xK-1}$.
	For the broken polynomial space $V_{\dx,b}^{K-1}(\Omega_\dx^x)$, instead, for each Lagrangian basis function $\psi_{j,\ell}|_{E^x_{j}} \in  \mathbb P^{K-1} (E^x_{j})$ for $j=0,\dots,N_x-1$ and $k=1,\dots,K$, we define another unique index $\beta = jK+k$ with $\beta = 1,\dots, N_xK$.
	
	Then, the matrix is defined for $\alpha=1,\dots,N_xK-1$ and $\beta=1,\dots,N_xK$ as
	\begin{align}
		(\tilde{D}_x)_{\alpha;\beta} &:= \int_{\Omega^x_\dx} \varphi_{\alpha} (x)\psi_{\beta}(x)\dd x.
	\end{align}
	Now, we want to compute the kernel of these operators, so we are solving for each of them $N_xK-1$ equations for $\alpha=1,\dots,N_xK-1$
	\begin{equation}\label{eq:kernel_equation}
		\sum_{\beta=1}^{N_xK}(\tilde D_x)_{\alpha,\beta} q_\beta = 0.
	\end{equation}
	
	In this first part of the proof, we show that any basis the kernel of $\tilde{D}_x$ must have  discontinuities across cells if $q_{i,0}\neq 0$ for all $i=1,\dots,N_x-1$. Later, we will show that $\tilde{D}_x$ has full rank $N_x-1$ and hence that basis generates the whole kernel.
	
	Each term of the matrices can be computed as sum of integrals of polynomials of degree $2K-1$ at most, hence, it will be exactly computed by the Gauss--Lobatto quadrature formula with $K$ nodes which defines also the polynomials of degree $K$.
	For the matrix $\tilde{D}_x$, if we focus on the degrees of freedom $\alpha=(i,0)$ for $i=1,\dots,N_x-1$, and using the Lobatto quadrature formula with weights $\lbrace w_r\rbrace_{r=0}^K$, we obtain that \eqref{eq:kernel_equation} becomes: find $q_h \in V^{K-1}_{\dx,b}(\Omega^x_\dx)$ such that
	\begin{equation}\label{eq:eq_at_interfaces}
		0=\int_{E_{x,i-1}} \varphi_{i-1,K}(x) q_h(x) \dd x +\int_{E_{x,i}} \varphi_{i,0}(x) q_h(x) \dd x = 
		\Delta x \, w_0 (q_{i-1,K-1}+q_{i,0}).
	\end{equation}
	Here, clearly we have that $q_{i,0}=-q_{i-1,K-1}$ for $i=1,\dots,N_x-1$. Hence, the constant $1$ cannot be in the kernel of $\tilde{D}_x$.
	
	Now, to show the matrix is full ranked, we study its structure in \eqref{eq:matrix_kernel}, recalling that $\tilde{D}_x \in \mathbb R^{(N_xK-1)\times (N_x K)}$. 
\begin{align}\label{eq:matrix_kernel}
	\footnotesize
	\newcommand*{\AddLeftone}[1]{%
		\vadjust{%
			\vbox to 0pt{%
				\vss
				\llap{$%
					{#1}\left\{
					\vphantom{
						\begin{matrix}1\end{matrix}
					}
					\right.\kern-\nulldelimiterspace
					\kern0.5em
					$}%
				\kern0pt
			}%
		}%
	}
	\newcommand*{\AddLefttwo}[1]{%
	\vadjust{%
		\vbox to 0pt{%
			\vss
			\llap{$%
				{#1}\left\{
				\vphantom{
					\begin{matrix}1\\1\end{matrix}
				}
				\right.\kern-\nulldelimiterspace
				\kern0.5em
				$}%
			\kern0pt
		}%
	}%
	}
	\newcommand*{\AddLeftthree}[1]{%
	\vadjust{%
		\vbox to 0pt{%
			\vss
			\llap{$%
				{#1}\left\{
				\vphantom{
					\begin{matrix}1\\1\\1\end{matrix}
				}
				\right.\kern-\nulldelimiterspace
				\kern0.5em
				$}%
			\kern0pt
		}%
	}%
}
\left[ \begin{array}{ccc|ccc|ccc}
		&\multirow{2}{*}{$Z$}&& 0       & 0      & 0      & 0     & 0       & 0      \\
		\AddLefttwo{K-1}
		&&                 & 0       & 0      & 0      & 0     & 0  & 0        \\ \hline		
		*      & *      & *      &    \multicolumn{3}{c|}{\multirow{3}{*}{$B$}}   & 0      & 0     & 0  \\		
		0      & 0      & 0      & \multicolumn{3}{c|}{} & 0       & 0        & 0\\	\AddLeftthree{K}
		0      & 0      & 0      & \multicolumn{3}{c|}{} & 0       & 0        & 0\\ \hline
		0       & 0      & 0      & *     &*      & *      & \multicolumn{3}{c}{\multirow{3}{*}{$B$}} \\
		0      & 0      & 0      & 0      & 0       & 0      & \multicolumn{3}{c}{}                 \\
		\AddLeftthree{K}
		0      & 0      & 0      & 0   & 0       & 0           & \multicolumn{3}{c}{}
	\end{array}\right]
\end{align}
We observe that it contains matrices $B\in\mathbb R^{(K-1)\times (K-1)}$ proportional to the $B$ matrix of Lemma~\ref{lem:invertibility_mixed} and a $Z\in \mathbb R^{K\times(K-1)}$ matrix that is given by $B$ without the first row. Since $B$ is invertible, there exists a minor of $Z$ of dimension $(K-1)\times (K-1)$ that is invertible as well. If we remove the corresponding column also from the matrix $\tilde{D}_x$, the resulting is also invertible. 
This can be seen by computing the determinant and looking at the minors that contribute to that computation, i.e., only the $B$ matrices and of the invertible minor of $Z$. 
This means that $\tilde{D}_x$ is of full rank $N_xK-1$ and, hence, possesses a kernel of dimension one generated by a nonconstant vector.
\end{proof}

\begin{lemma}[Invertibility of mixed derivative matrix]\label{lem:invertibility_mixed_der}
	Consider $\lbrace \hat{\varphi}_i \rbrace_{i=0}^K$ the set of Lagrangian polynomials generated by $K+1$ Gauss Lobatto points in $\mathbb P^K([0,1])$ and  $\lbrace \hat{\psi}_i \rbrace_{i=0}^{K-1}$ the set of Lagrangian polynomials generated by $K$ Gauss Lobatto points in $\mathbb P^{K-1}([0,1])$. Consider the matrix $C_{ij}:=\int_0^1 \partial_x \hat{\varphi}_i \hat{\psi}_j \dd x$ for $i=0,\dots,K$ and $j=0,\dots,K-1$.
	Now, consider the square matrices obtained as the restriction of $C$ that we denote by  $D_{ij}=C_{ij}$ for $i=1,\dots,K$, $j=0,\dots,K-1$. Then, $D$ is invertible.
\end{lemma}
\begin{proof}
	Instead of proving the invertibility of $D$, we will show the invertibility of its transpose, which is 
	\begin{equation}\label{eq:B_quad2}
		(D^T)_{ij}=\int_0^1 \hat{\psi}_i(x)  \partial_x \hat{\varphi}_j(x) \dd x .
	\end{equation}
	Now, the system is invertible if there exists one and only solution to the system of equations for the unknown $\lbrace q_j \rbrace_{j=1}^{K}$ and the right hand side  $\lbrace r_i \rbrace_{i=0}^{K-1}$
	\begin{equation}
		\int_0^1 \hat{\psi}_i(x)  \partial_x \hat{\varphi}_j(x) q_j \dd x = r_i \Longleftrightarrow 
		\int_0^1 \hat{\psi}_i(x)  q_h(x) \dd x = r_i,
	\end{equation}
	with $q_h = \sum_{j=1}^{K} \partial_x \hat\varphi_j(x) \in \mathbb P^{K-1}([0,1])$ for every $\hat\psi_i \in \mathbb P^{K-1}$. Classically, $q_h\in \mathbb P^{K-1}([0,1])$ could be rewritten in an expansion of the $\hat\psi_j$ basis functions, obtaining on the left hand side the classical symmetric and positive definite mass matrix for $\mathbb P^{K-1}([0,1])$. This is invertible, hence $D$ is invertible.
\end{proof}

\begin{proposition}[Kernel characterization of $\tilde{D}_x^x$]
	$\tilde{D}^x_x: \mathbb R^{N_x\times K }\to \mathbb R^{N_x \times K -1}$ has kernel of dimension one and it is generated by the constant function $1$.
\end{proposition}
\begin{proof}
	The matrix is defined for $\alpha=1,\dots,N_xK-1$ and $\beta =1,\dots,N_xK$ (with the notation of Theorem~\ref{th:kernel_char}) by
	\begin{equation}
				(\tilde{D}^x_x)_{\alpha;\beta} := \int_{\Omega^x_\dx} \partial_x \varphi_{\alpha} (x)\psi_{\beta}(x)\dd x.
	\end{equation}
Clearly $1$ belongs to the kernel of $\tilde{D}_x^x$, as 
\begin{equation}
	\int_{\Omega_\dx^x} \partial_x \varphi_\alpha (x) 1 \dd x = \left[\varphi_\alpha (x) \partial_x 1  \right]_{\partial \Omega^x_\dx}-\int_{\Omega_\dx^x}  \varphi_\alpha (x) \partial_x 1 \dd x =0
\end{equation} 
because $\varphi_\alpha \in V^{K}_{\dx,0}(\Omega_\dx^x)$ and $\partial_x 1 \equiv 0$. This corresponds to the space of all affine functions for the kernel of $D_x^x$.

\begin{align}\label{eq:matrix_kernel_Dxx}
	\footnotesize
	\newcommand*{\AddLeftone}[1]{%
		\vadjust{%
			\vbox to 0pt{%
				\vss
				\llap{$%
					{#1}\left\{
					\vphantom{
						\begin{matrix}1\end{matrix}
					}
					\right.\kern-\nulldelimiterspace
					\kern0.5em
					$}%
				\kern0pt
			}%
		}%
	}
	\newcommand*{\AddLefttwo}[1]{%
		\vadjust{%
			\vbox to 0pt{%
				\vss
				\llap{$%
					{#1}\left\{
					\vphantom{
						\begin{matrix}1\\1\end{matrix}
					}
					\right.\kern-\nulldelimiterspace
					\kern0.5em
					$}%
				\kern0pt
			}%
		}%
	}
	\newcommand*{\AddLeftthree}[1]{%
		\vadjust{%
			\vbox to 0pt{%
				\vss
				\llap{$%
					{#1}\left\{
					\vphantom{
						\begin{matrix}1\\1\\1\end{matrix}
					}
					\right.\kern-\nulldelimiterspace
					\kern0.5em
					$}%
				\kern0pt
			}%
		}%
	}
	\qquad\left[ \begin{array}{ccc|ccc|ccc}
		&& & 0       & 0      & 0      & 0     & 0       & 0               \\
		&D&                  & 0       & 0      & 0      & 0     & 0       & 0                \\
		\AddLeftthree{K}
	    &&                 & *       & *      & *      & 0     & 0       & 0            \\ \hline		
		0      & 0      & 0         & \multicolumn{3}{c|}{\multirow{3}{*}{$D$}} & 0       & 0      & 0     \\
		0      & 0      & 0       & \multicolumn{3}{c|}{}                  & 0      & 0      & 0     \\
		\AddLeftthree{K}
		0      & 0      & 0          & \multicolumn{3}{c|}{}          & *&*&*     \\   \hline
		0      & 0      & 0      & 0     & 0         & 0   & \multicolumn{3}{c}{\multirow{2}{*}{$E$}}     \\
		\AddLefttwo{K-1}
		0      & 0      & 0      & 0      & 0        & 0       & \multicolumn{3}{c}{}    
	\end{array}\right]
\end{align}
Now, looking at the structure of $\tilde{D}_x^x \in \mathbb R^{(N_x K-1)\times (N_x K)}$, we observe that the matrix $D \in \mathbb R^{K\times K}$ is proportional to the one of Lemma~\ref{lem:invertibility_mixed_der}. 
Matrix $E$ is the matrix $D$ without its last row. Since $D$ is invertible, there exists a minor of $E$ of dimension $(K-1)\times(K-1)$ that is invertible. Hence, if we consider the whole matrix without the column which is excluded by the invertible minor of $E$, then, 
we clearly see that it is invertible (the determinant is not 0 studying the relevant minors). Hence, the kernel of $D_x^x$ is of dimension 1 and it is generated by the constant function 1.
\end{proof}

\begin{proposition}[Kernel characterization of $Z_x=D^x_x - D^x M^{-1}D_x$]
	Let $Z_x=D^x_x - D^x M^{-1}D_x:\mathbb R^{N_x \times K +1} \to \mathbb R^{N_x \times K -1}$, with $M^{-1},D_x: \mathbb R^{N_x \times K +1} \to\mathbb R^{N_x \times K +1}$ be the matrices defined with test and trial functions in $V^K_{\Delta x}(\Omega_{\Delta x}^x)$ and $D^x_x, D^x:\mathbb R^{N_x \times K +1} \to \mathbb R^{N_x \times K -1}$ be defined with trial functions in $V^K_{\Delta x}(\Omega_{\Delta x}^x)$ and test in $V^K_{\Delta x,0}(\Omega_{\Delta x}^x)$. Then, the kernel of $Z$ contains $\langle 1, x \rangle$ and does not contain $w$ the vector of the kernel of $D_x$.
\end{proposition}
\begin{proof}
	Clearly, $1$ and $x$ belong to the kernel of $_xZ$ as they belong to the kernel of $D^x_x$ and since $D_x 1\equiv 0$ and $$D^x M^{-1 }D_x x \equiv D^x M^{-1 } 1 \equiv D^x 1 \equiv 0.$$
	
	Let us now restrict our operator on $V_0^K$, by taking the affine operation $\tilde{u} = u-u(0)+(u(0)-u(1))x$ that uses only elements in the kernel of $Z_x$ to bring $u\in V^K_{\Delta x}(\Omega_{\Delta x}^x)$ into $\tilde{u} \in V^K_{\Delta x,0}(\Omega_{\Delta x}^x)$. We aim showing that $Z_x$ on $V^K_{\Delta x}(\Omega_{\Delta x}^x)$ is symmetric non-negative definite. The symmetry is trivially shown. We focus on the non-negativeness.
	
	Before proceeding, let us better describe $Z_x$ on $V^K_{\Delta x}(\Omega_{\Delta x}^x)$. 
	We have that 
	\begin{align}
		 (D_x)_{\alpha;\beta} u^{\beta} = \begin{cases}
			\Delta x w_\alpha\left(\partial_x u (x_{\alpha}^+) + \partial_x u (x_{\alpha}^-)\right) & \text{if }\alpha = (i,0) = (i-1,K),\\			
			\Delta x w_\alpha \partial_x u (x_{\alpha})& \text{else,}
		\end{cases}
\end{align}
with $\partial_x u \in V^{K-1}_{\dx,b}(\Omega_{\dx}^x)$, $x_{\alpha}$ being the Gauss--Lobatto composite quadrature points defined by $K+1$ points in each cell and $w_\alpha$ the quadrature weight referred to the $k$-th degree of freedom in the reference element $[0,1]$ with $w_{(i,0)}=w_{(i-1,K)}$.
Then, if we study the bilinear form, we have that
\begin{align}
	u^T(D^x M^{-1}D_x) u =&
	\sum_{\alpha,\beta} \int \partial_x u\, \varphi_\alpha\, dx\, \frac{\delta_{\alpha,\beta}}{M_{\alpha,\alpha}}  \int \varphi_{\beta} \,\partial_x u\, dx =  \sum_\alpha \frac{1}{M_{\alpha,\alpha}} \left(\int \varphi_\alpha \partial_x u\right)^2=\\
	&\sum_{\alpha \in \mathcal{I}} \frac{\Delta x ^2 w_\alpha^2}{\Delta x w_\alpha}  (\partial_x u(x_\alpha))^2 + \sum_{\alpha \in \mathcal{E}} \frac{\Delta x ^2 w_\alpha^2}{2\Delta x w_\alpha}  (\partial_x u(x_\alpha^-)+\partial_x u(x_\alpha^+))^2=\\
	&\sum_{\alpha \in \mathcal{I}} {\Delta x w_\alpha} (\partial_x u(x_\alpha))^2 + \sum_{\alpha \in \mathcal{E}} \frac{\Delta x w_\alpha}{2}  (\partial_x u(x_\alpha^-)+\partial_x u(x_\alpha^+))^2,
\end{align}
where we have introduced the set of internal degrees of freedom $\mathcal {I}=\lbrace \alpha : \varphi_\alpha \in V^{K}_{\dx,0}(\Omega_{\dx}^x), \, \alpha = (i,k) \text { with }k \in [1,K-1]\rbrace$ and edges degrees of freedom  $\mathcal {E}=\lbrace \alpha : \varphi_\alpha \in  V^{K}_{\dx,0}(\Omega_{\dx}^x), \, \alpha = (i,0) \rbrace.$

Now, using this definition, we will show that the restriction of $Z_x$ to $ V^{K}_{\dx,0}(\Omega_{\dx}^x)$ is symmetric non-negative definite. Take $u\in  V^{K}_{\dx,0}(\Omega_{\dx}^x) \equiv\mathbb R^{N_x K -1} \subset  V^{K}_{\dx}(\Omega_{\dx}^x) \equiv \mathbb R^{N_xK+1}$, using the previous computations and the definition of $D_x^x$, we compute
\begin{align}
	u^T Z_x u=& u^T D^x_x u - u^T D^x M^{-1} D_x u = \int (\partial_x u)^2 -  u^T D^x M^{-1} D_x u =\\
	\begin{split}
		&\sum_{\alpha \in \mathcal{I}} {\Delta x w_\alpha} (\partial_x u(x_\alpha))^2 +  \sum_{\alpha \in \mathcal{E}} \Delta x w_\alpha  (\partial_x u(x_\alpha^-)^2+\partial_x u(x_\alpha^+)^2)\\
		-&\sum_{\alpha \in \mathcal{I}} {\Delta x w_\alpha} (\partial_x u(x_\alpha))^2 -\sum_{\alpha \in \mathcal{E}} \frac{\Delta x w_\alpha}{2}  (\partial_x u(x_\alpha^-)+\partial_x 	u(x_\alpha^+))^2=
	\end{split}\\
	&\sum_{\alpha \in \mathcal{E}} \frac{\Delta x w_\alpha}{2}  (\partial_x u(x_\alpha^-)-\partial_x u(x_\alpha^+))^2\geq 0.\label{eq:continuous_der_inte}
\end{align} 
We have just shown that $Z_x$ is non-negative definite. 
So, the element in the kernel of $Z_x$, i.e., $Z_xu=0$, must also be such that $u^TZ_xu=0$ and hence, they must have continuous derivative at the interfaces \eqref{eq:continuous_der_inte}. This was not the case for $w$ the element generating with $1$ the kernel of $D_x$.
\end{proof}
Unfortunately, we cannot say more about the matrix $Z_x$ and, experimentally, we have noticed that the kernel is indeed much larger than just these vectors. In particular, the dimension of the kernel increases with the order of the method and the number of cells.

\bibliographystyle{unsrt}
\bibliography{biblio}

%
%

\end{document}